\documentclass[reqno]{amsart}
\usepackage[utf8]{inputenc}
\usepackage{amssymb}
\usepackage{mathtools}
\usepackage{color}
\usepackage{enumitem}
\usepackage{caption}
\usepackage{tikz}
\usepackage{xcolor}
\usepackage{upgreek}
\usepackage[bookmarksdepth=3]{hyperref} \usepackage{bbm}
\usepackage{mathrsfs}
\usepackage{leftindex} 
\usepackage{booktabs} 
\mathtoolsset{showonlyrefs}

\newtheorem{Satz}{Satz}[section]
\newtheorem{definition}[Satz]{Definition}
\newtheorem{theorem}[Satz]{Theorem}
\newtheorem{lemma}[Satz]{Lemma}
\newtheorem{proposition}[Satz]{Proposition}
\newtheorem{corollary}[Satz]{Corollary}
\newtheorem{remark}[Satz]{Remark}

\newtheorem{assumption}[Satz]{Assumption}

\newcommand{\N}{\mathbb{N}}
\newcommand{\R}{\mathbb{R}}

\newcommand{\E}{\mathbb{E}}
\newcommand{\F}{\mathscr{F}}
\renewcommand{\P}{\mathbb{P}}
\newcommand{\sP}{\mathscr{P}}
\newcommand{\sB}{\mathscr{B}}
\newcommand{\sG}{\mathscr{G}}
\newcommand{\sL}{\mathscr{L}}

\newcommand{\sO}{\mathscr{O}}

\newcommand{\ce}{\coloneqq}
\newcommand{\ec}{\eqqcolon}
\newcommand{\seq}{\subseteq}
\newcommand{\hra}{\hookrightarrow}
\newcommand{\1}{\mathbf{1}}

\newcommand{\grad}{\nabla}
\newcommand{\rmd}{\mathrm{d}}
\newcommand{\dt}{\,\mathrm{d}t}
\newcommand{\ds}{\,\mathrm{d}s}

\newcommand{\dx}{\,\mathrm{d}x}
\newcommand{\dWs}{\,\mathrm{d}W(s)}
 
\newcommand{\dWsn}{\,\mathrm{d}W^n(s)}

\newcommand{\intt}{\int_0^t}

\newcommand{\inttm}{\int_0^{t\wedge\sigma_m}}
\newcommand{\intsO}{\int_{\sO}}

\newcommand{\CBDG}{C_{\mathrm{BDG}}}
\newcommand{\loc}{\textup{loc}}

\newcommand{\nH}{\ensuremath{\,_\nu H}}

\newcommand{\bpsn}[1]{\ensuremath{\nH^{#1}(\sO)}}
\newcommand{\Bpsn}[1]{\ensuremath{\nH^{#1}(\sO;\ell^2)}}

\newcommand{\dd}{\ensuremath{\mathrm{d}}}

\newcommand{\yosi}{\lambda R_\lambda}

\newcommand{\nrm}[1]{\left\lVert #1 \right\rVert}

\usepackage{bib}       
\usepackage{frontmatter}
\usepackage[margin=2.5cm]{geometry} 

\title{The stochastic Cahn--Hilliard equation in critical spaces}

\author{Simon Bau}
\address[Simon Bau]{University of Konstanz, Konstanz, Germany}
\email{simon.bau@uni-konstanz.de}

\author{Gideon Chiusole}
\address[Gideon Chiusole \orcidlink{0009-0005-3930-2924}]{Technical University of Munich, Munich, Germany}
\email{g.chiusole@tum.de}

\author{Sarah Geiss}
\address[Sarah Geiss]{Technische Universität Berlin, Berlin, Germany}
\email{geiss@math.tu-berlin.de}

\author{Katharina Klioba}
\address[Katharina Klioba \orcidlink{0009-0002-7946-917X}]{Delft University of Technology, Delft, the Netherlands}
\email{k.klioba-1@tudelft.nl}

\author{Tobias Werner}
\address[Tobias Werner]{University of Kassel, Kassel, Germany}
\email{twerner@mathematik.uni-kassel.de}

\thanks{The second author gratefully acknowledges support by the Studienstiftung des deutschen Volkes through a doctoral scholarship and the Marianne-Plehn-Program, and by the TopMath doctoral program. 
The fourth author gratefully acknowledges support by the Alexander von Humboldt foundation through a Feodor Lynen Research Fellowship.}

\date{August 21, 2026}

\begin{document}

\begin{abstract}
    We study stochastic Cahn--Hilliard equations in bounded smooth domains with a double-well potential, transport-type noise, and natural Neumann boundary conditions in dimensions $d\le 4$. By employing stochastic maximal regularity techniques and deriving suitable energy estimates, we prove local and global well-posedness. The initial data considered here are allowed to belong to the critical trace space $B^{d/q-1}_{q,p}$, which is locally invariant under the natural scaling of the Cahn--Hilliard equation. In particular, for arbitrary $\varepsilon>1/3$ one can find $q$ sufficiently large such that uniqueness and global existence of a probabilistically strong solution hold for every initial datum $u_0\in H^{\varepsilon, q}(\sO)$. If $u_0\in H^{1}(\sO)$, then these global solutions have $L_t^2H_x^3 \cap C_tH_x^1$-regularity on finite time intervals.
\end{abstract}



\maketitle

\tableofcontents

\section{Introduction}

This paper addresses the local and global well-posedness of stochastic Cahn--Hilliard equations of the form
\begin{equation}\tag{SCH}\label{eq:intro:SPDE}
    \left\{
    \begin{aligned}
        \dd u+ \Delta^2u \dt&= \Delta(u^3-u)\dt+ \sum_{n\ge 1} \psi_n(u,\nabla u)\,\dd W^n(t),&\text{in }& (0,\infty)\times\sO,\\
        \partial_\nu u&=\partial_\nu\Delta u=0, &\text{on }&(0,\infty)\times \partial \sO,\\
        u(0)&=u_0 &\text{in }&\sO,
    \end{aligned}
    \right.
\end{equation}
where $\sO\seq \R^d$, $d\in\{1,2,3,4\}$, is a bounded $C^\infty$ domain.
The noise terms under consideration admit the following transport-type structure, with an additional non-transport contribution:
\begin{equation}
    \psi_n(u,\nabla u)= \psi_n^0(u) + \psi_n^1\cdot\nabla u.
\end{equation}
Here, $(\psi^0_n)_{n\geq 1}$ satisfies suitable linear growth and Lipschitz assumptions, while $(\psi_n^1)_{n\ge 1}\in \ell^2(\R^d)$.
Our results allow initial data in the Besov space $B_{q,p}^{d/q-1}(\sO)$, which is the \emph{critical trace space} associated with the natural local scaling of the equation. Critical spaces are well-studied in the deterministic literature for a broad class of PDEs including the Navier--Stokes equations, the Allen--Cahn equation, and the Cahn--Hilliard equation; see~\cite{prussCriticalSpacesQuasilinear2018}. Our methods are based on stochastic maximal regularity techniques, in particular the framework of SPDEs in critical spaces developed in~\cite{agrestiNonlinearParabolicStochastic2022,agrestiNonlinearParabolicStochastic2022a, agrestiCriticalVariationalSetting2024a}, and energy methods adapted to fourth-order equations.

\subsection{Deterministic setting}
The Cahn--Hilliard equation, originally introduced by Cahn and Hilliard in the foundational papers~\cite{CahnHilliard1958, CahnHilliard1959, Cahn1961}, is a fundamental model for describing phase separation phenomena in binary mixtures.

Suppose we have $T \in (0,\infty)$, $m>0$ and consider a bounded domain $\sO \seq \mathbb{R}^d$ with sufficiently smooth boundary $\partial\sO$. The classical Cahn--Hilliard equation is given by
\begin{equation}
\left\{
\begin{aligned}
\partial_t u &= m \Delta  \mu && \text{in } \sO \times (0,T),\\
\mu &= - \Delta u+ \Phi'(u) && \text{in } \sO \times (0,T),
\end{aligned}
\right.
\label{CH}
\end{equation}
where the unknown function $u$ denotes the phase-field or order parameter with respect to the physical quantity of interest, e.g.~ volume or mass. 
Moreover, $\mu$ is the so-called chemical potential, and the function $\Phi$ is the free-energy density, commonly referred to as the bulk potential.
To be more precise, the phase field function $u$ describes the difference $u_1-u_2$ of the two local concentrations $u_1$ and $u_2$. After rescaling, one typically assumes $u_i\in [0,1]$ for $i=1,2$, $u_1+u_2=1$, and thus $u\in [-1,1]$. Then the regions $\{x\in \sO\colon u(x) = 1\}$  and $\{x\in \sO \colon u(x) = -1\}$ correspond to the pure phases of the two components in the underlying mixture and $\{x\in \sO \colon -1<u (x) < 1\}$ is the diffuse interface between them. The bulk potential $\Phi$ is generally of double-well form and may depend on the physical context. A typical choice is the
smooth double-well potential 
\begin{align*}
    \Phi(s)\coloneq \frac{1}{4}(1-s^2)^2,\qquad s\in\R,
\end{align*}
which is a polynomial approximation of the physically motivated singular logarithmic potential appearing in the original Cahn--Hilliard model~\cite{CahnHilliard1958}. The mobility $m$ is usually chosen as a positive constant, but can also be a concentration dependent non-negative function $m=m(u)$.
A treatment of \eqref{CH} with a degenerate mobility as well as a singular potential can be found in~\cite{Garcke1996}.
In case of a non-degenerate mobility coefficient, the Cahn--Hilliard equation is a fourth-order parabolic equation and thus has to be equipped with suitable boundary and initial conditions. A typical choice is given by the homogeneous Neumann conditions
\begin{alignat}{4}
    \partial_\nu \mu &= 0 \quad & \text{on } \partial\sO\times(0,T) \qquad \text{ and } \qquad
    \partial_\nu u&= 0 \quad & \text{on } \partial\sO\times(0,T),
\end{alignat}
along with the initial condition $u(0)=u_0$.
Here, $\nu$ denotes the outward unit normal vector to $\partial\sO$.
The first boundary condition is a no-flux boundary condition and is related to mass conservation of the system. Furthermore, adding the second boundary condition ensures that the Cahn--Hilliard free energy functional 
is non-increasing in time; cf.\ the review article~\cite{Wu2022}, which also discusses other possible boundary conditions. For a book solely devoted to the Cahn--Hilliard equation in the deterministic setting, we refer to~\cite{Miranville2019}.

\subsection{Stochastic setting}
\label{subsec:introStochasticLiterature}
Stochastic models of the Cahn--Hilliard equation, also known as the Cahn--Hilliard--Cook equation, were first discussed in~\cite{Cook1970}. From a physical perspective, the stochastic perturbation may account for thermal fluctuations arising during the phase separation process. The mathematical analysis of stochastic versions of the Cahn--Hilliard equation dates back to Elezovi\'c and Mikeli\'c~\cite{Elezovic91} for additive trace-class noise as well as Da Prato and Debussche~\cite{DaPratoDebusche1996} for additive space-time white noise. 
Subsequently, further regularity properties of solutions were investigated in~\cite{Cardon-Weber2001,Cardon-Weber2002}, while aspects of the qualitative behaviour of stochastic Cahn--Hilliard dynamics were studied in~\cite{Bloemker2001,Bloemker2008}.
Stochastic Cahn--Hilliard equations with reflection and singular logarithmic potentials were analysed in~\cite{Debusche2007,Goudenege2009,Debusche2011} by means of so-called reflection measures. For space-time white noise, existence of mild solutions for multiplicative bounded Lipschitz noise has been established in~\cite{AntonopoulouKarali}, which was later generalised to unbounded Lipschitz noise of sublinear growth in~\cite{AnatonopoulouKaraliMillet16}. The case of Lévy noise was studied in~\cite{Bo2006}. 
Global well-posedness of probabilistically strong solutions has been established by Scarpa~\cite{Scarpa2018} on smooth domains in $\R^d$, $d\in \{2,3\}$, with homogeneous Neumann boundary conditions and potentially singular potentials. Existence of martingale solutions was obtained for further generalisations to the case of degenerate mobility in~\cite{Scarpa2021b}. More recently, global well-posedness for transport noise and constant mobility as well as higher moment energy estimates for Fourier noise both in dimension $3$ were studied in~\cite{diPrimioPapiniScarpa26} and~\cite{Antonopoulou23}, respectively.
Global well-posedness results have also proved useful in studying nonlocal-to-local limits of stochastic Cahn--Hilliard equations~\cite{diprimioStochasticNonlocalCahn2026}.\\
From the perspective of stochastic maximal $L^p$-regularity as developed in the papers~\cite{agrestiNonlinearParabolicStochastic2022,agrestiNonlinearParabolicStochastic2022a} by Agresti and Veraar, it is more natural to formulate the stochastic version of the system~\eqref{CH} as a fourth-order parabolic SPDE
\begin{equation}
\label{eq:stochastic_CH}
    \rmd u = [-\Delta^2u+\Delta(u^3-u)
    ]\dt + \sum_{n \geq 1} \psi_n(u,\grad u)\,\rmd W^n,\quad u(0)=u_0.
\end{equation}
Here, the series on the right-hand side represents the noise term, which we will specify later. The linear dynamics are driven by the Bi-Laplacian endowed with the (formally equivalent) boundary conditions
\begin{alignat}{2}
    \partial_\nu \Delta u &= 0, \quad & \text{on } \partial\sO\times(0,T), \label{bc_mu_new} \\
    \partial_\nu u&= 0, \quad & \text{on } \partial\sO\times(0,T).\label{bc_phi_new}
\end{alignat}
    Local well-posedness results for \eqref{eq:stochastic_CH} by means of stochastic $L^p(L^q)$-theory were first established in~\cite[Section~7.3]{agrestiNonlinearParabolicStochastic2022}. The authors study the so-called almost very weak setting, which 
can be seen as a general weak solution concept in the spirit of the extrapolation techniques illustrated in~\cite[Chapter~V]{amannLinearQuasilinearParabolic1995}.
A treatment of \eqref{eq:stochastic_CH} in the critical variational setting is given in~\cite{agrestiCriticalVariationalSetting2024a}, where global well-posedness for $d=1,2$ is obtained in suitable Hilbert spaces.
The precise formulations of the obtained results are recalled in Subsection~\ref{weak_setting} for Hilbert spaces and in Section~\ref{section_rough_initial_data} for general Banach spaces.

\subsection{Scaling and criticality}
Before we state our main results, we provide intuition for the local scaling of the Cahn--Hilliard equation on $\R^d$, $d\ge 1$, motivating the notion of critical spaces. 

Let $u$ be a solution to the deterministic equation $\partial_t u +\Delta^2u = \Delta u^3$ with some initial value $u_0$, where we have omitted the lower-order term $-\Delta u$ because it is not relevant for the critical scaling.
For $\lambda>0$, write $J_\lambda u(t,x)\ce  u(\lambda t, \lambda^{1/4}x)$ as well as $J_\lambda u_{0}(x)\ce u_0(\lambda^{1/4}x)$, and  for some $\alpha\in\R$ let $u_\lambda\ce \lambda^{\alpha} J_\lambda u, $ $u_{0,\lambda}\ce \lambda^{\alpha}J_\lambda u_0$. 
The chain rule implies that 
\[\begin{aligned}
    (\partial_t +\Delta^2) u_\lambda &= \lambda^{\alpha+1} J_\lambda(\partial_t+\Delta^2)u\quad \text{and}\quad 
    \Delta(u_\lambda^3)= \lambda^{3\alpha+\frac{1}{2}}J_\lambda(6  u |\nabla u|^2 +3 u^2\Delta u).
\end{aligned}\]
Hence, both sides of the equation have the same scaling if we set $\alpha=1/4$. Moreover, 
\[\|u_\lambda(t)\|_{L^q(\R^d)}= \lambda^{\frac14-\frac{d}{4q}}\nrm{u(\lambda t)}_{L^q(\R^d)}\quad\text{and}\quad\nrm{u_\lambda(t)}_{\Dot{B}^{\sigma}_{q,p}(\R^d)}\eqsim\lambda^{\frac14+\frac{1}{4}(\sigma-\frac{d}{q})}\nrm{u(\lambda t)}_{\Dot{B}_{q,p}^\sigma(\R^d)}.\]
Note the scaling invariance of these Lebesgue and homogeneous Besov norms for $q=d$ and $\sigma = d/q-1$, respectively. With this choice, the \emph{Sobolev index} of both spaces is given by $-1$, independently of $p$, $q$, or $d$. The inhomogeneous space $B^{d/q-1}_{q,p}(\R^d)$ is often called \emph{critical} in the PDE literature, cf., e.g.~\cite{prussCriticalSpacesQuasilinear2018}.

\subsection{Main results}

The following theorem summarizes our main results concerning global well-posedness and a priori estimates for probabilistically strong solutions of the stochastic Cahn--Hilliard equation. More detailed statements under weaker assumptions, in particular allowing spatially and temporally dependent noise, are given in Theorem~\ref{thm:GWP:weak:setting} for local well-posedness, Theorems~\ref{thm:global_well_posedness_dim_3} and \ref{thm:global_well_posedness_dim_4-complete} for global well-posedness in dimensions three and four, respectively. Theorem~\ref{gwp:rough:initial:data} contains the transfer of the blow-up criterion and Corollary~\ref{cor:GWP:rough} combines these results to obtain global well-posedness for rough initial data. For a local well-posedness result in arbitrary dimensions $d\in\N$ under further weakened assumptions on the noise, see Theorem~\ref{lwp:rough:initial:data}.

Spaces with left subscript $\nu$ indicate suitable Neumann boundary conditions depending on the regularity of the space. Besov, Bessel potential, and anisotropic Hölder spaces are denoted by $B_{q,p}^s$, $H^{s,p}$ or $H^s=H^{s,2}$, and $C^{\alpha,\beta}$, respectively. We work with weighted Lebesgue spaces with the power-law weight $w_\kappa(t)\ce t^\kappa$ for $t>0$, which coincide with the unweighted Lebesgue spaces for $\kappa=0$. For the definitions of the spaces $C^{\alpha,\beta}_\loc$, $H^{\theta,p}_\loc([0,\infty),w_\kappa;X)$ and $\bpsn{s,q} $ and $_\nu B^{s}_{q,p}(\sO)$, we refer to Subsections~\ref{subsec:notation} and~\ref{subsec:function_spaces}, respectively.  

\begin{theorem}\label{thm:introMain}
    Let $d\in \{3,4\}$, and let $\sO\seq\R^d$ be a bounded smooth domain. 
Suppose that 
  \begin{equation*}
      \psi=(\psi_n)_{n\geq 1}:\R^{1+d}\to \ell^2,\quad \psi(y,z) \ce \psi^0(y)+\psi^1\cdot z,
  \end{equation*}
  with $\psi^0=(\psi^0_n)_{n\geq 1}\in C^1(\R;\ell^2)$ and $\psi^1\in \ell^2(\R^d)$.
 Moreover,  assume that 
 there exists a constant $L\geq0$ such that for all $y,y'\in\R$, 
    \begin{align*}
                \|\partial_y
            \psi^0(y)-\partial_y\psi^0(y')
        \|_{\ell^2}
        \leq L|y-y'| \quad\text{ and }\quad
        \|\psi^0(y)\|_{\ell^2}
        \leq L(1+|y|).
    \end{align*}
    Suppose that one of the following two cases holds.
    \begin{enumerate}[label=(\alph*)]
        \item \label{case:pq2} Let $p=q=2$ and $s=\kappa=0$.
        \item \label{case:pqnot2} Let $p\in (2,\infty)$, $ \kappa\in [0,p/2-1)$, $s\in [0,1)$, and  $q\in [2, \frac{2d}{2+s})$ satisfy
    \begin{equation}\label{eq:thm:introMain}
        \frac{1+\kappa}{p}\le 1-\frac{1}{4}\Big(s+\frac{d}{q}\Big).
    \end{equation}
    \end{enumerate}
     Then for all $u_0\in L^0_{\mathscr{F}_0}(\Omega;\,_\nu B^{3-s-4(1+\kappa)/p}_{q,p}(\sO))$ there exists a global solution $u$ to~\eqref{eq:intro:SPDE}, which is unique in the class $L^p_\loc([0,\infty),w_\kappa;\bpsn{3-s,q})$. In particular, in case \ref{case:pq2}, this holds for every $u_0\in L_{\F_0}^0(\Omega;H^1(\sO))$. In case \ref{case:pqnot2}, for any $\varepsilon>1/3$, one can choose suitable $p$ and $q$ such that this holds for all $u_0\in L_{\F_0}^0(\Omega; \nH^{\varepsilon, q}(\sO))$. Moreover, it holds that a.s.
    \begin{align*}
				u&\in H^{\theta,r}_{\loc}((0,\infty); \nH^{3-4\theta,\zeta}(\sO)),  \qquad \zeta\in [q,\infty),\, r\in (2,\infty),\, \theta\in [0,1/2),\\
                u&\in C_{\loc}^{\theta,6\theta}((0,\infty)\times \sO), \qquad\qquad\qquad\theta\in (0,1/2),
			\end{align*}
   and for each $0<t<T<\infty$,  whenever $u(t)\in L^4(\Omega;L^4(\sO))\cap L^2(\Omega;H^1(\sO))$, there is a constant $C_T\ge 0$ independent of $u(t)$ such that 
    \begin{equation}
    \label{eq:aprioriIntro}
        \E\nrm{u}^2_{C([t,T];H^1(\sO))}\le  C_{T} \big(1+\E\nrm{u(t)}_{L^4(\sO)}^4+\E\nrm{u(t)}_{H^1(\sO)}^2\big)<\infty.
    \end{equation}
    In case \ref{case:pq2}, it holds that $u\in L_\loc^2([0,\infty);\nH^3(\sO))\cap C_\loc([0,\infty);H^1(\sO))$ a.s.\ and \eqref{eq:aprioriIntro} also holds for $t=0$ provided that $u_0\in L^4(\Omega;L^4(\sO))\cap L^2(\Omega;H^1(\sO))$.
\end{theorem}
One of the main assertions of Theorem~\ref{thm:introMain} is the existence of a unique global solution to~\eqref{eq:intro:SPDE} for all $u_0\in\,_\nu B_{q,p}^{3-s-4(1+\kappa)/p}(\sO)$. 
We first observe that this implies that, for every $\varepsilon>1/3$, one can find suitable $p$ and $q$ such that the same conclusion holds for all $u_0\in L_{\F_0}^0(\Omega;\, \nH^{\varepsilon, q}(\sO))$ as claimed. Indeed, let $\varepsilon \in (1/3, 1/2)$ and $\theta \in (1/3, \varepsilon)$. Now choose $s=2/3$, $\kappa = 0$, $q = d/(1+\theta)$ and assume equality in \eqref{eq:thm:introMain}, which yields $p=12/(7-3\theta)$. It can be checked that these choices are admissible and the embedding $\nH^{\varepsilon, q}(\sO) \hookrightarrow\,_\nu B^{d/q-1}_{q,p}(\sO)$ holds true, which implies the claim.

The admissible regularity of the initial condition is controlled by the choices of the parameters $p$, $\kappa$, $s$, and $q$, whose roles we discuss now more precisely. A larger choice of $\kappa$ corresponds to a stronger admissible degeneracy at $t\sim 0$ and allows for less regular initial data. This mechanism reaches its limit at
\begin{equation*}
\kappa=\kappa_c\ce \Big(
1-\frac{1}{4}\Big(s+\frac{d}{q}\Big)\Big)p-1,
\end{equation*}
where~\eqref{eq:thm:introMain} holds with equality. For $\kappa=\kappa_c$, the initial trace space is $_\nu B_{q,p}^{d/q-1}(\sO)$ -- the critical space suggested by the scaling of the deterministic Cahn--Hilliard equation. For $p=q=2$ and $s=\kappa=0$, the initial trace space in Theorem~\ref{thm:introMain} is $\bpsn{1}$. This space is \emph{subcritical} for $d=3$ and \emph{critical} for $d=4$.
The critical case $\kappa=\kappa_c$ illustrates how the freedom in selecting $q$ translates into flexibility in the smoothness index of the trace space $_\nu B_{q,p}^{d/q-1}(\sO)$. More precisely, the relations given by~\eqref{eq:thm:introMain}, $\kappa_c<p/2-1$, and $q\in [2,\frac{2d}{2+s})$ ensure $d/q-1 \in (1/3,d/2-1]$, which implies that, by choosing  suitable $p$, $s$, and $\kappa$ as well as sufficiently large $q$, the critical smoothness may be chosen  in $(1/3,1/2]$ for $d=3$ and in $(1/3,1]$ for $d=4$. As a consequence, we can upgrade the global well-posedness in $d=3$ for $u_0\in L_{\F_0}^0(\Omega;H^1(\sO))$ from Section \ref{sec:globalWP} to rougher initial data $u_0\in L_{\F_0}^0(\Omega;H^{1/2}(\sO))$.

Another main assertion of Theorem~\ref{thm:introMain} is instantaneous regularisation of the solution, meaning that the low regularity of the initial condition has no effect on the regularity of the solution for positive times. In particular, for any $t_0>0$, it holds that $u(t_0)\in H^1(\sO)$, and it is shown in Section~\ref{section_rough_initial_data} that $u|_{[t_0,\infty)}$ agrees with the unique solution obtained from variational techniques.

Regarding the proof of global well-posedness in case \ref{case:pq2}, we work in the variational setting with the Gelfand triple $(\bpsn{3},\bpsn{1},\bpsn{-1})$. The trace space $\bpsn{1}$ already being \emph{critical} in $d=4$ complicates the verification of the blow-up criterion in $d=4$, which then contains an additional term due to criticality. Deducing an energy estimate for the latter based on the subcritical energy estimates can be interpreted as verifying a \emph{nonlinear elliptic maximal regularity estimate} for solutions of the parabolic problem~\eqref{eq:intro:SPDE} explicitly. The proof strategy for global well-posedness in the weakened setting with $p=q=2$ via suitable energy estimates in dimensions $d\in \{3,4\}$ is outlined in Subsection~\ref{subsec:overviewglobalWP}.

\subsection{Main contributions}
\label{subsec:contributions}

We first comment on the novelty of our results obtained using only the critical variational setting, i.e.\ Hilbert-space techniques, in Section \ref{sec:globalWP}, before discussing further improvements via $L^p(L^q)$-theory. 

Regarding global well-posedness in dimension $3$, building on the existence and uniqueness results of~\cite{Scarpa2018}, $H^1$- and, under higher regularity assumptions, $H^2$-moment estimates were established for multiplicative Fourier noise in~\cite{Antonopoulou23}. Our results in $d=3$ via Hilbert-space techniques, cf.\ Theorem~\ref{thm:global_well_posedness_dim_3}, generalise the a priori estimates of~\cite[Thm.~2.2]{Antonopoulou23} with $\varepsilon=p=1$ to more general noise: in our setting the noise may depend on $\grad u$ and more freely on $u$, and less regular trace-class noise is covered. During the preparation of this manuscript, global well-posedness for transport noise and constant mobility was independently obtained in~\cite{diPrimioPapiniScarpa26}, including potentials with a singular part. For the regular Allen--Cahn-type potential considered here, we obtain energy estimates in stronger norms, since transport noise is covered by Theorem \ref{thm:introMain} with $\psi^0=0$.
    
Much less is known about $H^1$-global well-posedness of the stochastic Cahn--Hilliard equation with multiplicative noise in the critical dimension $4$. Global existence of martingale solutions on smooth domains in $d=4$ has been established in~\cite[Thm.~2.2]{Scarpa2021b} for deterministic initial values $u_0\in H^1$, but their regularity statements do not apply to the cubic nonlinearity considered here and pathwise uniqueness is not shown. As far as the authors are aware, the present manuscript is the first to establish global well-posedness for probabilistically strong solutions with initial data in $H^1$ for the stochastic Cahn--Hilliard equation in dimension four, both for gradient-independent noise of linear growth ($\psi^1=0$) and for transport noise; see Theorem~\ref{thm:global_well_posedness_dim_4-complete}. 

To the best of our knowledge,~\cite[Section 7.3]{agrestiNonlinearParabolicStochastic2022}
provides the only previous treatment of the stochastic Cahn--Hilliard equation within the framework of stochastic $L^p(L^q)$--theory. Therein, a local theory for the \emph{almost very weak setting} is established. Our results in Section~\ref{section_rough_initial_data} go beyond this treatment in several directions. 

Whereas~\cite[Section 7.3]{agrestiNonlinearParabolicStochastic2022} establishes local well-posedness for~\eqref{eq:intro:SPDE}, we obtain global well-posedness by combining the local theory with blow-up criteria from~\cite[Theorem 4.10]{agrestiNonlinearParabolicStochastic2022a} and energy estimates, obtained by means of variational techniques in~\cite{agrestiCriticalVariationalSetting2024a} for $d\in\{1,2\}$ and in Section~\ref{sec:globalWP} for $d\in\{3,4\}$ in the \textit{weakened setting}, which is stronger than the almost very weak setting. The transference principle formulated in Theorem~\ref{gwp:rough:initial:data} then yields global well-posedness of~\eqref{eq:intro:SPDE} for initial values in the \emph{critical} trace space $B_{q,p}^{d/q-1}(\sO)$ whenever $d\in \{1,\dots,4\}$; see Corollary~\ref{cor:GWP:rough}.

\subsection*{Overview}
The rest of this paper is organised as follows. In Section~\ref{sec:preliminaries}, we introduce the functional-analytic background for our analysis of~\eqref{eq:intro:SPDE} as a stochastic evolution equation. In Section~\ref{sec:variational:lwp}, we prove local well-posedness for~\eqref{eq:intro:SPDE} in variational settings, which we extend to global well-posedness in Section~\ref{sec:globalWP} by deriving suitable energy estimates to verify appropriate blow-up criteria. In Section~\ref{section_rough_initial_data}, we extend the local well-posedness theory to rougher settings with the help of $L^p(L^q)$-theory and connect this via regularisation techniques to the variational settings.

\subsection*{Acknowledgements}

The authors wish to thank Antonio Agresti and Mark Veraar for organising the Oberwolfach Seminar ``Stochastic Partial Differential Equations in Critical Spaces'' from which this project originated based on their proposal.  Further, the authors would like to thank Antonio Agresti, Max Sauerbrey, and Mark Veraar for helpful discussion and comments. 

\section{Preliminaries}\label{sec:preliminaries}
\subsection{Notation}\label{subsec:notation}

Throughout the paper, $\sO\seq\R^d$ denotes a bounded domain (i.e.\ an open, connected, and bounded subset) with $C^{\infty}$-boundary, $\nu$ the outward unit normal field on $\partial\sO$, and $\partial_\nu=\nu\cdot\grad$ the associated normal derivative.
For nonnegative quantities $a$ and $b$, we write $a\lesssim b$ if $a\le Cb$ for some constant $C>0$ whose dependence on the relevant parameters is indicated when needed, and $a\eqsim b$ if $a\lesssim b\lesssim a$.

Let $(S,\Sigma,\mu)$ be a measure space, let $1\leq p<\infty$, and let $X$ be a Banach space. Then $L^p(S,\mu;X)$ stands for the space of all (equivalence classes of) strongly $\mu$-measurable, $p$-inte\-grable functions $f\colon S\to X$. 
	Moreover, $L^\infty(S,\mu;X)$ stands for the space of all (equivalence classes of) strongly $\mu$-measu\-rable functions $f\colon S\to X$ that are $\mu$-essentially bounded.
	Throughout the paper, if $X=\R$, then $X$ is omitted from the function spaces.
	If $v\colon S\to[0,\infty]$ is measurable, then $v\mu$ stands for the measure with density $v$ 
	with respect to $\mu$; $\mu$ is omitted if it is the Lebesgue measure or a probability measure. 
For a sub-$\sigma$-algebra $\sG\subseteq\Sigma$ we write $L^0_\sG(S,\mu;X)$ for the set of strongly $\mu|_\sG$-measurable functions $f\colon S\to X$.

For an interval $I=(a,b)$, $0\le a<b\le \infty$, and $\kappa\in [0,\frac{p}{2}-1)\cup\{0\}$, we write $w_\kappa^a(t)\ce (t-a)^\kappa$, $t>a$, and $w_\kappa \ce w_\kappa^0$.
For $k\in\N$, $W^{k,p}(I,w_\kappa^a;X)$ denotes the weighted Sobolev space of all $f \in L^p(I,w_\kappa^a;X)$ such that $D^\ell f \in L^p(I,w_\kappa^a;X)$ for all $\ell\le k$. For $s\in(0,1)$, we define 
\begin{equation*}
    H^{s,p}(I,w_\kappa^a;X)\ce [L^p(I,w_\kappa^a;X), W^{1,p}(I,w_\kappa^a;X)]_s.
\end{equation*}
For an interval $J\subseteq \overline{I}$ and  $\mathcal{A}\in \{L^p, H^{s,p}, W^{k,p}\}$, we write $\mathcal{A}_{\textup{loc}}(J,w_\kappa^a;X)$ for the set of all (equivalence classes of) strongly measurable maps $f\colon J\to X$ such that $f\in \mathcal{A}(J',w_\kappa^a;X)$ for all bounded intervals $J'$ with $\overline{J'}\subseteq J$. 
If $G\subseteq \R^d$, then $L^p_\loc(G)$ denotes the space of all $f\colon G\to \R$ such that for all compact subsets $K\subseteq G$, we have $f|_K\in L^p(K)$. Analogous notation is used for other function spaces.  

For $\theta_1,\theta_2\in (0,1)$, $C_\loc^{\theta_1,\theta_2}((a,b)\times G)$ denotes the space of all functions $v\colon (a,b)\times G\to \R$ such that for all $a<c<c'<b$ we have 
$$|v(t,x)-v(t',x')|\lesssim_{c,c'} |t-t'|^{\theta_1}+|x-x'|^{\theta_2}, \quad t,t' \in [c,c'],\,  x,x'\in G.$$
This definition is extended to $\theta_1,\theta_2\ge 1$ by requiring that the partial derivatives $\partial^{\alpha,\beta} v$ (with $\alpha\in \N_0$, $\beta\in \N_0^d$) exist and belong to $C_\loc^{\theta_1-|\alpha|, \theta_2-|\beta|}((a,b)\times G)$ for all $\alpha\le \lfloor \theta_1\rfloor,$ and $\sum_{i=1}^d \beta_i \le \lfloor \theta_2\rfloor$. 

For Banach spaces $X,Y$ and Hilbert spaces $H, H_1, H_2$, we write $\sL(X,Y)$ for the space of bounded linear operators, $\sL_2(H_1,H_2)$ for the space of Hilbert--Schmidt operators, and $\gamma(H,X)$ for the space of $\gamma$-radonifying operators.

\subsection{Stochastic setup}\label{subsec:stochastic_setup}
We fix a filtered probability space $(\Omega,\F,(\F_t)_{t\geq 0}, \mathbb{P})$ satisfying the usual conditions and denote by $\sP$ the progressive $\sigma$-algebra on $[0, \infty)\times\Omega$. Let $(W^n)_{n\geq1}$ be independent standard real-valued $(\F_t)_{t\geq0}$-Brownian motions, and let $W$ denote the associated $\ell^2$-cylindrical Brownian motion constructed from the canonical basis $(e_n)_{n\geq1}$ of $\ell^2$, so that $W(t)e_n=W^n(t)$. For a UMD Banach space $X$, stochastic integrals of progressively measurable $\gamma(\ell^2,X)$-valued processes against $W$ are understood in the sense of~\cite{vanneervenStochasticIntegrationBanach2015}. If $X=H$ is a Hilbert space, then $\gamma(\ell^2,H)=\sL_2(\ell^2,H)$ isometrically~\cite[Ch.~9]{hytonenAnalysisBanachSpaces2017}, and this construction agrees with the classical Hilbert-space stochastic integral. 

\subsection{Stochastic evolution equations}
Let $X_0,X_1$ be UMD Banach spaces of type $2$ and  $X_1\hookrightarrow X_0$. For $\theta\in(0,1)$ and $p\in[1,\infty]$, set
$    X_\theta\ce [X_0,X_1]_\theta
$, $X_{\theta,p}\ce (X_0,X_1)_{\theta,p}$, and $X_{p,\kappa}^{\textup{Tr}}\ce X_{1-\frac{1+\kappa}{p},p}$ for $p\in (1,\infty)$, $\kappa\in (-1,p-1)$. Assume that $A\in \mathscr{L}(X_1,X_0)$.

\subsubsection{Linear stochastic evolution equations}
Throughout, we adopt the relevant notions of solution and stochastic maximal regularity established in~\cite{agrestiNonlinearParabolicStochastic2022,agrestiNonlinearParabolicStochastic2022a,agrestiNonlinearSPDEsMaximal2025}, which we briefly introduce here. 
Let $0\le a < T<\infty$, let $\tau\colon\Omega\to[a,T]$ be a stopping time, and
consider the stochastic evolution equation
\begin{equation}\label{eq:linear-stochastic-evolution}
\left\{
\begin{aligned}
    \dd  u(t)+Au(t)\,\dd t
        &= f(t)\,\dd  t
        +g(t)\,\dd W(t),\quad t\in[a,T],\\
    u(a)&=u_a.
\end{aligned}\right.
\end{equation}

\begin{definition}[Strong solution]\label{def:linear:see:solution:strong}
    Let $f\in L^1(a,\tau;X_0)$ a.s.\ and $g\in L^2(a,\tau;\gamma(\ell^2,X_{1/2}))$ a.s.\ be progressively measurable. A progressively measurable process $u\colon[a,\tau]\times\Omega\to X_0$ is called a \emph{strong solution} to~\eqref{eq:linear-stochastic-evolution} on $[a,\tau]$ if $u\in C([a,\tau];X_0)\cap L^2(a,\tau;X_1)$ a.s.\ and a.s.\ for all $t\in [a,\tau]$ 
    $$u(t)-u_a +\int_a^t Au(s)\, \dd s = \int_a^t f(s)\,\dd s+ \int_a^t g(s)\,\dd W(s).$$
\end{definition}
We recall the notion of stochastic maximal regularity from~\cite{agrestiNonlinearSPDEsMaximal2025}. For a given $s\ge0$, we consider the weight $w_\kappa^s(t)= (t-s)^\kappa$ for $t>s$.
\begin{definition}[Stochastic maximal regularity]
    Let $p\in [2,\infty), \kappa\in [0,p/2-1)\cup\{0\}$.
    \begin{enumerate}
        \item We write $A\in \mathcal{SMR}_{p,\kappa}$ if for all $T\in (0,\infty)$ there exists a constant $C_T\ge 0$ such that for all $a\in [0,T]$, every stopping time $\tau\colon\Omega\to [a,T]$ and every progressively measurable $f\in L^p(\Omega;L^p(a,\tau,w_\kappa^a;X_0))$ and $g\in L^p(\Omega;L^p(a,\tau,w_\kappa^a;\gamma(\ell^2,X_{1/2})))$ there exists a strong solution $u$ to~\eqref{eq:linear-stochastic-evolution} with $u_a=0$ that is unique in $L^p(\Omega;L^p(a,\tau,w_\kappa^a;X_1))$ and satisfies
        \begin{equation}
            \nrm{u}_{L^p(\Omega;L^p(a,\tau,w_\kappa^a;X_1))}\le C_T\nrm{f}_{L^p(\Omega;L^p(a,\tau,w_\kappa^a;X_0))}+C_T\nrm{g}_{L^p(\Omega;L^p(a,\tau,w_\kappa^a;\gamma(\ell^2,X_{1/2})))}.
        \end{equation}
        \item For $p\in(2,\infty)$ and $ \kappa\in [0,p/2-1)$, we write $A\in \mathcal{SMR}_{p,\kappa}^\bullet$ if $A\in \mathcal{SMR}_{p,\kappa}$ and for all $\theta\in(0,1/2)$ and $T\in (0,\infty)$ there exists a constant $C_{T,\theta}\ge 0$ such that for all $[a,b]\seq[0,T]$, for every progressively measurable $f\in L^p(\Omega;L^p(a,b,w_\kappa^a;X_0))$ and $g\in L^p(\Omega;L^p(a,b,w_\kappa^a;\gamma(\ell^2,X_{1/2})))$ the strong solution $u$ to~\eqref{eq:linear-stochastic-evolution} with $u_a=0$ on $[a,b]$ satisfies
        \begin{equation}
            \nrm{u}_{L^p(\Omega;H^{\theta,p}(a,b,w_\kappa^a;X_{1-\theta}))}\le C_{T,\theta}\nrm{f}_{L^p(\Omega;L^p(a,b,w_\kappa^a;X_0))}+C_{T,\theta}\nrm{g}_{L^p(\Omega;L^p(a,b,w_\kappa^a;\gamma(\ell^2,X_{1/2})))}.
        \end{equation}
        \item For $p=2$ and $\kappa=0$, we write $A\in \mathcal{SMR}_{2,0}^\bullet$ if $A\in \mathcal{SMR}_{2,0}$ and for all $T\in (0,\infty)$ there exists a constant $C_T\ge 0$ such that for all $[a,b]\seq [0,T]$, every progressively measurable $f\in L^2(\Omega\times(a,b);X_0)$ and $g\in L^2(\Omega\times(a,b);\gamma(\ell^2,X_{1/2}))$, the strong solution to~\eqref{eq:linear-stochastic-evolution} with $u_a=0$ on $[a,b]$ satisfies
        \begin{equation}
            \nrm{u}_{L^2(\Omega;C([a,b];X_{1/2}))}\le C_T\nrm{f}_{L^2(\Omega;L^2(a,b;X_0))}+C_T\nrm{g}_{L^2(\Omega;L^2(a,b;\gamma(\ell^2,X_{1/2})))}.
        \end{equation}
    \end{enumerate}
\end{definition}
If for $p\in [2,\infty), \kappa\in [0,p/2-1)\cup\{0\}$ an operator $A$ belongs to the class $ \mathcal{SMR}_{p,\kappa}$, it follows that we can consider, in addition, initial values $u_a \in L_{\F_a}^p(\Omega;X_{p,\kappa}^{\textup{Tr}})$ with corresponding regularity estimates, see~\cite[Proposition 3.9]{agrestiNonlinearParabolicStochastic2022a}. 
For more information about stochastic maximal regularity we refer to~\cite{agrestiNonlinearParabolicStochastic2022,agrestiNonlinearParabolicStochastic2022a,agrestiNonlinearSPDEsMaximal2025} and the references given therein. 
In particular, if $X_0$ is isomorphic to a closed subspace of an $L^q$-space and $A$ is viewed as an unbounded linear operator on $X_0$ with domain $\mathsf{D}(A)=X_1$, then a sufficient condition for $A\in \mathcal{SMR}_{p,\kappa}^\bullet$ is that  there exists some $\lambda>0$ such that $\lambda+A$ is sectorial and has a bounded $H^\infty$-calculus of angle $<\pi/2$. In addition, an almost-sure version holds true, where the initial data and the forcing terms are not required to be $L^p$-integrable in $\Omega$, see~\cite[Proposition 3.11]{agrestiNonlinearSPDEsMaximal2025}.

\subsubsection{Semilinear stochastic evolution equations}
For $0\le a<T\le \infty$, we consider semilinear stochastic evolution equations of the form
\begin{equation}\label{eq:semilinear-stochastic-evolution}
\left\{
\begin{aligned}
    \dd  u(t)+Au(t)\,\dd t
        &= F(t,u(t))\,\dd  t
        +G(t,u(t))\,\dd W(t), \quad t\in [a,T],\\
    u(a)&=u_a.
\end{aligned}\right.
\end{equation}
\begin{assumption}\label{Assumption:semilinear:evolution:equation}
    Let $0\le a<T \le\infty$, $p\in[2,\infty)$, and $\kappa\in [0,p/2-1)\cup\{0\}$. Suppose that the mappings $$F\colon \Omega\times[a,T]\times X_1 \to X_0\quad \text{and}\quad G\colon \Omega\times[a,T]\times X_1\to \gamma(\ell^2,X_{1/2})$$ are strongly progressively measurable and satisfy the following local Lipschitz condition: There exist $m\in \N$ and constants $\varphi_j\in (1-\frac{1+\kappa}{p},1)$, $\beta_j\in (1-\frac{1+\kappa}{p},\varphi_j]$, $\rho_j\ge 0$ for $j\in \{1,...,m\}$ such that  for each $n\ge 1$ there is $L_n\ge 0$ such that for all $u,v\in X_1$ with $\nrm{u}_{X_{1-(1+\kappa)/p,p}},\nrm{v}_{X_{1-(1+\kappa)/p,p}}\le n $, 
    \begin{equation}\label{eq:assumption:semilinear:estimate:local:lipschitz}
        \begin{aligned}
            \nrm{F(t,\omega,u)-F(t,\omega,v)}_{X_0} &+\nrm{G(t,\omega,u)-G(t,\omega,v)}_{\gamma(\ell^2,X_{1/2})}\\
            &\le L_n\sum_{j=1}^m \big(1+\nrm{u}_{X_{\varphi_j}}^{\rho_j}+\nrm{v}_{X_{\varphi_j}}^{\rho_j}\big)\nrm{u-v}_{X_{\beta_j}}
        \end{aligned}
    \end{equation}
    a.s.\ for all $t\in [a,T]$. 
    Moreover, we assume that 
    \begin{equation}\label{eq:subcriticality}
        \rho_j\Big(\varphi_j-1+\frac{1+\kappa}{p}\Big)+\beta_j\le 1 \quad \text{for all}\quad j\in\{1,...,m\}.
    \end{equation}
\end{assumption}
\begin{remark}\label{rem:assumption:semilinear:evolution:equation}
    Assumption~\ref{Assumption:semilinear:evolution:equation} is a special case of~\cite[Hypothesis $\mathrm{H}$]{agrestiNonlinearParabolicStochastic2022a}. Therein, the quasilinear case is also treated and the conditions on $F$ and $G$ are more flexible. In their notation, we have $F=F_c$ and $G=G_c$. 
    In addition, Assumption~\ref{Assumption:semilinear:evolution:equation} is slightly more flexible than~\cite[Assumption 4.1]{agrestiNonlinearSPDEsMaximal2025} in that it allows for $F$  and $ G$ to depend on $\omega$ and $t$. The subcriticality condition~\eqref{eq:subcriticality} here coincides with theirs, provided that $\varphi_j=\beta_j$ for $ j\in\{1,...,m\}$. In this case, for $\rho_j>0$, \eqref{eq:subcriticality} may be reformulated as     \[ \frac{1+\kappa}{p}\le \frac{\rho_j+1}{\rho_j}(1-\beta_j),\quad j\in\{1,...,m\}.\]
\end{remark}
\begin{remark}\label{rem:subcriticality:invariance}
The subcriticality condition~\eqref{eq:subcriticality} is invariant under interpolation with the trace space $X_{p,\kappa}^{\textup{Tr}}$.
     To be more precise, let $p$ and $\kappa$ be as in Assumption~\ref{Assumption:semilinear:evolution:equation}, let $(\varphi,\rho,\beta)$ be one of the triples appearing there, and let $\alpha\in (\varphi,1)$. Then 
    \[X_\varphi \hookleftarrow X_{\varphi,1}=(X_{p,\kappa}^{\textup{Tr}}, X_{\alpha,\infty})_{\rho_\alpha,1}\quad \text{with}\quad\rho_\alpha = \frac{\varphi -1+\frac{1+\kappa}{p}}{\alpha-1+\frac{1+\kappa}{p}}.\]
    Therefore, since $X_\alpha\hookrightarrow X_{\alpha,\infty}$,
    $\| u\|_{X_\varphi}^\rho \lesssim \|u\|_{X_{p,\kappa}^{\textup{Tr}}}^{(1-\rho_\alpha)\rho}\| u\|_{X_{\alpha}}^{\rho_\alpha \rho}.$ This allows us to factor the $X_{p,\kappa}^{\textup{Tr}}$ part into the constant $L_n$, and obtain~\eqref{eq:assumption:semilinear:estimate:local:lipschitz} with the triple $(\alpha, \rho_\alpha \rho,\beta)$ instead of $(\varphi, \rho,\beta)$. 
    However, note that
    \[\rho \rho_\alpha\Big(\alpha-1+\frac{1+\kappa}{p}\Big) = \rho\Big(\varphi-1+\frac{1+\kappa}{p}\Big).\]
    This showcases that it is possible to lower the polynomial growth at the cost of imposing more regularity on the  interpolation space $X_\varphi$.
\end{remark}
\begin{definition}[Strong solution]
    Suppose that Assumption~\ref{Assumption:semilinear:evolution:equation} is satisfied for some $p\in[2,\infty)$ and $\kappa\in [0,p/2-1)\cup\{0\}$. For $T<\infty$, a pair $(u,\sigma)$ is called an $L^p_\kappa$-strong solution of~\eqref{eq:semilinear-stochastic-evolution} if $\sigma\colon\Omega\to[a,T]$ is a stopping time and $u\colon [a,\sigma]\to X_0$  is strongly progressively measurable, 
    \begin{equation*} u\in L^p(a,\sigma, w_\kappa^a;X_1)\cap C\big([a,\sigma];X_{1-\frac{1+\kappa}{p},p}\big),
    \end{equation*}
    and a.s.\ for all $t\in [a,\sigma]$ it holds that 
    \begin{equation}\label{eq:def:semilinear:strong:solution}
        u(t)-u_a +\int_a^t Au(s)\,\dd s = \int_a^t F(s,u(s))\, \dd s +\int_a^t\1_{[a,\sigma]}(s) G(s,u(s))\, \dd W(s).
    \end{equation}
    For $T=\infty$, $(u,\sigma)$ is called an $L^p_\kappa$-strong solution of~\eqref{eq:semilinear-stochastic-evolution} if $\sigma\colon\Omega\to [a,\infty)$ is a stopping time and $u\colon[a,\sigma]\to X_0$ is strongly progressively measurable with $u\in L^p_{\loc}(a,\sigma,w_\kappa^a;X_1)\cap C([a,\sigma]\cap[a,\infty);X_{1-\frac{1+\kappa}{p},p})$ and~\eqref{eq:def:semilinear:strong:solution} holds a.s.\ for all $t\in [a,\sigma]$. 
\end{definition}
\begin{remark}
    Note that due to the choice of $\kappa$ and Hölder's inequality, the integrals in~\eqref{eq:def:semilinear:strong:solution} are well-defined. 
    For more details on the differences between $T<\infty$ and $T=\infty$ in this setting, we refer the reader to~\cite[Section 4.3]{agrestiNonlinearParabolicStochastic2022a}.
\end{remark}
\begin{definition}[Local solution]\label{def:solution:local}
     Suppose that Assumption~\ref{Assumption:semilinear:evolution:equation} is satisfied for some $p\in[2,\infty)$ and $\kappa\in[0,p/2-1)\cup\{0\}$. 
     \begin{enumerate}
         \item A pair $(u,\sigma)$ is called an $L^p_\kappa$-\emph{local solution to~\eqref{eq:semilinear-stochastic-evolution}} if $\sigma\colon\Omega\to[a,\infty]$ is a stopping time, $u\colon[a,\sigma)\to X_0$ is a strongly progressively measurable process, there exists an increasing sequence of stopping times $(\sigma_n)_{n\ge 1}$ such that $\lim_{n\to\infty} \sigma_n=\sigma$ a.s., and for all $n\ge 1$ it holds that $(u|_{[a,\sigma_n]},\sigma_n)$ is an $L^p_\kappa$-strong solution to~\eqref{eq:semilinear-stochastic-evolution}. The sequence $(\sigma_n)_{n\ge 1}$ is called \emph{a localising sequence} for $(u,\sigma)$.
         \item An $L^p_\kappa$-local solution $(u,\sigma)$ is called \emph{unique} if for every $L^p_\kappa$-local solution $(v,\tau)$ one has $u=v$ a.s.\ on $[a,\tau\wedge \sigma)$.
         \item An $L^p_\kappa$-local solution $(u,\sigma)$ is called \emph{$L^p_\kappa$-maximal} if for any other unique $L^p_\kappa$-local solution $(v,\tau)$ to~\eqref{eq:semilinear-stochastic-evolution} one has that a.s.\ $\tau\le \sigma$ and $u=v$ on $[a,\tau)$.
         \item An $L^p_\kappa$-local solution $(u,\sigma)$ to~\eqref{eq:semilinear-stochastic-evolution} is called \emph{global} if $\sigma=\infty$ a.s.
     \end{enumerate}
\end{definition}
Note that, by definition, any $L^p_\kappa$-maximal solution is unique. 

\subsection{Function spaces}\label{subsec:function_spaces}
In this section we introduce the function spaces and operators which pave the way for the $L^p$-$L^q$-theory we will develop in Section~\ref{section_rough_initial_data}. Suppose $u$ is a smooth function on $\overline{\sO}$. We write $\gamma_1$ for the boundary operator $\gamma_1 u=\gamma_0 \partial_\nu u$, where $\gamma_0 u$ denotes the trace of the function $u$ on the boundary $\partial \sO$.
We remark that $\gamma_1$ extends to a continuous linear mapping $ \gamma_1\colon H^{s,p}(\sO)\to B_{p,p}^{s-1-1/p}(\partial\sO)$ if and only if $s>1+1/p$.
Let $1<q<\infty$ and set 
\begin{align*}
     \mathsf{D}(A_q)\ce \{ u \in H^{4,q}(\sO)\colon \gamma_1 u = \gamma_1 \Delta u=0\}.
\end{align*}
The strong $L^q(\sO)$-realisation of the Bi-Laplacian with the boundary conditions \eqref{bc_mu_new}, \eqref{bc_phi_new} is defined by 
\begin{align*}
     A_q \colon \mathsf{D}( A_q) \seq L^q(\sO)\to L^q(\sO), \quad A_q u\ce \Delta^2 u.
\end{align*}
By~\cite[Theorem~2.3]{denkNewThoughtsOld2004} there exists a sufficiently large $\lambda>0$ such that $ \lambda+ A_q$ has a bounded $H^\infty$-calculus of angle less than $\pi/2$. Therefore, we can consider the interpolation-extrapolation scale, in the sense of~\cite[Chapter~V]{amannLinearQuasilinearParabolic1995}, generated by the operator $A_q$ which we denote by
\begin{align*}
    \big( _\nu H^{s, q}(\sO), {_\nu\Delta_{s,q}^2} \big)_{s \ge -4}.
\end{align*}
Readers familiar with the concept might notice that we index the scale by Sobolev regularity and not by orders of the operator. In our case this is more convenient because we also consider the Neumann Laplacian in the scale.
We further point out that the theory requires the underlying operator to be invertible, but this can easily be dealt with by a shift argument which is independent of the parameter $\lambda>0$. To be more precise, for $k\in \N$ we have $_\nu H^{4k, q}(\sO)=\mathsf{D}((\lambda + A_q)^k)$, $_\nu H^{0, q}(\sO)=L^q(\sO),$ and the extrapolation space $_\nu H^{-4, q}(\sO)$ can be defined as the dual space of $_\nu H^{4, q'}(\sO)$ with respect to the duality pairing $\langle \cdot,\cdot \rangle_{L^q\times L^{q'}}$ with
$1/q+1/q'=1$. The definition of the spaces in between is then given via complex interpolation, and since the family $(\nH^{s, q}(\sO))_{s\geq - 4}$ is closed under complex interpolation we obtain for $s_0<s_1$, $\theta\in (0,1)$ and $s=(1-\theta)s_0+\theta s_1$,
\begin{align*}
 _\nu H^{ s, q}(\sO)  = [_\nu H^{s_0, q}(\sO),_\nu H^{s_1, q}(\sO)]_{\theta}.
\end{align*}
Once these spaces are given, suitable representatives of the Bi-Laplacian within these can be defined by  
\begin{align}
  _\nu\Delta_{s,q}^2   \ce \begin{cases}
        \nH^{ s, q}(\sO)\text{-realisation of} \ A_q, & s\geq 0,\\
        \text{closure of} \ A_q \ \text{in} \ \nH^{s, q}(\sO), & -4 \leq s<0.
    \end{cases}
\end{align}
In particular, we recover $A_q={_\nu\Delta_{0,q}^2}$ and as $A_q$ and 
$ _\nu\Delta_{s,q}^2$ can be shown to be similar, the operator 
\begin{align*}
     _\nu\Delta_{s,q}^2:\mathsf{D}(_\nu\Delta_{s,q}^2)\seq { _\nu H^{ s, q}(\sO)}\to{ _\nu H^{ s, q}(\sO)} 
\end{align*}
with $\mathsf{D}({_\nu \Delta_{s,q}^2})={_\nu H^{s+4, q}(\sO)}$ 
 has a bounded $H^\infty$-calculus of angle less than $\pi/2$ again. Therefore, since $\nH^{ s, q}(\sO)$ is isomorphic to $L^q(\sO)$ we have $_\nu \Delta_{s,q}^2\in \mathcal{SMR}_{p,\kappa}^\bullet(T)$ for all $2\leq q <\infty$ as a consequence of~\cite[Theorem~3.14]{agrestiNonlinearSPDEsMaximal2025}.
Later on we will also need extrapolations of the Neumann Laplacian 
with respect to the family of function spaces introduced above. For $q\in(1,\infty)$ we set $\mathsf D({}_\nu\Delta_q)=\{u\in H^{2,q}(\sO):\gamma_1 u=0\}$; the strong Neumann Laplacian is then defined as the unbounded operator
\[
{}_\nu\Delta_q\colon \mathsf D({}_\nu\Delta_q)\seq L^q(\sO)\to L^q(\sO)
,\qquad {}_\nu\Delta_q u\ce\Delta u,
\]
and we further define
\begin{align}
  _\nu\Delta_{s,q}   \ce \begin{cases}
        \nH^{ s, q}(\sO)\text{-realisation of} \ {}_\nu\Delta_q, & s\geq 0,\\
        \text{closure of} \ {}_\nu \Delta_q \ \text{in} \ \nH^{s, q}(\sO), & -4 \leq s<0.
    \end{cases}
\end{align}
In particular, it can be verified that we have 
\begin{align*}
     _\nu\Delta_{s,q}:\mathsf{D}(_\nu\Delta_{s,q})\seq { _\nu H^{ s, q}(\sO)}\to{ _\nu H^{ s, q}(\sO)} 
\end{align*}
with $\mathsf{D}({_\nu \Delta_{s,q}})={_\nu H^{s+2, q}(\sO)}$ as well as $(_\nu\Delta_{s,q})^2={_\nu\Delta_{s,q}^2}$.
Moreover, the operators introduced in this section are consistent in the following sense, for all $1<q_0,q_1<\infty$ and $-4\leq s_0,s_1 \leq 4$ we have
\begin{align}
    {}_\nu\Delta_{s_0,q_0}^2 u &= {}_\nu\Delta_{s_1,q_1}^2 u \quad  \text{ on  } \nH^{s_0+4, q_0}(\sO)\cap \nH^{ s_1+4, q_1}(\sO), \\
    _\nu\Delta_{s_0,q_0} u &= {}_\nu\Delta_{s_1,q_1} u \quad  \text{ on  } _\nu H^{s_0+2, q_0}(\sO)\cap \nH^{s_1+2, q_1}(\sO).
\end{align}
Using results on interpolation with boundary conditions, see e.g.~\cite{Seeley1972,roodenburgComplexInterpolationPowerweighted2026}, and invoking the definition of the extrapolation space via duality we get the characterisation
\begin{equation}\label{Bilaplace-concrete}
\bpsn{s,q}=
\begin{cases}
(\bpsn{-s,q'})^\ast, & -4 \leq s\leq  0,\\
H^{s,q}(\sO), & 0\le s<1+\frac{1}{q},\\[2pt]
\{u\in H^{s,q}(\sO):\gamma_1 u=0\}, & 1+\frac{1}{q}<s<3+\frac{1}{q},\\[2pt]
\{u\in H^{s,q}(\sO):\gamma_1 u=\gamma_1\Delta u=0\}, & 3+\frac{1}{q}<s\leq 4;
\end{cases}
\end{equation}
the values $s\in\{1+\frac{1}{q},\,3+\frac{1}{q}\}$ being excluded as the description is more involved in this case.
Since the phase spaces in the (stochastic) maximal $L^p$-regularity framework arise as real interpolation spaces we further define
\begin{align*}
 {_\nu B^{ s}_{q,p}(\sO)}  \coloneq \left(\bpsn{s_0,q},\,_\nu H^{s_1, q}(\sO)\right)_{\theta,p}
\end{align*} 
for $1< p< \infty$, where $-4 \leq s_0<s_1 \leq 4$, $\theta\in (0,1)$ and $s=(1-\theta)s_0+\theta s_1$.
Note that this definition is independent of the chosen $s_0,s_1$ due to the reiteration property of the real interpolation functor. Again, these spaces admit a description as in \eqref{Bilaplace-concrete}; cf.~\cite{Guidetti1991}. For $1< p<\infty$
we have
\begin{equation}
{_\nu B^{ s}_{q,p}(\sO)}=
\begin{cases}
( {_\nu B^{ -s}_{q',p'}(\sO)})^\ast, & -4 < s\leq  0,\\
{ B^{ s }_{q,p}(\sO)}, & 0\le s<1+\frac{1}{q},\\[2pt]
\{u\in { B^{ s }_{q,p}(\sO)}:\gamma_1 u=0\}, & 1+\frac{1}{q}<s<3+\frac{1}{q},\\[2pt]
\{u\in { B^{ s }_{q,p}(\sO)}:\gamma_1 u=\gamma_1\Delta u=0\}, & 3+\frac{1}{q}<s< 4.
\end{cases}
\end{equation}
Throughout the paper we will rely on Sobolev embeddings for the scale of function spaces introduced above. We collect them in the following lemma and frequently use them without further reference.

\begin{lemma}
Let $1<q_0,q_1<\infty$, $1< p_0,p_1< \infty$ and $-4\leq s_0,s_1 \leq 4$.
\begin{enumerate}[label=(\roman*)]
    \item \label{item:lemSobEmb1} If $q_0\leq q_1$ and $s_0-\frac{d}{q_0}\geq s_1-\frac{d}{q_1}$ hold true, then 
\begin{align*}
    \bpsn{s_0,q_0} \hookrightarrow \bpsn{s_1,q_1}.
\end{align*}
\item \label{item:lemSobEmb2} Let $s_0,s_1 \neq \pm 4$. If $q_0\leq q_1$ , $s_0-\frac{d}{q_0}\geq s_1-\frac{d}{q_1}$, and $p_0\leq p_1$ hold true, then 
\begin{align*}
    {_\nu B^{ s_0 }_{q_0,p_0}(\sO)}\hookrightarrow {_\nu B^{ s_1 }_{q_1,p_1}(\sO)}.
\end{align*}
\end{enumerate}
\end{lemma}
Note that in \ref{item:lemSobEmb2}, the last condition is only relevant if we have equality in the second one.

\begin{proof}
We begin with the proof of \ref{item:lemSobEmb1}. First, let $0\leq s_0,s_1\leq 4$ be non-critical in the sense of $s_0 \notin \{1+1/q_0, 3+1/q_0 \}$ and $s_1 \notin \{1+1/q_1, 3+1/q_1 \}$. Since by \eqref{Bilaplace-concrete},  $  \bpsn{s_1,q_1}$, $  \bpsn{s_0,q_0}$ are subspaces of the classical Bessel potential spaces the statement follows if we can verify that the boundary conditions are compatible. Let $s_1>1+1/q_1$, then we have
\begin{align*}
s_0
\geq
s_1+d\left(\frac{1}{q_0}-\frac{1}{q_1}\right)
>
1+\frac{1}{q_1}
+d\left(\frac{1}{q_0}-\frac{1}{q_1}\right)
\geq
1+\frac{1}{q_0},
\end{align*}
which proves the statement for $\gamma_1$; the assertion for $\gamma_1\Delta$ can be checked in the same way.\\
If $0\leq s_0,s_1\leq 4$ are allowed to be critical we can argue by means of interpolation. To do so we restrict ourselves to the case $s_0\neq 4$ and $s_1\neq 0$, as the cases $s_0=4$ or $s_1=0$ can always be handled by making use of the already established embeddings
\begin{align*}
    \bpsn{4,q_0} \hookrightarrow \bpsn{4-\delta,\tilde{q}_0}, \quad \bpsn{\delta,\tilde{q}_1} \hookrightarrow \bpsn{0,q_1},
\end{align*}
with $\delta>0$ being sufficiently small and $\tilde{q}_0,\tilde{q}_1$ chosen appropriately. Now, it is always possible to choose $\varepsilon>0$ such that for 
$s_j^- \coloneq s_j-\varepsilon$, $ 
s_j^+ \coloneq s_j+\varepsilon$ with $j=0,1$ we have
\begin{align*}
0\leq s_j-\varepsilon<s_j<s_j+\varepsilon\leq 4,
\qquad j=0,1,
\end{align*}
where $s_j^-$ and $s_j^+$ are non-critical. Hence
\begin{align*}
\bpsn{s_0^-,q_0}
\hookrightarrow
\bpsn{s_1^-,q_1},
\qquad
\bpsn{s_0^+,q_0}
\hookrightarrow
\bpsn{s_1^+,q_1},
\end{align*}
and, as we have $ 
s_j=\frac12 s_j^-+\frac12 s_j^+$ for $ j=0,1$, 
the case for $s_0,s_1$ follows by complex interpolation. Finally, by duality the case $-4\leq s_0,s_1\leq 0$ can be reduced to the case $0\leq s_0,s_1\leq 4$, and the mixed case $-4\leq s_1<0<s_0\leq 4$ can be taken care of by introducing a suitable $r\in (1,\infty)$ with 
\begin{align*}
    \bpsn{s_0,q_0}
\hookrightarrow
\bpsn{0,r}
\hookrightarrow
\bpsn{s_1,q_1}.
\end{align*}
This concludes the proof of \ref{item:lemSobEmb1}. The proof of \ref{item:lemSobEmb2} is then a consequence of \ref{item:lemSobEmb1} and the definition of the spaces   ${_\nu B^{ s}_{q,p}(\sO)}$  by real interpolation. 
\end{proof}
The remainder of this section is devoted to Lemma~\ref{handle_noise}. The result is not new, but to the best of our knowledge there is no reference in the literature which applies to our situation. Thus, for the reader's convenience, we include a proof. In the abstract stochastic maximal $L^p$-regularity setting, regularity of the noise is quantified by means of $\gamma$-radonifying operators having values in a complex interpolation space; cf.\ Definition~\ref{def:linear:see:solution:strong}.  
For $\gamma$-radonifying operators with values in $L^q$-based spaces there  commonly is a more approachable characterisation available, which is presented in Lemma~\ref{handle_noise}. As this requires the $\ell^2$-valued 
analogue of the function spaces introduced above, we define
\begin{align*}
    \Bpsn{2,q}
    &\ce \{v\in H^{2,q}(\sO;\ell^2) \colon \gamma_1 v=0 \}, ~~
    \Bpsn{0,q}
    \ce L^q(\sO;\ell^2),~~
    \Bpsn{-2,q}
    \ce \big(\Bpsn{2,q'}\big)^*,\\
    \Bpsn{s,q}
    &\ce [\Bpsn{0,q},\Bpsn{2,q}]_{s/2},\qquad s\in(0,2),\\
    \Bpsn{s,q}
    &\ce [\Bpsn{-2,q},\Bpsn{0,q}]_{(s+2)/2},
    \qquad s\in(-2,0).
\end{align*}
Note that this is exactly the construction employed in the scalar-valued case. In particular, we can identify the family $(\Bpsn{s,q})_{-2\leq s\leq 2}$ with the scale of the $\ell^2$-valued Neumann Laplacian; see the proof of Lemma~\ref{handle_noise}.  Hence, $\Bpsn{s,q}$ can be densely embedded into $\Bpsn{s',q}$ for $-2\leq s'< s<2$.
The following lemma seems not to be available in the literature, but the Hilbert-space structure of $\ell^2$ allows for an elementary proof. For statements of this type without boundary conditions see~\cite[Appendix Theorem~1.3.6]{amannLinearQuasilinearParabolic2019}.

\begin{lemma}\label{lemma_density}
Let $1<q<\infty$ and let $(e_k)_{k\geq 1}$ be an orthonormal basis of
$\ell^2$. Then the space
\begin{align}\label{elementaryFunction}
    \bigg\{\sum_{k=1}^N f_k\otimes e_k \colon N\ge 1,\;f_k\in \nH^{2,q}(\sO) \bigg\}
\end{align}
is dense in $\nH^{2,q}(\sO;\ell^2)$. Here, $(f_k\otimes e_k)(x)\ce f_k(x)e_k$ for $x\in \sO$.
\end{lemma}

\begin{proof}
We denote by $P_N$ the orthogonal projection of $\ell^2$ onto $\operatorname{span}\{e_1,\dots,e_N\}$, and set
\begin{align*}
    f_N\coloneqq P_N f=
    \sum_{k=1}^N
    \langle f,e_k\rangle_{\ell^2}\otimes e_k
\end{align*}
for $  f\in \nH^{2,q}(\sO;\ell^2)$.
Since $P_N$ is a bounded linear operator, it commutes with weak derivatives, and therefore we have $ 
    \partial^\alpha f_N
    =
    P_N \partial^\alpha f$ 
for all $\alpha \in \mathbb{N}_0^d$ with $|\alpha|\leq 2$. This proves $f_N\in H^{2,q}(\sO;\ell^2)$, as well as  $\langle f,e_k\rangle_{\ell^2}\otimes e_k=P_k f-P_{k-1}f\in H^{2,q}(\sO;\ell^2)$, with the obvious modification for $k=1$.
In particular, 
$\partial_\nu \langle f,e_k\rangle_{\ell^2}=\langle \partial_\nu f,e_k\rangle_{\ell^2}$ and as the trace commutes with the functional $\langle \cdot, e_k \rangle_{\ell^2}$, we have
$\langle f,e_k\rangle_{\ell^2}\in \nH^{2,q}(\sO)$.
Moreover, $P_Nx\to x$ in $\ell^2$ for every $x\in\ell^2$ and
$  \|(I-P_N)\partial^\alpha f(x)\|^q_{\ell^2}
    \leq
    \|\partial^\alpha f(x)\|^q_{\ell^2}$.
Consequently, dominated convergence yields
   $  \partial^\alpha f_N
    \to
    \partial^\alpha f$ in $ 
    L^q(\sO;\ell^2)$,
and we can infer $f_N \to f $ in $H^{2,q}(\sO;\ell^2)$, proving the claim.
\end{proof}

\begin{lemma}\label{handle_noise}
Let $1<q<\infty$ and $-2\leq s\leq 2$. For $f\in \Bpsn{s,q}$, the mapping
\begin{align}\label{linear_mapping}
    \bigl[h\mapsto \langle h,f(\cdot)\rangle_{\ell^2}\bigr] \in \sL(\ell^2,\bpsn{s,q})
\end{align}
induces a conjugate-linear isomorphism from $\Bpsn{s,q}$ onto  $\gamma(\ell^2, \bpsn{s,q}) $.
\end{lemma}    

\begin{proof}
To give the proof we employ the interpolation-extrapolation scales of the Neumann Laplacian and its vector-valued realisation on $L^q(\sO;\ell^2)$. To be more precise, for some sufficiently large $\lambda_0>0$ we consider
\begin{alignat*}{3}
    B_q &\colon \nH^{2,q}(\sO) \to L^q(\sO),
    \qquad &B_q u &\coloneqq (\lambda_0-\Delta)u,
    \\
    \mathbf{B}_q &\colon \nH^{2,q}(\sO;\ell^2) \to L^q(\sO;\ell^2),
    \qquad &\mathbf{B}_q v &\coloneqq (\lambda_0-\Delta_{\ell^2})v,
\end{alignat*}
where $\Delta_{\ell^2}$ denotes the Laplacian acting on functions with values in $\ell^2$. According to~\cite[Theorem~2.3]{denkNewThoughtsOld2004} both operators have a bounded $H^\infty$-calculus and hence generate the corresponding interpolation-extrapolation scales.
Furthermore, the adjoint of $\mathbf B_q$ is given by the corresponding
realisation $\mathbf B_{q'}$ on $L^{q'}(\sO;\ell^2)$. We can therefore use~\cite[Theorem~1.5.4]{amannLinearQuasilinearParabolic1995} to obtain for every $-2\leq s \leq 2 $ the isomorphisms
\begin{align*}
    B^{s/2}_{q}
    &\colon \nH^{s,q}(\sO)
    \to L^q(\sO),
    \qquad
    \mathbf{B}_q^{s/2}
    \colon \nH^{s,q}(\sO;\ell^2)
    \to L^q(\sO;\ell^2).
\end{align*}
Note that the operators $B^{s/2}_q$ and $\mathbf{B}_q^{s/2}$ should be thought of as lifting operators and they agree on $\nH^{s,q}(\sO)\cap L^q(\sO)$ and $\nH^{s,q}(\sO;\ell^2)\cap L^q(\sO;\ell^2)$, respectively, with the classical fractional powers defined by Dunford calculus.
To continue, let $f= \sum_{k=1}^N f_k\otimes e_k$ be a function as given in \eqref{elementaryFunction}. Uniqueness of the corresponding
resolvent problem gives
 $(\lambda+\mathbf B_q)^{-1}(f_k\otimes e_k)
    =(\lambda+B_q)^{-1}f_k\otimes e_k$
for all sufficiently large $\lambda$ in a suitable sector, so that by construction of the Dunford calculus we get
\begin{align*}
    \mathbf{B}_q^{s/2} f= \sum_{k=1}^N B_q^{s/2} f_k\otimes e_k.
\end{align*}
Now, fix $-2\leq s \leq 2$ and define $T_f\in\sL(\ell^2,\bpsn{s,q})$ by
\begin{align*}
    T_fh
    \coloneqq
    \langle h,f(\cdot)\rangle_{\ell^2}
    =
    \sum_{k=1}^N
    f_k\langle h,e_k\rangle_{\ell^2}.
\end{align*}
Using the ideal property of $\gamma$-radonifying operators (see~\cite[Theorem~9.1.10]{hytonenAnalysisBanachSpaces2017}) in the first step we can estimate
\begin{align*}
    \|T_f\|_{\gamma(\ell^2,\nH^{s,q}(\sO))}
    &\eqsim
    \|B_q^{s/2}T_f\|_{\gamma(\ell^2,L^q(\sO))}
    =
    \bigg\|
        h\mapsto
        \sum_{k=1}^N
        B_q^{s/2}f_k
        \langle h,e_k\rangle_{\ell^2}
    \bigg\|_{\gamma(\ell^2,L^q(\sO))}
    \eqsim
    \bigg\|
        \sum_{k=1}^N
        B_q^{s/2}f_k\otimes e_k
    \bigg\|_{L^q(\sO;\ell^2)}
    \\
    &=
    \|\mathbf B_q^{s/2}f\|_{L^q(\sO;\ell^2)}
    \eqsim
    \|f\|_{\nH^{s,q}(\sO;\ell^2)}.
\end{align*}
Here, the third step is a consequence of the $\gamma$-Fubini
isomorphism as stated in~\cite[Theorem~9.4.8]{hytonenAnalysisBanachSpaces2017}.
Hence, by the preceding density statement in Lemma~\ref{lemma_density}, the mapping \eqref{linear_mapping} is well-defined and induces a conjugate-linear isomorphism from $\Bpsn{s,q}$ onto a closed subspace  of $\gamma(\ell^2, \bpsn{s,q}) $.

It remains to verify surjectivity. Let $T\in \gamma(\ell^2,\nH^{s,q}(\sO))$ and set $S\coloneqq B_q^{s/2}T$.
Due to the ideal property of $\gamma$-radonifying operators, we have $S\in\gamma(\ell^2,L^q(\sO))$, and we choose a sequence $(S_n)_{n\geq 1}$ of finite-rank operators given by
$ S_n h
    =
    \sum_{k=1}^{N_n}
    g_{k,n} \langle h, e_{k,n}\rangle_{\ell^2}$
where $g_{k,n}\in L^q(\sO)$ and
$e_{1,n},\dots,e_{N_n,n}\in\ell^2$ are orthonormal, such that $S_n\to S$ 
   in $\gamma(\ell^2,L^q(\sO))$.
Now, define $f_n\coloneq \mathbf{B}_q^{-s/2}
    \left(
        \sum_{k=1}^{N_n}g_{k,n}\otimes e_{k,n}
    \right)\in \nH^{s,q}(\sO;\ell^2)$ for every $n\in \N$. Then
\begin{align*}
    T_{f_n}h
    =
    \sum_{k=1}^{N_n}
    B_q^{-s/2}g_{k,n}
    \langle h,e_{k,n}\rangle_{\ell^2}
    =
    B_q^{-s/2}S_nh,
\end{align*}
and hence $T_{f_n}
    =
    B_q^{-s/2}S_n$.
Finally, another application of the ideal property yields
\begin{align*}
    T_{f_n}
    =
    B_q^{-s/2}S_n
    \to
    B_q^{-s/2}S
    =
    B_q^{-s/2}B_q^{s/2}T
    =
    T \quad \text{ in } \gamma(\ell^2,\nH^{s,q}(\sO)).
\end{align*}
Since we have already established closedness of the range, the linear mapping in \eqref{linear_mapping} is indeed surjective.
\end{proof}

\section{Local well-posedness via \texorpdfstring{$L^2$}{\textit{L2}}-theory}\label{sec:variational:lwp}
In this section we provide a local well-posedness theory by means of the critical variational setting established in~\cite{agrestiCriticalVariationalSetting2024a}, which was later extended in~\cite{BechtelGermVeraar2026}. 
To make this framework applicable we formulate the stochastic Cahn--Hilliard equation \eqref{eq:intro:SPDE} as an abstract variational problem of the form
\begin{equation}\tag{AVP}\label{eq:AVP}
    \begin{cases}
        \dd u(t) + Au(t)\dt  &= F(u(t)) \dt + G(t,u(t))\,\dd W(t), \\
        u(0) &=u_0.
    \end{cases}
\end{equation}
Here, $A$, $F$ and $G$ represent suitable representatives of the Bi-Laplacian, the nonlinear term $\Delta(u^3-u)$ and the gradient-dependent noise in the chosen setting. The setting in turn is given by a Gelfand triple $V\seq H \seq V^\ast$ of function spaces, where the choice of the Gelfand triple determines the notion of the solution concept. Thus, we can avoid the interpolation-extrapolation scale techniques introduced in Subsection~\ref{subsec:function_spaces} and do not rely on the concept of the $H^\infty$-calculus.

In the first part of this section, Subsection~\ref{weak_setting}, we quickly recall global well-posedness for \eqref{eq:AVP} in the weak setting $V\coloneq \nH^2(\sO)$ $ H\coloneq L^2(\sO)$ as an application of~\cite[Thm~5.2]{agrestiCriticalVariationalSetting2024a} to the choice $f(y) = y^3-y$. The growth exponent $\rho = 2$ and the condition $\rho \leq 4/d$ identify $d=2$ as the critical dimension for this setting. Therefore, we pass to the weakened setting $V\coloneq\nH^3(\sO)$, $ H\coloneq H^1(\sO)$ in Subsection~\ref{weakened_setting} and show local well-posedness in dimensions $3$ and $4$ and a relaxed blow-up criterion there, which is then used in Section~\ref{sec:globalWP} to show global well-posedness in dimensions $3$ and $4$.

Since the drift part is fixed by our choice of $f(x) = x^3-x$, the two dimension-dependent sets of assumptions, i.e.\ Assumptions~\ref{assumption_weak_setting} and ~\ref{assumption_weakened_setting}, only concern the noise coefficient $\psi$. 

\subsection{Weak setting: \texorpdfstring{$V=\nH^2(\sO), H=L^2(\sO)$}{\textit{V= HN2(O), H=L2(O)}}}\label{weak_setting}

A more general account of this (with more general drift term $f$) can be found in~\cite[Subsection~5.1]{agrestiCriticalVariationalSetting2024a}. We recall it here for the sake of completeness. The assumptions on the noise coefficient are as follows:

\begin{assumption}\label{assumption_weak_setting}
Let $d\in \{1,2 \}$. Suppose that $\psi=(\psi_n)_{n\geq 1}:[0,\infty)\times \Omega\times\sO\times \R^{1+d}\to \ell^2$ is $\sP \otimes \sB(\sO)\otimes \sB(\R^{1+d})$-measurable and there is a constant $L\geq 0$ such that a.e.\ on $[0,\infty)\times \Omega\times \sO$ and for all $y,y'\in \R$ and $z,z'\in \R^d$ we have
\begin{align*}
    \| \psi(\cdot,y,z)-\psi(\cdot,y',z') \|_{\ell^2} &\leq L(|y-y'|+|z-z'|),\\
   \| \psi(\cdot,y,z)\|_{\ell^2} &\leq  L (1+|y|+|z|).
\end{align*}
\end{assumption}

Then $A\in \sL(V,V^\ast)$ and $F:V\to V^\ast$ are defined by
\[ \langle Au,v \rangle=(\Delta u,\Delta v)_{L^2} \text{ and } 
\langle F(u),v \rangle=-(- u^3+u, \Delta v)_{L^2}, \quad u,v\in V,\]
and further, $G:[0,\infty)\times \Omega \times V \to \sL_2(\ell^2,H)$ is defined by 
\[  (G_n(t,u))(x)=\psi_n(t,x,u(x),\nabla u(x)),\quad n\geq 1. \]

\begin{theorem}[Global well-posedness in the weak setting]\label{thm:GWP:weak:setting}
Suppose that Assumption~\ref{assumption_weak_setting} holds. Let $u_0\in L^0_{\F_0}(\Omega;L^2(\sO))$, then \eqref{eq:AVP} has a unique global solution with

\begin{equation}
    u\in C([0,\infty);L^2(\sO))\cap L^2_{\loc}([0,\infty);\nH^2(\sO)) \quad \text{a.s.}
\end{equation}
If, in addition, $u_0\in  L_{\F_0}^2(\Omega;L^2(\sO))$, then for every $T>0$  there exists a constant $C_T > 0$ independent of $u_0$ such that 
\begin{equation}
    \E \Vert u \Vert_{C([0,T]; L^2(\sO))}^2 + \E \Vert u \Vert_{L^2(0, T; H^2(\sO))}^2 \leq C_T \big(1 + \E \Vert u_0 \Vert_{L^2(\sO)}^2\big).
\end{equation}
\end{theorem}
\begin{proof}
This is a consequence of~\cite[Subsection~5.1]{agrestiCriticalVariationalSetting2024a}. In the notation of the paper, we consider the case $f(y) = y(y^2-1)$ and $\rho = 2$. 
\end{proof}

\subsection{Weakened setting: \texorpdfstring{$V= \nH^3(\sO), H=H^1(\sO)$}{\textit{V=HN3(O), H=H1(O)}}}\label{weakened_setting}

Next we study \eqref{eq:AVP} by means of the critical variational setting with $H\coloneq H^1(\sO)$ and $V\coloneq \nH^3(\sO)$. This case corresponds to a variational formulation of the $(H^1(\sO))^{\ast_2}$-realisation of the Bi-Laplacian, as we can identify $V^\ast $ with $(H^1(\sO))^{\ast_2}$ where the superscript $\ast_2$ denotes the duality with respect to the $L^2$-pairing. To verify this, we first observe
\begin{align}\label{dense_embedding}
 H^1(\sO)\overset{d}{\hookrightarrow} V^\ast, \quad  H^1(\sO)\overset{d}{\hookrightarrow} (H^1(\sO))^{\ast_2}.   
\end{align}   
Now let $h\in H^1(\sO) \seq V^\ast$ and $v\in \nH^3(\sO)$, we have
\begin{align*}
    \langle h , v\rangle_{V^\ast \times V}=( h , v)_{H^1}=( \nabla h , \nabla v )_{L^2} +( h ,v )_{L^2} =( h , (-\Delta+1) v )_{L^2} 
\end{align*}
and thus 
\begin{align*}
   | \langle h , v\rangle_{V^\ast \times V} | \leq \|h \|_{(H^1)^{\ast_2}} \cdot  \|  (-\Delta+1)v \|_{H^1} \leq \|h \|_{(H^1)^{\ast_2}} \cdot  \| v \|_{V}.
\end{align*}
Conversely, let $h\in H^1(\sO) \seq (H^1(\sO))^{\ast_2}$, $v\in H^1(\sO)$ and note that by elliptic regularity, $-\Delta+1$ yields an isomorphism from $\nH^3(\sO)$ to $ H^1(\sO)$. Therefore, there exists $\tilde{v}\in V$ with $(-\Delta+1)\tilde{v}=v$ and we get
\begin{align*}
    \langle h , v\rangle_{(H^1)^{\ast_2} \times  H^1}=( h,v )_{L^2}=( h, (-\Delta+1)\tilde{v}  )_{L^2}=( \nabla h , \nabla \tilde{v} )_{L^2} +( h ,\tilde{v} )_{L^2} = ( h, \tilde{v} )_{H^1}.
\end{align*}
Finally, employing the isomorphism once again gives 
\[ | \langle h , v\rangle_{(H^1)^{\ast_2} \times  H^1} | \leq \|h \|_{V^\ast}\cdot \| \tilde{v} \|_{V} \lesssim \|h \|_{V^\ast}\cdot \| v \|_{H^1},\]
and we can conclude $\| h\|_{V^\ast}\eqsim\|h \|_{(H^1)^{\ast_2}}$ for all $h\in H^1(\sO)$ so that by virtue of \eqref{dense_embedding} we can identify $V^\ast$
with $(H^1(\sO))^{\ast_2}$.
Moreover, we have $V^\ast \simeq (H^1(\sO))^{\ast_2}$ by means of
\begin{align}
\label{eq:dualIsomorphism}
   V^\ast \ni \Lambda &\mapsto \Lambda\left(\,(-\Delta+1)^{-1} \cdot\,\right) \in (H^1(\sO))^{\ast_2},\nonumber\\  
  (H^1(\sO))^{\ast_2} \ni \Lambda  &\mapsto \Lambda\left(\,(-\Delta+1) \cdot\,\right) \in V^\ast.
\end{align}

We define $A \in \sL(V,V^\ast)$ and $F:V\to V^\ast$ by
\begin{equation}\label{def_weakened_Bilaplacian}
    \langle Au,v \rangle=(\Delta u, \Delta v)_{H^1} \text{ and } 
\langle F(u),v \rangle=-(- u^3+u, \Delta v)_{H^1}, \quad \forall u, v \in V.
\end{equation}

In particular, it follows that for all $u, v \in V$

\begin{equation}
    \langle Au, v\rangle_{(H^1)^{\ast_2}, H^1} = -(\grad \Delta u,\grad v)_{L^2} \quad \text{and} \quad \langle F(u),v\rangle_{(H^1)^{\ast_2},H^1} = (-\grad(u^3 - u) ,\grad v)_{L^2}.
\end{equation}

The definition of $G$ can essentially be done in the same way as in the weak setting, but the weakened setting necessitates
higher regularity of the noise. It is therefore that we adapt the assumptions on the $(\psi_n)_{n\geq 1}$; compare with Assumption~\ref{assumption_weak_setting}.

\begin{assumption}\label{assumption_weakened_setting}
  Let $d\in \{3,4 \}$. Suppose that 
  \begin{equation*}
      \psi=(\psi_n)_{n\geq 1}:[0,\infty)\times \Omega\times\sO\times \R^{1+d}\to \ell^2,\quad \psi(t,\omega,x,y,z) \ce \psi^0(t,\omega,x,y)+\psi^1(t,\omega,x) \cdot z,
  \end{equation*}
  where $\cdot$ denotes the usual scalar product in $\R^d$ and the functions
 \begin{align*}
     &\psi^0=(\psi^0_n)_{n\geq 1}:[0,\infty)\times \Omega\times\sO\times \R\to \ell^2,\quad
     \psi^1=(\psi^1_n)_{n\geq 1}:[0,\infty)\times \Omega\times\sO\to \ell^2(\R^d)
 \end{align*}
 are $\sP \otimes \sB(\sO)\otimes \sB(\R)$-measurable and $\sP \otimes \sB(\sO)$-measurable, respectively.
 Moreover, we assume that 
\begin{enumerate}
    \item the mappings  $(x,y)\mapsto \psi^0(\cdot,x,y)$ and $x\mapsto \psi^1(\cdot,x)$  are in $C^1(\sO\times \R;\ell^2)$ and $C^1(\sO;\ell^2(\R^d))$  a.e.\ on $[0,\infty)\times \Omega$, respectively;
    \item  there exists a constant $L\geq0$ such that a.e.\ on
    $[0, \infty)\times\Omega\times\sO$, for all $y,y'\in\R$, we have
    \begin{align}
        \|\nabla_x[
            \psi^0(\cdot,y)-\psi^0(\cdot,y')
        ]\|_{\ell^2(\R^d)}
        &\leq L(1+|y|+|y'|)|y-y'|, \label{eq:psi0-gradx-local-lipschitz}
        \\
                \|\partial_y[
            \psi^0(\cdot,y)-\psi^0(\cdot,y')
        ]\|_{\ell^2}
        &\leq L|y-y'|, \label{eq:psi0-dy-lipschitz}
        \\
        \|\psi^0(\cdot,y)\|_{\ell^2}+ \|\nabla_x\psi^0(\cdot,y)\|_{\ell^2(\R^d)}
        &\leq L(1+|y|), \label{eq:psi0-linear-growth}
        \\
        \|\psi^1(\cdot)\|_{\ell^2(\R^d)}
        +
        \|\nabla_x\psi^1(\cdot)\|_{\ell^2(\R^{d\times d})}
        &\leq L. \label{eq:psi1-uniform-bound}
    \end{align}
\end{enumerate}
\end{assumption}

The following estimates follow immediately from Assumption~\ref{assumption_weakened_setting} and the fundamental theorem of calculus.

\begin{corollary} \label{cor:weakened_noise_consequences}
    Under Assumption~\ref{assumption_weakened_setting}, the following holds. There exists a constant $C_L > 0$ such that almost everywhere on $[0, \infty) \times \Omega \times \sO$ and for all $y, y' \in \R$ and $z, z' \in \R^d$ 

    \begin{align}
        \|\partial_y\psi^0(\cdot,y)\|_{\ell^2} &\leq C_L(1+|y|), \label{eq:corNoiseDerivativeGrowth}\\
        \|\psi^0(\cdot,y)-\psi^0(\cdot,y')\|_{\ell^2} & \leq C_L(1+|y|+|y'|)|y-y'|, \label{eq:corNoiseLocalLipschitz}\\
        \|\psi(\cdot,y,z)-\psi(\cdot,y',z')\|_{\ell^2} & \leq C_L\Big((1+|y|+|y'|)|y-y'|+|z-z'|\Big). \label{eq:corFullNoiseLocalLipschitz}
    \end{align}
\end{corollary}

Now, $G:[0,\infty)\times \Omega \times V \to \sL_2(\ell^2,H)$ is defined by 
\begin{equation}\label{eq:def_G}
    (G_n(t,u))(x)=\psi_n(t,x,u(x),\nabla u(x)),\quad n\ge 1. 
\end{equation}

Recall at this point the notation for the complex interpolation spaces $V_{\beta} \ce  [V^{\ast}, V]_{\beta}$ for $\beta \in [0,1]$ and the shorthand $\Vert \cdot \Vert_{\beta} \ce  \Vert \cdot \Vert_{V_{\beta}}$.

\begin{lemma} \label{lem:weakened_estimatesAFG}
Under Assumption~\ref{assumption_weakened_setting}, the following hold:
\begin{enumerate}[label=(\alph*)]
\item\label{lemItem:weakenedEstimatesA} There exists a constant $C > 0$ such that for all $u\in V$ we have
\begin{align}
     \langle Au,u \rangle= \| \Delta u\|^2_{H^1} \geq \frac{1}{C} \|u \|^2_{V}-\|u \|^2_H.
\end{align}
\item\label{lemItem:weakenedEstimatesF} For all $u,v\in V$ we have 
\begin{align}\label{eq:FestimatesLocWP}
    \|F(u)-F(v) \|_{V^\ast} &\lesssim \sum_{j=1}^3 \big(1+\|u \|_{\beta_j}^{\rho_j}+\|v \|^{\rho_j} _{\beta_j}\big)\| u-v\|_{\beta_j},\\
     \|F(u) \|_{V^\ast} &\lesssim \sum_{j=1}^3 \big(1+\|u \|_{\beta_j}^{\rho_j+1}\big),
\end{align}

where $(\beta_1, \rho_1) \coloneq (\frac{7}{12}, 0)$, $(\beta_2, \rho_2) \coloneq  ( \frac{7}{12}, 2 )$ and $(\beta_3, \rho_3) = ( \frac{1}{2}+\frac{d-2}{12}, 2 )$.

\item\label{lemItem:weakenedEstimatesG} 
The mapping $G:[0,\infty)\times \Omega \times V \to \sL_2(\ell^2,H)$ is $\sP\otimes \sB(V)$-measurable and almost everywhere on $[0, \infty) \times \Omega$ and for all $u, v \in V$ we have

\begin{align}\label{eq:GestimatesLocWP}
    \| G(t,u)-G(t,v) \|_{\sL_2(\ell^2,H)} &\lesssim \sum_{j=4}^6 \big(1+\|u \|_{\beta_j}^{\rho_j}+\|v \|^{\rho_j} _{\beta_j}\big)\| u-v\|_{\beta_j},\\
     \| G(t,u) \|_{\sL_2(\ell^2,H)} &\lesssim \sum_{j=4}^6 \big(1+\|u \|_{\beta_j}^{\rho_j+1}\big),
\end{align}

where $(\beta_4,\rho_4) \coloneq  (\frac{7}{12},1)$, $(\beta_5,\rho_5) = (\frac{d+6}{16},1)$, and $(\beta_6,\rho_6) = (\frac{3}{4},0)$.
\end{enumerate}
\end{lemma}
\begin{proof}
\textit{ (a)} Let $u,v\in V$, we start with the estimate for $A$. As mentioned above, $-\Delta+1$ yields an isomorphism from $\nH^3(\sO)$ to $ H^1(\sO)$, therefore we have
    \begin{align*}
        \|u \|^2_{\nH^3}\lesssim \| (-\Delta+1)u\|^2_{H^1}\lesssim \| \Delta u\|^2_{H^1}+\| u\|^2_{H^1},
    \end{align*}
    and we can conclude that there exists a constant $C > 0$ such that 
    \begin{align*}
       \langle Au,u \rangle=(\Delta u,\Delta u)_{H^1}= \| \Delta u\|^2_{H^1} \geq \frac{1}{C} \|u \|^2_{V}-\|u \|^2_H.
    \end{align*}
    
    \textit{(b)} Let $\theta \in (0,1)$, we have $V_{\frac{1+\theta}{2}}=[H^1(\sO),\nH^3(\sO)]_{\theta}\hookrightarrow H^{(1-\theta)+3\theta}(\sO)$ and hence
\begin{align}\label{tau}
   V_{\tau} \hookrightarrow H^{4\tau-1}(\sO)
\end{align}
for $\tau \in (1/2,1)$, which we will employ frequently in the sequel. To estimate $F$ we first observe
\begin{align}
    \|F(u)-F(v) \|_{V^\ast}   &\lesssim \|-u^3+u+v^3-v\|_{H^1} 
 \lesssim \| u-v \|_{H^1} +\|u^3-v^3 \|_{H^1} \nonumber\\ 
  &\lesssim \| u-v \|_{H^1} + \| u^3-v^3 \|_{L^2}+  \|\nabla ( u^3-v^3 )\|_{L^2}.\label{eq:F_to_be_bound}
\end{align}    
By \eqref{tau}, the first term can be upper bounded by $\| u-v \|_{\frac{7}{12}}$ and thus we set $\beta_1 \ce \frac{7}{12}$ and $\rho_1 \ce 0$. The second term is treatable by the difference-of-powers identity $u^3-v^3=(u-v)(u^2+uv+v^2)$ and Hölder's inequality, giving
\begin{align*}
    \|u^3-v^3 \|_{L^2}& \lesssim (1+\|u \|^2_{L^6}+\|v \|^2_{L^6} ) \|u-v \|_{L^6}
    \lesssim (1+\|u \|^2_{\beta_2}+\|v \|^2_{\beta_2} ) \|u-v \|_{\beta_2},
\end{align*}
where by \eqref{tau} the condition on the Sobolev embedding is given by $4\beta_2-1-\frac{d}{2} \geq -\frac{d}{6}$ and thus $\frac{d+3}{12} \leq \beta_2$. The dimension-independent choice $\beta_2 \ce \frac{7}{12}$ is consistent with $\beta_2 \in (\frac{1}{2}, 1)$.
In order to estimate the final term in \eqref{eq:F_to_be_bound}, applying the gradient to the difference-of-powers identity gives
\[ \nabla (u^3-v^3)=\nabla (u-v) \cdot (u^2+uv+v^2)+ (u-v)\cdot \nabla (u^2+uv+v^2)\]
and we consider both terms separately. We will start with the second one. Observe that
\[\nabla (u^2+uv+v^2)=2u\nabla u+u\nabla v +v\nabla u+2v\nabla v,\]
and employ the generalized Hölder inequality on each of the four resulting summands to obtain
\begin{align}\label{generalized_hölder}
\|  (u-v)\cdot \nabla (u^2+uv+v^2) \|_{L^2}
&\leq \|u-v \|_{L^{r_1}} \cdot \|2u \|_{L^{r_1}}  \cdot \| \nabla u \|_{L^{r_2}}  + \|u-v \|_{L^{r_1}} \cdot \|u \|_{L^{r_1}}  \cdot \| \nabla v\|_{L^{r_2}}  \nonumber\\
  &\phantom{\le }+ \|u-v \|_{L^{r_1}} \cdot \|v \|_{L^{r_1}}  \cdot \| \nabla u\|_{L^{r_2}} + \|u-v \|_{L^{r_1}} \cdot \|2v \|_{L^{r_1}}  \cdot \| \nabla v \|_{L^{r_2}},
\end{align}
where $\frac{1}{2}=\frac{2}{r_1}+\frac{1}{r_2}$ and $r_1,r_2\in [2,\infty)$. Since we want the seminorms $\|\cdot \|_{L^{r_1}}, \| \nabla \cdot \|_{L^{r_2}}$  to admit the same scaling we further get the condition
\[ \frac{-d}{r_1}=1-\frac{d}{r_2} \quad \implies \quad r_1=\frac{dr_2}{d-r_2}.\]
Inserting this in the first condition gives
\[ \frac{1}{2}=2\left( \frac{d-r_2}{dr_2}\right)+\frac{1}{r_2} \quad \implies \quad  r_2=\frac{6d}{d+4}.\]
Next, to verify $V_{\beta_3}\hookrightarrow H^{1,r_2}(\sO)$ 
and  $V_{\beta_3}\hookrightarrow L^{r_1}(\sO)$ for some $\beta_3\in (1/2,1)$ we use \eqref{tau} to get
\[ 4\beta_3-1-\frac{d}{2}=1-\frac{d}{r_2}=1-\frac{d+4}{6} \quad  \implies \quad \beta_3=\frac{1}{2}+\frac{d-2}{12}. \]
Hence, applying this to \eqref{generalized_hölder} we have
\begin{align*}
 \begin{aligned}
\|  (u-v)\cdot \nabla (u^2+uv+v^2) \|_{L^2} &\lesssim \|u-v \|_{\beta_3} \cdot \|u \|_{\beta_3}^2 + \|u-v \|_{\beta_3} \cdot \|u \|_{\beta_3}  \cdot \| v\|_{\beta_3}  \\
  &\phantom{\lesssim }+ \|u-v \|_{\beta_3} \cdot \|v \|_{\beta_3}  \cdot \|  u\|_{\beta_3}  + \|u-v \|_{\beta_3} \cdot \|v \|_{\beta_3}^2\\
  &\lesssim  (\|u \|_{\beta_3}^2+\|v \|_{\beta_3}^2) \|u-v \|_{\beta_3},
\end{aligned}   
\end{align*}
where in the last step we used Young's inequality. The first term can be dealt with in a similar manner. Let $r_1,r_2,\beta_3$ be as above. Then we obtain
\begin{align*}
 \begin{aligned}
\|  \nabla (u-v) \cdot (u^2+uv+v^2) \|_{L^2} &\leq \|\nabla (u-v) \|_{L^{r_2}} \cdot \|u \|_{L^{r_1}}^2 + 2 \|\nabla (u-v) \|_{L^{r_2}}  \cdot \|u \|_{L^{r_1}} \cdot \| v\|_{L^{r_1}} \\
  &\phantom{\le }+ \|\nabla (u-v) \|_{L^{r_2}}  \cdot \|v \|_{L^{r_1}}^2 \\
  &\lesssim (\|u \|_{\beta_3}^2+\|v \|_{\beta_3}^2) \|u-v \|_{\beta_3}.
\end{aligned}   
\end{align*}
The growth estimate is now an immediate consequence of the locally Lipschitz estimate since $F(0)=0$.

\textit{(c)} We continue with the estimate for $G$ and illustrate how to obtain the Lipschitz estimates. The growth estimates can be checked as in part (b),
since $\psi(t,\cdot,0,0)=\psi^0(t,\cdot,0)$ and Assumption~\ref{assumption_weakened_setting} yields
\[
\|\psi^0(t,\cdot,0)\|_{H^1(\sO;\ell^2)}\lesssim 1.
\]
Regarding the Lipschitz estimate, for $t\geq 0$ and $u,v\in V$ we have
\begin{align*}
    \|G(t,u)-G(t,v) \|_{H^1(\sO;\ell^2)} \eqsim \| G(t,u)-G(t,v)\|_{L^2(\sO;\ell^2)}+ \sum_{j=1}^d\|\partial_j [ G(t,u)-G(t,v)]\|_{L^2(\sO;\ell^2)}.
\end{align*}
By \eqref{eq:corFullNoiseLocalLipschitz} in Corollary~\ref{cor:weakened_noise_consequences} we have
\begin{equation}
    \|\psi(\cdot,u,\nabla u)-\psi(\cdot,v,\nabla v)\|_{\ell^2} \lesssim
    (1+|u|+|v|)|u-v|+|\nabla(u-v)|.
\end{equation}
Taking $L^2(\sO)$-norms and applying Hölder's inequality yields
\begin{align*}
    \|\psi(\cdot,u,\nabla u)-\psi(\cdot,v,\nabla v)\|_{L^2(\sO;\ell^2)}
    & \lesssim \|u-v\|_{H^1} +(\|u\|_{L^4}+\|v\|_{L^4})\|u-v\|_{L^4}\\
    &\lesssim
    (1+\|u\|_{H^1}+\|v\|_{H^1})
    \|u-v\|_{H^1}, \\
    &\lesssim
    (1+\|u\|_{\beta_4}+\|v\|_{\beta_4})
    \|u-v\|_{\beta_4},
\end{align*}
where we have used that $V_{\frac{7}{12}} \hookrightarrow H^1(\sO)\hookrightarrow L^4(\sO)$ for $d\in\{3,4\}$ such that $\beta_4\coloneqq\frac{7}{12}$.

Let $1\leq j\leq d$, for the derivative term we can use the chain rule to write 
\begin{align*}
    \partial_j [&\psi(\cdot,u,\nabla u)-\psi(\cdot,v,\nabla v)]=\, \partial_{x_j}\psi(\cdot,u,\nabla u)-\partial_{x_j}\psi(\cdot,v,\nabla v)\\
    & + \partial_y \psi(\cdot,u,\nabla u)\partial_{x_j}u-\partial_y \psi(\cdot,v,\nabla v)\partial_{x_j}v
    + \nabla_z \psi(\cdot,u,\nabla u)\cdot \partial_{x_j}\nabla u-\nabla_z\psi(\cdot,v,\nabla v)\cdot\partial_{x_j} \nabla v,
\end{align*}
and by means of the affine form of $\psi$ from Assumption~\ref{assumption_weakened_setting} this can be further simplified to
\begin{align*}
    \partial_j [&\psi(\cdot,u,\nabla u)-\psi(\cdot,v,\nabla v)]=\, 
 \partial_{x_j}\psi^0(\cdot,u)-\partial_{x_j}\psi^0(\cdot,v)
     + \partial_{x_j}\psi^1(\cdot)\cdot \nabla(u-v)\\
    &\ + \partial_y\psi^0(\cdot,u)\partial_{x_j}u
    -\partial_y\psi^0(\cdot,v)\partial_{x_j}v
     + \psi^1(\cdot)\cdot \partial_{x_j}\nabla(u-v).
\end{align*}
We begin with the first term. By \eqref{eq:psi0-gradx-local-lipschitz} in Assumption~\ref{assumption_weakened_setting} we get the pointwise estimate
\begin{align*}
    \|\partial_{x_j}\psi^0(\cdot,u)
    -\partial_{x_j}\psi^0(\cdot,v)\|_{\ell^2}
    \lesssim (1+|u|+|v|)|u-v|,
\end{align*}
and taking $L^2(\sO)$-norms for fixed $t$ and applying Hölder's inequality yields again
\begin{align*}
    \|\partial_{x_j}\psi^0(\cdot,u)
    -\partial_{x_j}\psi^0(\cdot,v)\|_{L^2(\sO;\ell^2)}
    &\lesssim
    \|u-v\|_{L^2}
    +\bigl(\|u\|_{L^4}+\|v\|_{L^4}\bigr)
    \|u-v\|_{L^4}
    \\
    &\lesssim
    \bigl(1+\|u\|_{\beta_4}+\|v\|_{\beta_4}\bigr)
    \|u-v\|_{\beta_4}.
\end{align*}
We proceed in a similar manner with the second term involving $\psi^0$. By \eqref{eq:psi0-dy-lipschitz} in Assumption~\ref{assumption_weakened_setting} and \eqref{eq:corNoiseDerivativeGrowth} in Corollary~\ref{cor:weakened_noise_consequences}, we have pointwise 
\begin{align*}
     \|\partial_{y}\psi^0(\cdot,u)\partial_{x_j}u
     -\partial_{y}\psi^0(\cdot,v)\partial_{x_j}v \|_{\ell^2}
   &\leq  \| [\partial_{y}\psi^0(\cdot,u)-\partial_{y}\psi^0(\cdot,v)]
   \partial_{x_j}u \|_{\ell^2}
   + \|\partial_{y}\psi^0(\cdot,v)
   \partial_{x_j}(u-v) \|_{\ell^2}\\
   &\lesssim |u-v|\,|\partial_{x_j}u|
   +(1+|v|)|\partial_{x_j}(u-v)|,
\end{align*}
and further obtain
\begin{align*}
   \|\partial_{y}\psi^0(\cdot,u)\partial_{x_j}u
   -\partial_{y}\psi^0(\cdot,v)\partial_{x_j}v
   \|_{L^2(\sO;\ell^2)}
  &\lesssim
  \|u-v\|_{L^{r_1}}\|u\|_{H^{1,r_2}}
  +(1+\|v\|_{L^{r_1}})
  \|(u-v)\|_{H^{1,r_2}}\\
  &\lesssim
  (1+\|u\|_{\beta_5}+\|v\|_{\beta_5})
  \|u-v\|_{\beta_5}.
\end{align*}
Here $r_1,r_2\in [2,\infty)$ are determined by $\frac{1}{r_1}+\frac{1}{r_2}
    =\frac{1}{2}$ and $-\frac{d}{r_1}
    =1-\frac{d}{r_2}$ which yields 
$r_1=\frac{4d}{d-2}$, $r_2=\frac{4d}{d+2}$, and hence, by \eqref{tau} we choose
$\beta_5$ with
$ 
    4\beta_5-1-\frac{d}{2}
    =-\frac{d}{r_1}$ such that $ 
\beta_5=\frac{d+6}{16}$.
It remains to consider the terms corresponding to $\psi^1$. By \eqref{eq:psi1-uniform-bound} in Assumption~\ref{assumption_weakened_setting} we get
\begin{align*}
    \|\partial_{x_j}\psi^1(\cdot)\cdot\nabla(u-v)
    +\psi^1(\cdot)\cdot\partial_{x_j}\nabla(u-v)\|_{\ell^2} \lesssim |\nabla(u-v)|+|\partial_{x_j}\nabla(u-v)|,
\end{align*}
and finally by choosing $\beta_6=\frac34$ it follows
\begin{align*}
    \|\partial_{x_j}\psi^1(\cdot)\cdot\nabla(u-v)
    +\psi^1(\cdot)\cdot\partial_{x_j}\nabla(u-v)\|_{L^2(\sO;\ell^2)}\lesssim \|u-v\|_{H^2}
    \lesssim \|u-v\|_{\beta_6}.
\end{align*}
This proves the desired estimate for $G$, in particular we have verified that $G$ is an $\sL_2(\ell^2,H)$-valued mapping and we continue with the measurability statement. By~\cite[Example~9.1.16]{hytonenAnalysisBanachSpaces2017} and a limit argument this is equivalent to the strong $\sP\otimes \sB(V)$-measurability of the mapping 
\[ (t,\omega,u)\mapsto (G_n(t,u))=\psi_n(t,\omega,\cdot,u(\cdot),\nabla u(\cdot)) \in H^1(\sO) \quad (n\in \N). \]
As we have seen above, this mapping is continuous with respect to the third variable. Thus, if we can show that the mapping
\begin{align}\label{eq:mappingPsi}
    (t,\omega) \mapsto \psi_n(t,\omega,\cdot,u(\cdot),\nabla u(\cdot)) \in H^1(\sO) 
\end{align}
is strongly $\sP$-measurable for all $u\in V$, the claim follows by~\cite[Lemma~III.14]{castaingConvexAnalysisMeasurable1977}. Let $u\in V$ and $f\in L^2(\sO)$, by Assumption~\ref{assumption_weakened_setting} the real-valued function $(t,\omega,x)\mapsto \psi_n(t,\omega, x,u(x),\nabla u(x))$ is $\sP\otimes \sB(\sO)$-measurable and by Fubini's theorem 
\[ (t,\omega)\mapsto \int_\sO \psi_n(t,\omega,x,u(x),\nabla u(x))f(x) \dx\]
is $\sP$-measurable. Now, since $L^2(\sO) $ is dense in $(H^1(\sO))^{\ast_2}$ we can infer from the Pettis measurability theorem~\cite[Theorem~1.1.20]{hytonenAnalysisBanachSpaces2016} that the mapping in \eqref{eq:mappingPsi} is indeed strongly $\sP$-measurable.
\end{proof}

\begin{theorem}[Local well-posedness in the weakened setting] \label{thm:localWP_weakened}
Suppose that Assumption~\ref{assumption_weakened_setting} holds. Then for all $u_0\in L^0_{\F_0}(\Omega;H^1(\sO))$, there exists a (unique) maximal solution $(u,\sigma)$ to \eqref{eq:AVP} such that a.s.\ $u\in C([0,\sigma);H^1(\sO))\cap L^2_{\loc}([0,\sigma);\nH^3(\sO))$. Moreover, the following blow-up criterion holds
\begin{equation}
    \label{eq:blowUp_bare}
    \P \bigg(\sigma < \infty ,\sup\limits_{t\in [0,\sigma)} \| u(t) \|_{H^1(\sO)} + \int_0^\sigma \|u(t) \|^2_{H^3(\sO)} \dt<\infty  \bigg) =0.
\end{equation}

In the noncritical case $d=3$ the blow-up criterion can be relaxed to 
\begin{equation}
    \label{eq:blowUpSubcritd3}
    \P \bigg(\sigma < \infty ,\sup\limits_{t\in [0,\sigma)} \| u(t) \|_{H^1(\sO)} <\infty  \bigg) =0.
\end{equation}
\end{theorem}
\begin{proof}
     To obtain local well-posedness we apply~\cite[Theorem~3.3]{agrestiCriticalVariationalSetting2024a}, which amounts to checking~\cite[Assumption~3.1]{agrestiCriticalVariationalSetting2024a}, with $U = \ell^2$, drift equal to $A - F$ as in \eqref{def_weakened_Bilaplacian}, and noise coefficient equal to $G$ as in \eqref{eq:def_G}, and $f = g = 0$. The measurability of these maps as well as the coercivity condition for $A$ are shown in Lemma~\ref{lem:weakened_estimatesAFG}. Boundedness of $A$ follows by construction, while the local Lipschitz conditions for $F$ and $G$ are satisfied. Finally, Lemma~\ref{lem:weakened_estimatesAFG} shows that for $d=3$ each of the coefficients $(2\beta_j-1)(\rho_j+1)$ is strictly smaller than $1$ and for $d=4$ at most $1$. Therefore the cited theorem applies and supplies maximal existence, and the blow-up criterion \eqref{eq:blowUp_bare}. Due to strict subcriticality of the coefficients, the relaxed blow-up criterion \eqref{eq:blowUpSubcritd3} is provided by \cite[Thm.~6.2]{agrestiNonlinearSPDEsMaximal2025}, since the real interpolation spaces $V_{\beta,1}$ used for the subcriticality condition there are embedded in the complex interpolation spaces $V_{\beta}$ from Lemma~\ref{lem:weakened_estimatesAFG}.
\end{proof}    

\begin{remark}
    As far as local well-posedness is concerned, the term $\psi^1$ can be upgraded to depend on $u$ in an affine way; i.e.\ one can replace $\psi^1$ by $\psi^{1} + u \psi^{2}$. We do not pursue this further here, but refer to Section~\ref{section_rough_initial_data}.
\end{remark}

\section{Global well-posedness in the weakened setting}
\label{sec:globalWP}

\subsection{Overview and proof strategy}
\label{subsec:overviewglobalWP} We show global well-posedness in the weakened setting of Subsection
\ref{weakened_setting}, whose notation we adopt. More precisely, we consider the stochastic Cahn--Hilliard equation for $d\in\{3,4\}$
\begin{equation}\tag{$\text{SCH}_2$}
\label{eq:SCHtorus}
    \rmd u = [-\Delta^2u+\Delta(u^3-u)
    ]\dt + \sum_{n \geq 1} \psi_n(u,\grad u)\,\rmd W^n,\quad u(0)=u_0,
\end{equation}
in the variational setting $p=2$, $\kappa=0$, $X_1 = V = \nH^3(\sO)$, $H=H^1(\sO)$, and $X_0 = V^*$. We thus regard \eqref{eq:SCHtorus} as an instance of \eqref{eq:AVP}, with coefficients $A$, $F$, and $G$ defined in \eqref{def_weakened_Bilaplacian}, Assumption ~\ref{assumption_weakened_setting} and \eqref{eq:def_G}. Throughout this section, we assume that Assumption~\ref{assumption_weakened_setting} holds. By Theorem~\ref{thm:localWP_weakened}, this ensures that, for any $u_0\in L^0_{\F_0}(\Omega;H^1(\sO))$, there exists a (unique) maximal solution $(u,\sigma)$ to \eqref{eq:SCHtorus} such that a.s.\ $u\in C([0,\sigma);H^1(\sO))\cap L^2_{\loc}([0,\sigma);\nH^3(\sO))$.

\subsubsection{The subcritical dimension \texorpdfstring{$d=3$}{d=3}}
\label{subsec:overviewglobalWPd3}

 Global well-posedness of the deterministic
Cahn--Hilliard equation is a well-studied topic. In the subcritical case, i.e.\
dimensions $d\le 3$, global well-posedness is well-known, see e.g.
\cite{ElliottSongmu86}. Moreover, it is well-known that the Cahn--Hilliard
equation has an $H^{-1}$-gradient flow structure, which can be derived from the
energy functional
    \begin{equation*}
        E(u)= \int_{\sO}\frac{1}{2}|\grad u|^2 + \Phi(u) \dx,
    \end{equation*}
    whereby $\Phi$ denotes the potential. In our setting, the potential is $\Phi(u) =
\frac{1}{4}(u^2-1)^2$ since $\Phi'(u) = u^3-u$. Since by the gradient flow
structure $E(u(t))$ is non-increasing over time, it can be used to derive an
$H^1(\sO)$-energy estimate. In the
following we define the chemical potential $\mu$, which is the first variational
derivative of $E$.
    \begin{definition}\label{def:chemPotential}
    The \emph{chemical potential} $\mu$ associated with the solution $u$ of
\eqref{eq:SCHtorus} is given by $\mu\ce -\Delta u+u^3-u$.
    \end{definition}

Let us recall some central global well-posedness results in dimension $3$ from Subsections \ref{subsec:introStochasticLiterature} and \ref{subsec:contributions} and further elaborate how our result compares to them.
For the global well-posedness obtained in~\cite{Scarpa2018} on smooth domains in $\R^3$ with homogeneous Neumann boundary conditions, working in a weaker setting and with gradient-independent noise allowed to study a larger class of potentially singular potentials. For multiplicative Fourier noise, $H^1$- and, under higher regularity assumptions, $H^2$-moment estimates were established in~\cite{Antonopoulou23}. In our setting, the Fourier noise corresponds to the special case of $\psi_n(x,u,\grad u)=\sigma(u)a_n e_n(x)$ for an orthonormal basis $(e_n)_{n\ge 1}$ of $L^2(\sO)$ and eigenvalues $a_n^2$ of the covariance operator. Hence, our results in $d=3$, cf.\ Theorem~\ref{thm:global_well_posedness_dim_3}, generalise the a priori estimates of~\cite[Thm.~2.2]{Antonopoulou23} for $\varepsilon=p=1$ to more general noise that may depend on $\grad u$ and more freely on $u$, and weaker summability assumptions on $(e_n)_{n\geq 1}$ are required. Closely related transport noise was studied very recently in~\cite{diPrimioPapiniScarpa26} including potentials with a singular part, for which $L_t^\infty H_x^1 \cap L_t^2H_x^2$-energy estimates are proven. Transport noise is included in Assumption \ref{Assumption:noise}, corresponding to $\psi^0=0$, and we show energy estimates in the stronger $C_tH^1\cap L_t^2H_x^3$-norm for the regular Allen--Cahn-type potential.\\

    We first give a heuristic proof of global well-posedness in $d=3$ before
we explain how to make the argument rigorous. To prove global well-posedness, it
suffices to show, due to subcriticality by Theorem~\ref{thm:localWP_weakened},
that
    \begin{equation}\label{eq:sec3-intro-heuristicproof1}
     \sup_{t\in[0,\sigma\wedge T)} \|u(t)\|_{H^1(\sO)} \text{ is } \P \text{-almost
surely finite for any }T>0.
     \end{equation}
    We use the energy functional from above to derive an $H^1(\sO)$-energy estimate.
We note here that it would make the rigorous Itô calculus considerably simpler if one
could directly close the estimates of the $H^1(\sO)$-norm, but this does not seem to be
possible. We rewrite our SPDE as $\dd u = \Delta \mu \dt + \sum_{n\geq
1}\psi_n(u,\grad u) \dd W^n$, heuristically apply Itô's formula and integrate by
parts, which for some stopping time $\sigma_m$ gives
    \begin{equation*}
    \begin{aligned}
     \dd E(u(t\wedge \sigma_m)) & = \dd \int_\sO  \frac{1}{2}|\grad u(t\wedge \sigma_m)|^2
        + \Phi(u(t\wedge \sigma_m)) \dx \\
    & =  \int_{\sO} \grad u  \cdot \grad \Delta \mu  + \Phi'(u) \Delta \mu  \dx \dt
    + \text{noise terms}\ \\
    & = - \int_{\sO} |\grad  \mu|^2 \dx \dt + \text{noise terms}.
    \end{aligned}
    \end{equation*}
    Combining suitable growth estimates to bound some of the terms induced by the noise and
    a stochastic version of Gronwall's inequality yields $\sup_{t\in[0,\sigma\wedge
T)}E(u(t))< \infty$ almost surely. Noting that  $\|u(t\wedge
\sigma_m)\|^2_{H^1(\sO)} \lesssim \|u(t\wedge \sigma_m)\|^2_{L^2(\sO)} + E(u(t\wedge
\sigma_m))$ and deriving an $L^2(\sO)$-energy estimate implies
\eqref{eq:sec3-intro-heuristicproof1} and hence global well-posedness.

The key step in making this proof rigorous is to establish an Itô formula for 
computing $E(u(t\wedge\sigma_m))$. More precisely, for computational convenience, 
we modify $E(u)$ up to lower-order terms and instead work with
\[
\mathcal{E}(u(t))\ce\frac{1}{4}\|u(t)\|_{L^4(\sO)}^4 + \frac{1}{2}\|\grad u(t)\|_{L^2(\sO)}^2,\quad (t \in [0,\sigma)).
\]
While the $\frac{1}{2}\|\grad u(t)\|_{L^2(\sO)}^2$-term can be computed using the standard Itô formula in the variational setting, the $\|u(t)\|^4_{L^4(\sO)}$-term requires additional justification. Indeed, the corresponding Itô formula does not follow directly from existing general Itô formulas, since sufficient regularity of coefficients of the SPDE is not ensured by local well-posedness. As we work on a bounded domain, we establish the required Itô formula in Proposition~\ref{prop:itoformulaL4} via a regularisation argument based on Yosida approximations.

\subsubsection{The critical dimension \texorpdfstring{$d=4$}{d=4}}
\label{subsec:overviewglobalWPd4}

Dimension four is the critical dimension in the weakened setting with the Allen--Cahn-type nonlinearity we consider and multiplicative transport-type noise. Due to the criticality, the blow-up criterion additionally contains the $L_t^2H_x^3$-norm, thus requiring an additional energy estimate. By elliptic regularity of the Neumann Laplacian, it suffices to obtain an estimate for $\int_0^{t}\|\grad \Delta u\|_{L^2}^2\ds$ up to some stopping time. As in the subcritical case, we get an $\mathcal{E}$-energy estimate, resulting in
\begin{equation*}
    \int_0^{t\land\sigma_m}\|\grad \mu\|_{L^2}^2\dt = \int_0^{t\land\sigma_m}\|-\grad \Delta u+\grad u^3-\grad u\|_{L^2}^2\dt <\infty.
\end{equation*}
At first sight, this seems very close to the estimate required. However, some additional work is needed to pass from this estimate for the gradient of the chemical potential to an estimate for $\grad \Delta u$ alone. This step can be interpreted as a \textit{nonlinear elliptic maximal regularity estimate for solutions of the parabolic SPDE}: We show that, up to some stopping time, not only the difference $\grad \mu$ of linear and nonlinear terms but also each term individually has $L^2(0,\sigma;L^2(\sO))$-regularity. That is, \begin{equation*}
    \grad \mu = -\grad \Delta u + \grad u^3-\grad u \in L^2(0,\sigma;L^2(\sO)) \quad\Rightarrow \quad\grad \Delta u, \grad u^3,\grad u \in L^2(0,\sigma;L^2(\sO)),
\end{equation*}
which can be seen as a nonlinear elliptic maximal regularity estimate. Since its proof relies on the $L_t^\infty H_x^1$-energy estimate for $u$ established via $\mathcal{E}$, it is, however, not merely an elliptic estimate but uses the underlying parabolic structure. Besides, this step makes use of the Gagliardo--Nirenberg, Poincaré--Wirtinger, and Young inequalities. It is worth pointing out that we impose no additional regularity assumptions compared to the subcritical case.

In the deterministic literature, several approaches exist to treat the critical dimension in the weakened setting. On the full space $\R^4$, an $L_t^\infty H_x^2$-energy estimate uniformly in times away from zero is established in~\cite[Thm.~3.1]{CholewaRodriguezB14} under an additional coercivity assumption by leveraging the regularizing effect of the semigroup. Also on the full space, but in a stronger setting for initial data in $H^5(\R^4)$,~\cite{Cahn-Hilliard-Brinkman-global-d4} establishes global well-posedness of the more general class of Cahn--Hilliard--Brinkman equations via testing the chemical potential with $u^5$ followed by Gagliardo--Nirenberg and Young's inequalities. Extending global well-posedness from $\R^4$ to smooth bounded domains $\sO\seq \R^4$ and Neumann boundary conditions appears accessible by slight modifications but has, to the best of the authors' knowledge, not yet been formulated in the weakened setting for the deterministic Cahn--Hilliard equation. Since the latter is recovered by setting the noise in \eqref{eq:SCHtorus} equal to zero, global well-posedness of the deterministic Cahn--Hilliard equation in $H^1(\sO)$ follows from our result in Theorem~\ref{thm:global_well_posedness_dim_4-complete} as a special case.

For the stochastic version of the Cahn--Hilliard equation with multiplicative noise, much less is known regarding global well-posedness in dimension $4$. 
For deterministic initial values $u_0\in H^1$ and gradient-independent, spatially homogeneous multiplicative noise of Nemytskii type, global existence of martingale solutions has been established in~\cite[Thm.~2.2]{Scarpa2021b}. Pathwise uniqueness or regularity $u\in L_t^2H_x^3$ have not been obtained therein, noting that~\cite[(2.4)]{Scarpa2021b} is not satisfied for the cubic nonlinearity considered here. To the the best of the authors' knowledge, global well-posedness of probabilistically strong solutions with initial data in $H^1$ in dimension $4$ has not been established before, neither for gradient-independent or transport noise.

\subsubsection*{Overview}  Subsection~\ref{subsec:L2energyEst} contains the $L^2$-energy estimate in dimensions $d\in \{3,4\}$ and the Itô formula adapted to fourth-order equation and Neumann boundary conditions used throughout. The $L^4$-energy estimate is derived in Subsection~\ref{subsec:L4energyEst} based on an $L^4$-Itô formula under lower regularity assumptions, which is proved via Yosida approximations. Together with the $H^1$-energy estimate established in Subsection~\ref{subsec:H1energyEst}, global well-posedness in dimension three is obtained via the energy functional $\mathcal{E}$ in Subsection~\ref{subsec:globalWPd3}. Global well-posedness in the critical dimension $4$ is shown in Subsection~\ref{subsec:globalWPd4}. The main results are Theorems~\ref{thm:global_well_posedness_dim_3} and~\ref{thm:global_well_posedness_dim_4-complete} in dimensions $3$ and $4$, respectively.

\subsection{\texorpdfstring{$L^2$}{\textit{L2}}-energy estimate}
\label{subsec:L2energyEst}

We start with an $L^2$-energy estimate, which will provide the estimates required for the highest-order terms remaining in $\mathcal{E}(u)$ after completing the square. Before, we recall some of the growth properties of the noise following from Assumption~\ref{assumption_weakened_setting} that will frequently be used in the following. 

\begin{remark}\label{remark:assumptionsNoiseGlobal}

Suppose that Assumption~\ref{assumption_weakened_setting} holds and recall Corollary~\ref{cor:weakened_noise_consequences}. Then there exists a constant $C_\psi\ge 0$ such that for a.e.\ $t\ge 0$, $\omega\in\Omega$, and $x\in\sO$ for all $y\in \R$ and $z\in\R^d$, the noise $\psi$ satisfies
\begin{align}\label{eq:noiseLinGrowthyz}
   \| \psi(t,\omega,x,y,z)\|_{\ell^2}^2&\leq \big(\| \psi^0(t,\omega,x,y)\|_{\ell^2} + \| \psi^1(t,\omega,x)\|_{\ell^2(\R^d)} |z|\big)^2 \le C_\psi (1+|y|^2+|z|^2),\\
   \|\partial_y \psi(t,\omega,x,y,z)\|_{\ell^2}^2 & = \|\partial_y \psi^0(t,\omega,x,y)\|_{\ell^2})^2 \le C_\psi(1+|y|^2), \\ 
  \|\grad_z \psi(t,\omega,x,y,z)\|_{\ell^2}^2 & = \|\psi^1(t,\omega,x)\|_{\ell^2(\R^d)}^2 \le C_\psi, \\
  \sum_{n\ge 1}|(\grad_x\psi_n)(t,\omega,x,y,z)|^2 &\le C_\psi (1+|y|^2+|z|^2).
\end{align}
In the following, we use $C_\psi$ to denote a generic constant depending on the constants from Assumption~\ref{assumption_weakened_setting} whose value may change from line to line.
\end{remark}

The following Itô formula for computing $\|u\|_{L^{q}(\sO)}^q$ is a slight modification of~\cite[Lemma A.6]{agrestiNonlinearSPDEsMaximal2025}, where the original result for Dirichlet boundary conditions is adapted to the boundary conditions considered here. The proof is analogous to that of~\cite[Lemma A.6]{agrestiNonlinearSPDEsMaximal2025} and is therefore omitted.

\begin{lemma}\label{lemma:Modified-Ito-lemma-a6}
 Suppose that $\sO\seq \R^d$ is an open and bounded set. Let $q\in[2,\infty)$. Let $v$, $\phi$, $\Phi_j$, $\psi_n \colon [0,T]\times\Omega \to L^q(\sO)$ be progressively measurable for $j\in\{1, \ldots, d\}$ and $n\geq 1$, and suppose that
 \begin{itemize}
  \item $v_0 \in L^0_{\F_0}(\Omega; L^q(\sO))$,
  \item $v\in L^2(0, T; H^{1,q}(\sO))\cap C([0,T];L^q(\sO))$ a.s.,
  \item $\phi \in L^1(0,T;L^q(\sO))$ a.s.,
  \item $\Phi \in L^2(0,T; L^q(\sO;\R^d))$ a.s.,
  \item $(\psi_n)_{n\geq 1}\in L^2(0,T;L^q(\sO; \ell^2))$ a.s.,
 \end{itemize}
and a.s.\ for all $t\in[0,T]$, the following identity holds true for all $\varphi\in H^1(\sO)$
\begin{equation}\label{eq:ito-lemma-assumption}
(v(t), \varphi)_{L^2} = (v_0, \varphi)_{L^2}
+ \int_0^t (\phi(s), \varphi)_{L^2} \ds - \int_0^t (\Phi(s), \grad \varphi )_{L^2} \ds  + \sum_{n \geq 1} \int_0^t (\psi_n(s), \varphi)_{L^2} \dWsn.
\end{equation}
Then, a.s.\ for all $t \in [0,T]$,
\begin{align*}
\|v(t)\|_{L^q(\sO)}^q
&= \|v_0\|_{L^q(\sO)}^q
+ q \int_0^t \int_{\sO} |v(s)|^{q-2}
\left[v(s)\phi(s) - (q-1)\nabla v(s)\cdot \Phi(s)\right] \dx\ds \\
&\phantom{= } + q \sum_{n \geq 1} \int_0^t \int_{\sO}
|v(s)|^{q-2} v(s)\psi_n(s) \dx \dWsn + \frac{q(q-1)}{2} \int_0^t \int_{\sO}
|v(s)|^{q-2} \|\psi(s)\|_{\ell^2}^2 \dx\ds.
\end{align*}
\end{lemma}

The following two lemmas provide auxiliary computations that establish the regularity required for the application of Itô's formula.

\begin{lemma}
\label{lem:cubicIsInH}
Let $d\in \{3,4\}$, $u_0\in L^0_{\F_0}(\Omega;H^1(\sO))$, and suppose that Assumption~\ref{assumption_weakened_setting} holds. Let $(u,\sigma)$ be the unique maximal solution to \eqref{eq:SCHtorus} and $(\sigma_m)_{m\ge 1}$ a localising sequence for $(u,\sigma)$. 
    Then $u^3\in L^2(0,\sigma_m;H^1(\sO))$ almost surely for all $m \ge 1$.
\end{lemma}
\begin{proof}
    We omit the stopping from the notation of this proof for the sake of readability. Since $\grad(u^3)=3u^2\grad u$, it is sufficient to show $u^3\in L_t^2L_x^2$ and $u^2\grad u\in L_t^2L_x^2$. The former is equivalent to $u\in L_t^6L_x^6$ due to $\|u^3\|_{L_t^2L_x^2}= \|u\|_{L_t^6L_x^6}^3$. To estimate the latter, we apply Hölder's inequality first in space and in time and make use of the Sobolev embeddings $H^{4/3}(\sO)\hra L^6(\sO)$ and $\nH^{7/3}(\sO)\hra H^{1,6}(\sO)$ valid in dimensions $d\in \{3,4\}$. This results in
    \begin{equation*}
        \|u^2\grad u\|_{L_t^2L_x^2} \le \big\| \|u\|_{L_x^6}^2 \|\grad u\|_{L_x^6}\big\|_{L_t^2} \le \|u\|_{L_t^{12}L_x^6}^2 \|u\|_{L_t^3H_x^{1,6}} \le \|u\|_{L_t^{12}H_x^{4/3}}^2 \|u\|_{L_t^3H_x^{7/3}} \le \|u\|_{L_t^{12}(X_{7/12})}^2 \|u\|_{L_t^3(X_{5/6})},
    \end{equation*}
    where we have embedded into the complex interpolation spaces $X_\beta$, $\beta \in (0,1)$, between $V^\ast$ and $V=\nH^3(\sO)$ in the last step, using that $[H,V]_\theta=[V^\ast,V]_{(1+\theta)/2}=\nH^{1+2\theta}$ for $\theta\in(0,1)$. An application of~\cite[Lemma~4.3]{agrestiNonlinearSPDEsMaximal2025}, once with $\rho=5$, $p=2$, $\beta=\frac{7}{12}$, $\theta=\frac{1}{6}$, and once with $\rho=\frac{1}{2}$, $p=2$, $\beta=\frac{5}{6}$, $\theta=\frac{2}{3}$ is possible due to $\frac{1+\rho}{\rho}(1-\beta)=\frac{1}{2}$ for both sets of parameters. It yields the existence of $C \ge 0$ independent of $T$, $n$ and $u$ such that for all $m\geq 1$
    \begin{equation*}
        \|u^2\grad u\|_{L^2(0,\sigma_m;L^2(\sO))} \le \|u\|_{L_t^{12}(X_{7/12})}^2 \|u\|_{L_t^3(X_{5/6})} 
        \le C \|u\|_{L^\infty(0,\sigma_m;H^1(\sO))}^{2} \|u\|_{L^2(0,\sigma_m;\nH^3(\sO))}.
    \end{equation*}
    By Theorem~\ref{thm:localWP_weakened}, a.s.\ $u\in C([0,\sigma_m];H^1(\sO))\cap L^2(0,\sigma_m;\nH^3(\sO))$, whence the right-hand side is finite. This concludes the proof, since the bound for the norm of $u$ in $L_t^{12}L_x^6$ implies that $u\in L_t^6L_x^6$, as the time interval considered is finite. 
\end{proof}

\begin{lemma}
\label{lem:noiseIsInH}
    Let $d\in \{3,4\}$, $u_0\in L^0_{\F_0}(\Omega;H^1(\sO))$, and suppose that Assumption~\ref{assumption_weakened_setting} holds. Let $(u,\sigma)$ be the unique maximal solution to \eqref{eq:SCHtorus} and $(\sigma_m)_{m\ge 1}$ a localising sequence for $(u,\sigma)$. Then $\psi(u,\grad u)\in L^2(0,\sigma_m;H^1(\sO;\ell^2))$ almost surely for all $m\ge 1$.
\end{lemma}
\begin{proof}
    Lemma~\ref{lem:weakened_estimatesAFG}\ref{lemItem:weakenedEstimatesG} implies
    that there is a constant $C\ge 0$ such that for all $u\in V$, we have
    \begin{align*}
         \| G(\cdot,u) \|_{\sL_2(\ell^2,H)}=\|\psi(u,\grad u)\|_{\sL_2(\ell^2,H_x^1)} &\lesssim C\sum_{j=4}^6 \big(1+\|u \|_{\beta_j}^{\rho_j+1}\big),
    \end{align*}
    where all $(\beta_j,\rho_j)$ satisfy the (sub-)criticality condition $\frac{1}{2}\le \frac{1+\rho_j}{\rho_j}(1-\beta_j)$. Hence, by~\cite[Remark~4.2]{agrestiNonlinearSPDEsMaximal2025} we can conclude that $\|\psi(u,\grad u)\|_{\sL_2(\ell^2,H^1)}$ is in $L^2(0,\sigma_m)$ since $u\in L^2(0,\sigma_m;V)$ by local well-posedness.   
\end{proof}

 \begin{lemma}[$L^2$-energy estimate]
    \label{lem:L2energyEstimate}
        Let $d\in\{3,4\}$, $u_0\in L^0_{\F_0}(\Omega; H^1(\sO))$  and let Assumption~\ref{assumption_weakened_setting} hold. Let $(u,\sigma)$ be the unique 			maximal solution to \eqref{eq:SCHtorus}. Let $T>0$ and let $(\sigma_m)_{m\geq 1}$ be a localising sequence for $(u,\sigma)$. Then, almost surely, for all $t\geq 0$
        \begin{align}\label{eq:lemL2energyEstimate-martingaleEq}
         \|u(t\wedge\sigma_m)\|_{L^2}^2 &+ \inttm \|\Delta u\|^2_{L^2}\ds + 6  \inttm \|u \grad u\|_{L^2}^2 \ds \nonumber\\
         &\le \|u_0\|_{L^2}^2 + C_{\psi} \inttm \|u\|_{L^2}^2\ds +C_\psi|\sO|t
        +M_{t \land \sigma_m}
        \end{align}
        for some continuous local martingale $(M_t)_{t\geq 0}$ with $M_0=0$. Furthermore, 
        \begin{equation}
        \label{eq:L2energyEstimateAS}
         \sup_{t\in[0,T]\cap[0,\sigma)} \|u(t)\|_{L^2}^2  +  \int_0^{T\wedge \sigma} \|\Delta u\|_{L^2}^2\ds +  \int_0^{T\wedge \sigma}  \|u\grad u\|_{L^2}^2\ds  < \infty  \qquad \mathbb{P}\text{-a.s.}
    \end{equation}
    If, in addition, $u_0 \in L_{\F_0}^2(\Omega; L^2(\sO))$, then there exists a constant $C_T>0$, depending only on $T$, $d$ and the constants in Assumption~\ref{assumption_weakened_setting}, such that for all $m \geq 1$, we have
    \begin{align}
    \label{eq:L2energyEstimateMoment}
        \E\Big[ \|u(T\wedge\sigma_m)\|^2_{L^2} + \int_0^{T\wedge\sigma_m} \|\Delta u\|_{L^2}^2 \ds+ \int_0^{T\wedge\sigma_m} \|u\grad u\|_{L^2}^2\ds\Big]
        \leq C_T\big(\E\big[\|u_0\|_{L^2}^2\big]+1\big).
    \end{align}
    \end{lemma}
   
    \begin{proof}   
    
 Let $m \geq 1$ be arbitrary. We will apply the Itô formula from
    Lemma~\ref{lemma:Modified-Ito-lemma-a6} to prove the $L^2$-energy estimate. To this
    end, we first verify the assumptions for $v \ce u(\cdot \wedge \sigma_m)$, $\Phi \ce \1_{[0,\sigma_m]}(-\grad \Delta u + \grad (u^3-u))$, and $\phi\equiv 0$. By local well-posedness we have $u\in C([0,\sigma_m];H^1(\sO))$ and hence $v\in L^2(0, T; H^1(\sO))\cap C([0,T];L^2(\sO))$ a.s. Again by local well-posedness we have $u\in L^2(0,\sigma_m; V)$, and by Lemma~\ref{lem:cubicIsInH}  we have $u^3 \in L^2(0,\sigma_m; H^1(\sO))$, which implies $\Phi\in L^2(0,T;L^2(\sO;\R^d))$. By Lemma~\ref{lem:noiseIsInH} we have $(\1_{[0,\cdot\wedge\sigma_m]}\psi_n)_{n\geq 1} \in L^2(0,T; L^2(\sO; \ell^2))$ a.s., ensuring that the assumptions are satisfied.
    
    Let $t\in[0,T]$, and write $\psi_n(u,\grad u)$ for $(t,\omega,x) \mapsto \psi_n(t,\omega,x,u(\omega,t,x), \grad u(\omega,t,x))$. An application of Itô's formula,  the identity $\grad(u^3-u)=3u^2\grad u-\grad u$, the linear growth assumption \eqref{eq:noiseLinGrowthyz} from Remark~\ref{remark:assumptionsNoiseGlobal}, 
    and integration by parts yields a.s.\

        \begin{equation}\label{eq:L2energy-calc1}
      \begin{aligned}
        \|u(t\wedge\sigma_m)\|_{L^2}^2 &= \|u_0\|_{L^2}^2  - 2 \inttm \intsO \nabla u \cdot (-\grad \Delta u + 3u^2 \grad u- \grad u) \dx\ds  \\
         & \phantom{\le } + 2 \sum_{n \geq 1}  \inttm \intsO  u \psi_n(u,\grad u) \dx \dWsn  +  \inttm \intsO \|\psi(u,\grad u)\|_{\ell^2}^2 \dx \ds \\
        & \leq \|u_0\|_{L^2}^2 - 2 \inttm \intsO |\Delta u|^2 \dx\ds - 6 \inttm \intsO|\nabla u|^2 u^2  \dx\ds   \\
        & \phantom{\le }+ 2 \inttm \intsO |\nabla u|^2  \dx\ds + M_{t \land \sigma_m}
        +  C_\psi \inttm \intsO (1+|u|^2+|\grad u|^2)\dx\ds \\
        &\overset{(*)}{\le} \|u_0\|_{L^2}^2-(2-1) \inttm \intsO |\Delta u|^2\dx\ds - 6  \inttm \intsO |\grad u|^2u^2\dx\ds  \\
        & \phantom{\le } + C_\psi \inttm \intsO |u|^2\dx\ds +C_\psi|\sO|T
        +M_{t \land \sigma_m},
    \end{aligned}
    \end{equation}
    where $(M_t)_{t\ge 0}$ denotes the local martingale $(2 \sum_{n \geq 1}  \intt\intsO   u \psi_n(u,\grad u) \dx \dWsn)_{t\ge 0}$ and $|\sO|$ denotes the volume of $\sO$.
    In $(*)$, we have integrated by parts and applied Young's inequality to estimate
    \begin{equation*}
       (C_\psi+2) \intsO |\grad u|^2\dx = -(C_\psi+2) \intsO u \Delta u \dx \le  \intsO |\Delta u|^2\dx + C_\psi\intsO u^2\dx
    \end{equation*}
    after increasing the generic constant $C_\psi$ in the last inequality. This shows \eqref{eq:lemL2energyEstimate-martingaleEq}.
    
    We fix some $T>0$. The stochastic Gronwall inequality~\cite[Corollary 5.4]{Geiss24} applied to \eqref{eq:lemL2energyEstimate-martingaleEq} yields that for all $a,b>0$
    \begin{equation*}
    \begin{aligned}
     \P\bigg(\sup_{t\in[0,T]} \|u(t\wedge\sigma_m)\|_{L^2}^2 + \inttm \|\Delta u\|_{L^2}^2\ds + 6  \inttm \|u\grad u\|_{L^2}^2\ds > a \bigg) \\
     \leq \frac{1}{a}\exp(C_\psi T)
     \E\left[ (\|u_0\|_{L^2}^2 +C_\psi|\sO|T) \wedge b\right] + \P\left( \|u_0\|_{L^2}^2 +C_\psi|\sO|T  >b\right).
    \end{aligned}
    \end{equation*}
    Letting $m\to\infty$ and choosing first $b$ then $a$ sufficiently large, \eqref{eq:L2energyEstimateAS} follows.
    
Let $(\tilde{\sigma}_k)_{k \ge 1}$ be a localising sequence of $M$. We see by rearranging terms in \eqref{eq:L2energy-calc1} that
\begin{equation*}
\|u_0\|_{L^2}^2 + C_\psi \inttm  \|u(s)\|_{L^2}^2 \ds +C_\psi|\sO|T + M_{t\wedge\sigma_m} \geq 0 \qquad (t\in[0,T]),
\end{equation*}
and thus by Fatou's lemma we have
\begin{align*}
     \E&\Big[\|u_0\|_{L^2}^2 + C_\psi \inttm \|u(s)\|_{L^2}^2 \ds +C_\psi|\sO|T + M_{t\wedge\sigma_m}\Big] \\
      &\leq \liminf_{k\to\infty} \E\Big[\|u_0\|_{L^2}^2 + C_\psi \int_0^{t\wedge\sigma_m\wedge\tilde{\sigma}_k} \|u(s)\|_{L^2}^2\ds +C_\psi|\sO|T + M_{t\wedge\sigma_m\wedge \tilde{\sigma}_k}\Big] \\
      &\leq \E\Big[\|u_0\|_{L^2}^2 + C_\psi \inttm \|u(s)\|_{L^2}^2\ds+C_\psi|\sO|T\Big]
\end{align*}
Thus, by rearranging terms in \eqref{eq:L2energy-calc1}, adding the missing positive terms on the right-hand side, and applying the deterministic Gronwall inequality, we obtain for any deterministic $T>0$ that 
    \begin{align*}
        & \E\Big[\|u(T\wedge \sigma_m)\|_{L^2}^2 + \int_0^{T\land \sigma_m} \intsO|\Delta u(s,x)|^2 \dx\ds + 6 \int_0^{T\land \sigma_m} \intsO |\grad u(s,x)|^2u(s,x)^2\dx\ds\Big]  \\
        & \qquad \leq \E\big[ \|u_0\|_{L^2}^2\big] + C_\psi \E\Big[ \int_0^{T\wedge\sigma_m} \|u(s)\|_{L^2}^2\ds \Big] +C_\psi|\sO|T \\
         & \qquad \leq  \E\big[ \|u_0\|_{L^2}^2\big] + C_\psi \int_0^T \E\big[ \|u(s\wedge \sigma_m)\|_{L^2}^2\big]\ds +C_\psi|\sO|T \\
         & \qquad \leq   \big(\E\big[ \|u_0\|_{L^2}^2\big]+C_\psi|\sO|T\big) e^{C_\psi T} <\infty. \qedhere
    \end{align*}
    \end{proof}

\subsection{\texorpdfstring{$L^4$}{\textit{L4}}-energy estimate}\label{subsec:L4energyEst}

\begin{remark}
    To obtain an equation for $\|u(t)\|^4_{L^4(\sO)}$, we would like to apply~\cite[Lemma~A.6]{agrestiNonlinearSPDEsMaximal2025} or a version thereof tailored towards our boundary conditions. However, the regularity $u\in C([0,\sigma);H^1(\sO))\cap L^2_{\loc}([0,\sigma);\nH^3(\sO))$ a.s.\ provided by Theorem~\ref{thm:localWP_weakened} is not sufficient, as it implies neither the required $\Delta^2 u \in L^1_\loc ([0,\sigma); L^4(\sO))$ nor $\grad \Delta u \in L^2_\loc([0,\sigma); L^4(\sO; \R^d))$. There are two possible remedies. In Proposition~~\ref{prop:itoformulaL4} below, we deduce the $L^4$-energy estimate via a regularisation argument based on the Yosida approximations $\lambda(\lambda+A)^{-1}u$ of $u$, as we believe this method is of independent interest. Alternatively, via $L^p(L^q)$-theory we can deduce the regularity
    \[ 
        u\in H_\loc^{\theta,r}(0,\sigma; \nH^{3-4\theta,\zeta}(\sO)), \quad \forall \theta\in\Big[0,\frac{1}{2}\Big),\quad r,\zeta\in [2,\infty),
    \]
    without relying on Section~\ref{sec:globalWP}, cf.\ Lemma~\ref{thm:regularization:rough:initial:data:L2L2}, which ensures by taking $\theta = 0$, $r=2$, and $\zeta = 4$ that $\grad \Delta u \in L^2_{\loc}([0,\sigma); L^4(\sO; \R^d))$.
\end{remark}

\begin{lemma}
\label{lem:resolventBoundLambda}
    For $\lambda \in \rho(-A)\cap \R$, 
    denote by $R_\lambda \ce (\lambda+A)^{-1}:V^\ast\to V$ the resolvent of the Bi-Laplacian $A$ as defined in \eqref{def_weakened_Bilaplacian}. Then for all $\varepsilon>0$, it holds that $\sup_{\lambda>\varepsilon}\|\lambda R_\lambda\|_{\sL(X_\beta)}\le C_{\varepsilon,\beta}$ and $\lambda R_\lambda \to I$ as $\lambda \to \infty$ in the strong operator topology on $X_\beta$ for all $\beta\in [0,1]$, i.e.\ $\|\lambda R_\lambda x-x\|_{X_\beta} \to 0$ as $\lambda \to \infty$ for all $x \in X_\beta$. 
\end{lemma}
\begin{proof}
    As shown in Subsection~\ref{weakened_setting}, the operator $A\in \sL(V, V^\ast)$ satisfies the boundedness and variational conditions for the local well-posedness. Thus, by~\cite[Lemma~4.1]{agrestiCriticalVariationalSetting2024a}, $A$ has stochastic maximal $L^2$-regularity and in particular $-A$ generates an analytic semigroup on $V^\ast=X_0$. This implies convergence of the Yosida approximations $\lambda R_\lambda$ (see e.g.\ \cite[Prop.~10.1.7]{hytonenAnalysisBanachSpaces2017}). Moreover, by~\cite[Theorem~2.1.3]{amannLinearQuasilinearParabolic1995}, the restriction of this analytic semigroup to the space $X_\beta$, $\beta\in (0,1]$, yields again an analytic semigroup and its generator is given by the $X_\beta$-realisation of $-A$. Denoting this operator by $-A_{\beta}$, we have $\rho(-A)=\rho(-A_{\beta})$ and the corresponding resolvents are consistent. Therefore, for $x\in X_\beta$, the Yosida approximation $\lambda R_\lambda x$ also converges to $x$ in the norm of $X_\beta$. 
\end{proof}

\begin{proposition}[Itô's formula on $L^4$]
\label{prop:itoformulaL4}
    Let $d\in \{3,4\}$, $u_0\in L^0_{\F_0}(\Omega;H^1(\sO))$, and suppose that Assumption~\ref{assumption_weakened_setting} holds. Let $(u,\sigma)$ be the unique maximal solution to \eqref{eq:SCHtorus} and $(\sigma_m)_{m\ge 1}$ a localising sequence for $(u,\sigma)$. Then a.s.\ for all $t\in [0,T]$
    \begin{align}
    \label{eq:ItoFormulaL4}
        \frac{1}{4}\|u(t\wedge\sigma_m)\|_{L^4}^4 &=  \frac{1}{4}\|u_0\|_{L^4}^4 
        +3\inttm\int_{\sO} |u|^2 \grad u \cdot [\grad\Delta u- \grad(u^3-u)]\dx\ds\\
        &\phantom{= } + \sum_{n \geq 1} \inttm \int_{\sO}|u|^2 u \,\psi_n(u,\grad u) \dx \dWsn + \frac{3}{2} \inttm\int_{\sO} |u|^2 \|\psi(u,\grad u)\|_{\ell^2}^2 \dx \ds. \nonumber
\end{align}
\end{proposition}

Before proving this Itô formula, we derive the $L^4$-energy estimate required for the energy functional $\mathcal{E}$.

\begin{proposition}[$L^4$-energy estimate] \label{prop:L4energyEstimate}
Under the assumptions and with the notation of Proposition~\ref{prop:itoformulaL4}, 
 a.s.\ for all $t\in[0,T]$
    \begin{align}\label{eq:L4energyEst}
        \frac{1}{4}\|u(t\land \sigma_m)\|_{L^4}^4 &\le \frac{1}{4}\|u_0\|_{L^4}^4 +
         3 \inttm\intsO |u|^2 \grad u \cdot \grad \Delta u \dx\ds + 3 \inttm\intsO |u|^2 \grad u \cdot [- 3u^2 \grad u + \grad u] \dx\ds \nonumber\\
        &\phantom{\le } + M_{t\land \sigma_m} + \frac{3C_\psi}{2} \inttm\intsO |u|^2 (1+|u|^2 + |\grad u|^2)  \dx \ds,
    \end{align}
    for $M_t \ce \sum_{n \geq 1} \intt\intsO |u|^2 u \psi_n(u, \grad u) \dx \dWsn$, which is a local martingale with quadratic variation satisfying for all $\varepsilon>0$ and $t\in [0,T]$
    \begin{equation}
    \label{eq:M1quadrVarEstimate}
     \E\big[[M]_{t\wedge \sigma_m}^{1/2}\big]\le  \varepsilon\E\bigg[\sup_{s\in[0,t\land \sigma_m]} \|u(s)\|^4_{L^4}\bigg]
    +\frac{L}{4\varepsilon}\E\Big[\inttm  \intsO \big(|u|^2+|u|^4+|u|^2|\grad u|^2\big)\dx\ds\Big].
    \end{equation}
\end{proposition}
\begin{proof}
     The estimate of the $L^4$-norm follows from Proposition~\ref{prop:itoformulaL4}, $\grad(u^3)=3u^2\grad u$, and the linear growth estimate \eqref{eq:noiseLinGrowthyz} of $\psi$ ensured by Assumption~\ref{assumption_weakened_setting}.
     We estimate the quadratic variation using Tonelli's theorem, linear growth of $\psi$, and Hölder's as well as Young's inequalities, which results in
\begin{align*}
    \E\big[[M]_{t\wedge \sigma_m}^{1/2}\big]
    &  =\E\bigg[\Big(\inttm \sum_{n \geq 1} \Big|\intsO |u|^2 u \psi_n(u,\grad u) \dx \Big|^2 \ds\Big)^{1/2} \bigg]\\
    & \leq \E\bigg[\Big(\inttm \|u(s)\|^4_{L^4} \intsO u^2 \|\psi(u,\grad u)\|_{\ell^2}^2\dx \ds\Big)^{1/2} \bigg] \\
    &\le \sqrt{C_\psi} \E\Big[\Big(\inttm \|u(s)\|^4_{L^4} \intsO u^2 (1+|u|^2+|\grad u|^2)\dx \ds\Big)^{1/2} \Big] \\
    &\le \sqrt{C_\psi}\E\bigg[\bigg(\sup_{s\in[0,t\land \sigma_m]} \|u(s)\|^4_{L^4}\bigg)^{1/2}\Big(\inttm \intsO \big(|u|^2+|u|^4+|u|^2|\grad u|^2\big)\dx \ds\Big)^{1/2} \bigg] \\
    &\le \varepsilon\E\bigg[\sup_{s\in[0,t\land \sigma_m]} \|u(s)\|^4_{L^4}\bigg]
    +\frac{C_\psi}{4\varepsilon}\E\Big[\inttm  \intsO \big(|u|^2+|u|^4+|u|^2|\grad u|^2\big)\dx\ds\Big]. \qedhere
\end{align*}
\end{proof}

\begin{proof}[Proof of Proposition~\ref{prop:itoformulaL4}]
    We regularize the equation \eqref{CH} using the resolvent
    \begin{equation*}
        R_{\lambda} \ce R(\lambda, -A)\ce (\lambda + A)^{-1}\colon V^\ast \to V \seq V^\ast
    \end{equation*}
    for $\lambda \in \rho(-A)$ in order to apply Itô's formula in the version of Lemma~\ref{lemma:Modified-Ito-lemma-a6} to the solution of the regularised equation. 
    This results in
    \begin{equation*}
        \rmd R_\lambda u = R_\lambda[-Au+F(u)]\dt + \sum_{n \geq 1} R_\lambda \psi_n(u,\grad u)\,\rmd W^n,\quad R_\lambda u(0)= R_\lambda u_0
    \end{equation*}
    in $V^\ast$. By the identification $\langle R_\lambda v^*, v\rangle_{V^\ast \times  V} = (R_\lambda v^*, v)_{H}$ for $v^\ast\in V^\ast, v\in V$, this corresponds to
    \begin{equation*}
        \rmd (R_\lambda u, v)_H  =  (R_\lambda[-Au], v)_H  \dt + (R_\lambda F(u), v)_H \dt + \sum_{n \geq 1} (R_\lambda \psi_n(u,\grad u), v)_H \,\rmd W^n,\quad \forall v\in V.
    \end{equation*}
    Using that $R_\lambda$ maps into $V \seq H$, we have $(r, v)_H = (r, (1-\Delta) v)_{L^2}$ for all $r\in H$, $v\in V$, and since $(1-\Delta):V \to H$ is an isomorphism, we obtain for any $\varphi \in H$ via $v\ce (1-\Delta)^{-1}\varphi$ that
    \begin{equation*}
        \rmd (R_\lambda u, \varphi)_{L^2} =  (R_\lambda[-Au], \varphi)_{L^2}\dt     + (R_\lambda F(u), \varphi)_{L^2} \dt + \sum_{n \geq 1} (R_\lambda \psi_n(u,\grad u), \varphi)_{L^2} \,\rmd W^n.
    \end{equation*}
    Define $u_\lambda \ce \lambda R_\lambda u$. Using the Sobolev embeddings $V = \nH^3(\sO) \hookrightarrow H^{1,4}(\sO) \hra L^4(\sO)$, we immediately see that the requirements of Lemma~\ref{lemma:Modified-Ito-lemma-a6} are met,
    resulting in
    \begin{align}
    \label{eq:ItoFormulaLambda}
      \frac{1}{4}&\|u_\lambda(t\wedge\sigma_m)\|_{L^4}^4 =  \frac{1}{4}\|\lambda R_\lambda u_0\|_{L^4}^4 +  \inttm\int_{\sO} |u_\lambda|^2 u_\lambda
    \lambda R_\lambda[-Au+ F(u)]\dx\ds\nonumber\\
      &\phantom{\ec }+ \sum_{n \geq 1} \inttm \int_{\sO}|u_\lambda|^2 u_\lambda 
      \lambda R_\lambda \psi_n(u,\grad u) \dx \dWsn  + \frac{3}{2} \inttm\int_{\sO} |u_\lambda|^2 \|\lambda R_\lambda \psi(u,\grad u)\|_{\ell^2}^2 \dx \ds.\nonumber \\
      &\ec I_1+I_2+I_3+I_4.
    \end{align}
    We investigate the behaviour of each term as $\lambda \to \infty$. Henceforth, denote by $C$ a generic constant not depending on $\lambda$.
    Without loss of generality, assume that $\lambda >1$. By Lemma~\ref{lem:resolventBoundLambda}, it then holds for all $\beta \in [0,1]$ that $\|\lambda R_\lambda\|_{\sL(X_\beta)}\le C_{1,\beta}$.
    Using the Sobolev embedding $H=H^1(\sO) \hra L^4(\sO)$ and $u_0\in H^1$ a.s.,
\begin{equation*}
   \|(\lambda R_\lambda - I) u_0\|_{L^4} \le  C \|(\lambda R_\lambda - I) u_0\|_{H^1}  \to 0 \quad \text{as } \lambda \to\infty
\end{equation*}
follows by strong convergence of $\lambda R_\lambda \to I$ in $X_{1/2}=H^1$ (cf.\ Lemma~\ref{lem:resolventBoundLambda}) and thus in $L^4$. Hence, $I_1 \to \frac{1}{4}\|\lambda R_\lambda u_0\|_{L^4}^4$. Analogously, convergence of the left-hand side of \eqref{eq:ItoFormulaLambda} to $\frac{1}{4}\|u(t\land \sigma_m)\|_{L^4}^4$ can be deduced because $u\in C([0,\sigma_m];H^1(\sO))$ a.s.\ by local well-posedness (cf.\ Theorem~\ref{thm:localWP_weakened}).

We next estimate the term $I_2$. Here and in the following, $L_t^p$-norms are taken over the time interval $(0,t\land\sigma_m)$.  Since $A\in \sL(V, V^\ast)$, it holds that $\|Au\|_{L^2(0,t\land \sigma_m;V^\ast)}\le C \|u\|_{L^2(0,t\land \sigma_m;V)}$. Further invoking~\cite[Remark~4.2]{agrestiNonlinearSPDEsMaximal2025} and the growth estimate from Lemma~\ref{lem:weakened_estimatesAFG}\ref{lemItem:weakenedEstimatesF}, we deduce
    \begin{equation}
    \label{eq:globalWellposedness-verifyIto1}
        \inttm \|Au(s)\|^2_{V^\ast} + \|F(u(s))\|^2_{V^\ast} \ds <\infty\quad \mathbb{P}\text{-a.s.}
    \end{equation}
    from the local well-posedness $u\in C([0,\sigma_m];H)\cap L^2(0,\sigma_m;V)$ a.s.\ for all $m\ge 1$. We fix some $\omega \in \Omega$ in the full-probability event on which $u(\omega)\in C([0,\sigma);H^1(\sO))\cap L^2_{\loc}([0,\sigma);\nH^3(\sO))$. Denote by $\ast_2$ the dual space w.r.t.\ $L^2$ and by $\ast$ the dual space in our variational framework, that is, w.r.t.\ $H=H^1$. For a.e.\ $s\in [0,\sigma_m]$, which we omit from the notation together with $\omega$ in the following, we have
\begin{align}
    \int_{\sO} |u_\lambda|^2 u_\lambda
    \lambda R_\lambda[-Au+ F(u)]\dx&= \big( u_\lambda^3, \lambda R_\lambda[-Au+ F(u)] \big)_{L^2}\nonumber \\
    &= \big\langle u_\lambda^3, \lambda R_\lambda[-Au+ F(u)] \big\rangle_{H^1,(H^1)^{\ast_2}},\label{eq:prop-7-8-eq1}
\end{align}
where the terms are well-defined because $R_\lambda$ maps to $V\seq L^2$ and we show $u_\lambda^3(s)\in H^1$ a.s.\ below. Moreover, due to $\langle v,-Au+F(u)\rangle_{H^1,(H^1)^{\ast_2}}= (\grad v, \grad\Delta u-\grad(u^3-u))_{L^2}$ for all $v\in H^1$ and $\grad(u^3)=3u^2\grad u$, we have
\begin{align}\label{eq:prop-7-8-eq2}
    3\int_{\sO} |u|^2 \grad u \cdot [\grad\Delta u- \grad(u^3-u)]\dx
    =\langle u^3, -Au+ F(u)\rangle_{H^1,(H^1)^{\ast_2}}.
\end{align}
Hence, by \eqref{eq:prop-7-8-eq1} and \eqref{eq:prop-7-8-eq2}, it suffices to show that, as $\lambda \to \infty$, the quantity
\begin{equation}\label{eq:Lemma311-I3}
\begin{aligned}
	&\Big| \inttm \int_{\sO} |u_\lambda|^2 u_\lambda 
    \lambda R_\lambda[-Au+ F(u)]-3|u|^2 \grad u\cdot[\grad\Delta u- \grad(u^3-u)]\dx\ds \Big|\\
    &= \Big|\inttm  \big\langle u_\lambda^3-u^3, \lambda R_\lambda[-Au+ F(u)] \big\rangle_{H^1,(H^1)^{\ast_2}}
    + \big\langle u^3, (\lambda R_\lambda-I)[-Au+ F(u)] \big\rangle_{H^1,(H^1)^{\ast_2}} \ds\Big|
\end{aligned}
\end{equation}
converges to $0$ in order to conclude that $I_2$ converges to the desired limit.

    To this end, we apply the isomorphism $(H^1)^{\ast_2}\simeq V^\ast$ established in \eqref{eq:dualIsomorphism}, Lemma~\ref{lem:resolventBoundLambda} with $\beta=0$, and Hölder's inequality to deduce the following upper bound of \eqref{eq:Lemma311-I3}:
\begin{align*}
 C \|u_\lambda^3-u^3\|_{L_t^2 H_x^1} \|-Au+ F(u)\|_{L^2_t V^\ast}
   + C \|u^3\|_{L_t^2H_x^1} \|(\lambda R_\lambda-I)[-Au+ F(u)]\|_{L^2_tV^\ast}.
\end{align*}
Lemma~\ref{lem:cubicIsInH} ensures that $\|u^3\|_{L_t^2H_x^1}$ is finite and \eqref{eq:globalWellposedness-verifyIto1} implies that $\|-Au+ F(u)\|_{L^2_t V^\ast}$ is finite. An application of Lemma~\ref{lem:resolventBoundLambda} with $\beta=0$ yields that $\|(\lambda R_\lambda-I)[-Au+ F(u)]\|_{V^*} \leq C_{1,0} \|-Au+ F(u)\|_{V^*}$ for every $\lambda\geq 1$ and for a.e.\ $s\in[0,\sigma_m]$ as well as $\lim_{\lambda \to\infty}\|(\lambda R_\lambda-I)[-Au+ F(u)]\|_{V^*} =0$ for a.e.\ $s\in[0,\sigma_m]$. Hence, we may conclude by dominated convergence that $\|(\lambda R_\lambda-I)[-Au+ F(u)]\|_{L^2_tV^\ast}$ converges to $0$. It remains to verify that $\|u_\lambda^3-u^3\|_{L_t^2 H_x^1} $ converges to $0$. To this end, we observe that
\begin{align*}
    \grad (u_\lambda^3)-\grad (u^3) = (u_\lambda-u)(u+u_\lambda)\grad u_\lambda + u^2\grad(u_\lambda-u)\quad\text{and}\quad
    u_\lambda^3-u^3 = (u_\lambda-u)(u^2+uu_\lambda+u^2).
\end{align*}
This enables us to repeat the arguments from the proof of Lemma~\ref{lem:cubicIsInH} for the difference $u_\lambda^3-u^3$ rather than $u^3$. For $a,b,c\in \{u,u_\lambda,u-u_\lambda\}$ chosen according to this difference representation, we estimate 
\begin{equation}\label{eq:Lemma311-I3-2}
\begin{aligned}
    \|ab\grad c\|_{L_t^2L_x^2} & \le C \|a\|_{L_t^{12}(X_{7/12})}\|b\|_{L_t^{12}(X_{7/12})}\|c\|_{L_t^{3}(X_{5/6})}  \\
    & \leq C \|a\|^{5/6}_{L_t^{\infty}H^1_x}\|a\|^{1/6}_{L_t^{2}H_x^3}
    \|b\|^{5/6}_{L_t^{\infty}H^1_x}\|b\|^{1/6}_{L_t^{2}H_x^3}
    \|c\|^{1/3}_{L_t^{\infty}H^1_x}\|c\|^{2/3}_{L_t^{2}H_x^3}.
    \end{aligned}
 \end{equation}
Repeating the analogous arguments as for $\|(\lambda R_\lambda-I)[-Au+ F(u)]\|_{L_t^2 V^*}$, i.e.\ applying Lemma~\ref{lem:resolventBoundLambda} and using dominated convergence, we can deduce that $\lim_{\lambda \to 0}\|u_\lambda - u\|_{L^2_t H_x^3}=0$. Noting that either $a$ or $c$ is $u-u_\lambda$ in the difference representation, we note that in \eqref{eq:Lemma311-I3-2} one factor will converge to $0$. At the same time, the remaining terms are bounded by Lemma~\ref{lem:resolventBoundLambda} and the regularity of $u$ ensured by local well-posedness. Thus, $\|u_\lambda^3-u^3\|_{L_t^2 H_x^1} $ converges to $0$. 

Next, we show convergence of $I_4$. As before,  we fix some $\omega \in \Omega$ in the full-probability event on which $u(\omega)\in C([0,\sigma);H^1(\sO))\cap L^2_{\loc}([0,\sigma);\nH^3(\sO))$. We split the difference with the expected limit according to
\begin{equation}\label{eq:Lemma311-eqI4-difference}
\begin{aligned}
& \Big|\inttm\intsO |u_\lambda|^2 \|\lambda R_\lambda \psi(u,\grad u)\|_{\ell^2}^2 
- |u|^2 \|\psi(u,\grad u)\|_{\ell^2}^2 \dx \ds \Big| \\
& \leq \inttm\intsO |u_\lambda-u| \big(|u_\lambda|+|u|\big) \|\lambda R_\lambda \psi(u,\grad u)\|_{\ell^2}^2 \dx \ds \\
&\quad + \inttm\intsO |u|^2 \|(\lambda R_\lambda-I) \psi(u,\grad u)\|_{\ell^2}\big( \|\lambda R_\lambda \psi(u,\grad u)\|_{\ell^2} + \|\psi(u,\grad u)\|_{\ell^2}\big) \dx \ds. 
\end{aligned}
\end{equation}
Each of the terms thus arising can be estimated by Hölder's inequality in space as well as the embedding $H^1\hra L^4$ and the isomorphism $H^1(\sO;\ell^2)\cong \sL_2(\ell^2,H^1)$ via
\begin{equation}\label{eq:Lemma311-eqI4-difference-2}
\inttm\intsO |a| |b| \|c\|_{\ell^2} \|d\|_{\ell^2} \dx \ds \leq C \inttm \|a\|_{H^1}\|b\|_{H^1} \|c\|_{\sL_2(\ell^2,H^1)} \|d\|_{\sL_2(\ell^2,H^1)}
\ds, 
\end{equation}
where 
\begin{equation*}
    a\in \{u_\lambda-u,u\},~ b \in \{u_\lambda, u\},~ c \in \{\lambda R_\lambda\psi(u,\grad u), (\lambda R_\lambda-I)\psi(u,\grad u)\},~ d\in \{\lambda R_\lambda\psi(u,\grad u), \psi(u,\grad u)\}
\end{equation*}
are chosen according to the splitting above. 
To bound the noise term, we apply Lemma~\ref{lem:resolventBoundLambda} on  $H^1(\sO)=X_{1/2}$ for each $n$, resulting in
\begin{equation*}
   \|\lambda R_\lambda \psi(u,\grad u)\|_{\sL_2(\ell^2,H^1)}^2 =  \sum_{n \geq 1} \|
 \lambda R_\lambda \psi_n(u,\grad u)\|_{H^1}^2 \\
    \le C \|\psi(u,\grad u)\|_{\sL_2(\ell^2,H^1)}^2 
\end{equation*}
 and likewise for $(\lambda R_\lambda-I) \psi(u,\grad u)$. By Lemma~\ref{lem:noiseIsInH}, $\psi(u,\grad u) \in L^2([0,\sigma_m]; \sL_2(\ell^2,H^1))$. So, by dominated convergence, $(\lambda R_\lambda-I) \psi(u,\grad u)$ converges in $\sL_2(\ell^2,H^1)$ to $0$ for a.e.\ $s$. 
We bound the integrands which appear after applying \eqref{eq:Lemma311-eqI4-difference-2} to \eqref{eq:Lemma311-eqI4-difference}:
\begin{equation*}
\begin{aligned}
\max\big\{&\|u\|^2_{H^1}\big(\|\lambda R_\lambda \psi(u,\grad u)\|_{\sL_2(\ell^2,H^1)}  +  \|\psi(u,\grad u)\|_{\sL_2(\ell^2,H^1)}\big) \|\lambda R_\lambda \psi(u,\grad u)-\psi(u,\grad u)\|_{\sL_2(\ell^2,H^1)},\\
&\|u_\lambda - u\|_{H^1}\big(\|u\|_{H^1} + \|u_\lambda\|_{H^1}\big) \|\lambda R_\lambda \psi(u,\grad u)\|^2_{\sL_2(\ell^2,H^1)} \big\}
 \leq C \|u\|^2_{H^1} \|\psi(u,\grad u)\|^2_{\sL_2(\ell^2,H^1)}.
\end{aligned}
\end{equation*}
By the observations above both integrands converge for a.e.\ $s\in[0,\sigma_m]$ to $0$ and  
\begin{equation*}
\inttm  \|u\|^2_{H^1} \|\psi(u,\grad u)\|^2_{\sL_2(\ell^2,H^1)} \ds \leq \|u\|^2_{L_t^\infty H_x^1}\|\psi(u,\grad u)\|^2_{L_t^2\sL_2(\ell^2,H_x^1)} <\infty,
\end{equation*}
so the convergence of $I_4$ to the correct limit follows by dominated convergence.

Lastly, we show convergence of the martingale term $I_3=M^{(\lambda)}_{t \land \sigma_m}$ with the martingale $M^{(\lambda)}$ defined accordingly.
Fix only $t\in [0,T]$, not $\omega\in\Omega$, and let  $M_t$ denote the local martingale defined in Proposition~\ref{prop:L4energyEstimate}.
We show that $M^{(\lambda)}_{t \land \sigma_m} \to M_{t \land \sigma_m}$ in probability as $\lambda \to \infty$, which is sufficient, since it guarantees the existence of a subsequence converging almost surely.
By Lenglart's inequality convergence in probability is ensured by convergence of the quadratic variation $[M^{(\lambda)}_{\cdot \wedge\sigma_m} - M_{\cdot \wedge\sigma_m}]_{T}\to 0$ in probability for $\lambda \to \infty$. Indeed, for any bounded stopping time $\tau$, by Itô's formula and Fatou's lemma, 
\begin{align*}
    \E\big[(M^{(\lambda)}_{\tau \wedge\sigma_m} - M_{\tau\wedge\sigma_m})^2 \big| \F_0 \big]
    &= \E\Big[ 2\int_0^\tau (M^{(\lambda)}_{s\wedge\sigma_m} - M_{s \wedge\sigma_m})\,\dd \big(M^{(\lambda)}_{s\wedge\sigma_m} - M_{s \wedge\sigma_m}\big)  + \big[M^{(\lambda)}_{\cdot \wedge\sigma_m} - M_{\cdot \wedge\sigma_m}\big]_{\tau} \Big| \F_0\Big] \\
    &\le  \E\big[ [M^{(\lambda)}_{\cdot \wedge\sigma_m} - M_{\cdot \wedge\sigma_m}]_{\tau} \big| \F_0\big].
\end{align*}
Lenglart's inequality (see~\cite[Th\'eor\`eme I]{Lenglart} or~\cite[Lemma 2.2.]{MehriScheutzow})   applied to $\sup_{t\le T} |M^{(\lambda)}_{t \land \sigma_m}-M_{t \land \sigma_m}|^2$ yields the claim after taking square roots. Hence, we show that the quadratic variation
\begin{align}
\label{eq:quadraticVariationM1InProof}
    [M^{(\lambda)} - M]_{T
    \wedge \sigma_m} & = \sum_{n \geq 1} \int_0^{T
    \wedge \sigma_m} \Big| \int_{\sO} \left(|u_\lambda|^2 u_\lambda  \lambda R_\lambda \psi_n(u,\grad u)   - |u|^2 u \psi_n(u,\grad u) \right) \dx \Big|^2 \dt   \nonumber\\
    &\le\sum_{n \geq 1} \int_0^{T
    \wedge \sigma_m}    \|u_\lambda^3\lambda R_\lambda\psi_n(u,\grad u) - u^3 \psi_n(u,\grad u) \|_{L^1}^2\dt
\end{align}
converges to $0$ in probability. Hölder's inequality, the embedding $H^1\hra L^4$, and Lemma~\ref{lem:resolventBoundLambda} with $\beta=\frac{1}{2}$ imply
\begin{align*}
    &\|u_\lambda^3\yosi\psi_n(u,\grad u)-u^3\psi_n(u,\grad u)\|_{L^1}\\
    &\le 
    \|(u_\lambda-u)(u_\lambda^2+u_\lambda u+u^2)\yosi\psi_n(u,\grad u)\|_{L^1}+
    \|u^3
    (\yosi-I)\psi_n(u,\grad u)\|_{L^1}\\
    &\le \|u_\lambda-u\|_{L^4}\|u_\lambda^2+u_\lambda u+u^2\|_{L^2}\|\yosi\psi_n(u,\grad u)\|_{L^4} + \|u\|_{L^4}^3
    \|(\yosi-I)\psi_n(u,\grad u)\|_{L^4}\\
    &\lesssim \|(\yosi-I)u\|_{H^1}\big(\|u_\lambda\|_{H^1}^2+\|u_\lambda\|_{H^1} \|u\|_{H^1}+\|u\|_{H^1}^2\big)\|\yosi\psi_n(u,\grad u)\|_{H^1} + \|u\|_{H^1}^3
    \|(\yosi-I)\psi_n(u,\grad u)\|_{H^1}\\
    &\lesssim \|(\yosi-I)u\|_{H^1}\|u\|_{H^1}^2\|\psi_n(u,\grad u)\|_{H^1} + \|u\|_{H^1}^3\|(\yosi-I)\psi_n(u,\grad u)\|_{H^1}.
\end{align*}
By the same arguments as before, local well-posedness, Lemma~\ref{lem:resolventBoundLambda} and Lemma~\ref{lem:cubicIsInH} imply convergence of the latter expression to $0$ for a.e.\ $s\in[0,\sigma_m]$, for all $n\in N$ and $\mathbb{P}$-almost surely.
Moreover, the previous estimate also implies
\begin{align*}
\|u_\lambda^3\yosi\psi_n(u,\grad u)-u^3\psi_n(u,\grad u)\|_{L^1} \lesssim C \|u\|^3_{H^1}\|\psi_n(u,\grad u)\|_{H^1}.
\end{align*}
Hölder's inequality and Lemma~\ref{lem:cubicIsInH} then yield
\begin{equation*}
\int_0^{T\wedge\sigma_m} \sum_{n \geq 1} \|u\|^6_{H^1}\|\psi_n(u,\grad u)\|^2_{H^1} \dt \leq  \|u\|^6_{L_t^\infty H^1_x}\|\psi_n(u,\grad u)\|^2_{L_t^2 \sL_2(\ell^2,H_x^1)} <\infty.
\end{equation*}
Thus, by dominated convergence, we may conclude that $[M^{(\lambda)}_{\cdot \wedge\sigma_m} - M_{\cdot \wedge\sigma_m}]_{T}\to 0$ in probability for $\lambda \to \infty$.
\end{proof}

\color{black}

\subsection{\texorpdfstring{$H^1$}{\textit{H1}}-energy estimate}
\label{subsec:H1energyEst}
\begin{proposition}[$H^1$-energy estimate]
\label{prop:H1energyEstimate}
    Let $d\in \{3,4\}$, $u_0\in L^0_{\F_0}(\Omega;H^1(\sO))$, and suppose that Assumption~\ref{assumption_weakened_setting} holds. Let $(u,\sigma)$ be the unique maximal solution to \eqref{eq:SCHtorus} and $(\sigma_m)_{m\ge 1}$ a localising sequence for $(u,\sigma)$. Let $C_\psi$ denote a constant that depends on the constants in Assumption~\ref{assumption_weakened_setting}. Then we have for all $t\in[0,T]$  $\mathbb{P}$-a.s.
    \begin{equation}\label{eq:H1energyEst}
    \begin{aligned}
    \frac{1}{2}|u(t \land \sigma_m)|^2_{H^1}
    &\le \frac{1}{2}|u_0|^2_{H^1}
    - \inttm \intsO |\grad \Delta u|^2 \dx \ds
    + \inttm \intsO  (3u^2 \grad u- \grad u) \cdot \grad \Delta u \dx \ds  \\
    &\phantom{\le } + M_{t\land\sigma_m}
    +  C_\psi\inttm  \intsO  (1+|u|^2+|\grad u|^2 + |u|^2|\grad u|^2) \dx \ds  + C_\psi \inttm   |u|_{H^2}^2 \ds
    \end{aligned}
    \end{equation}
    for a continuous local martingale  $M$ with quadratic variation $[M]$ which satisfies for all $t\in[0,T]$  and $\varepsilon>0$
    \begin{equation}
    \label{eq:M2quadrVarEstimated3}
    \begin{aligned}
    \E\big[ [M]^{1/2}_{t\wedge \sigma_m}\big] & \leq \varepsilon \E\bigg[\sup_{s\in[0,t\wedge \sigma_m]}  \|\grad u(s) \|^2_{L^2} \bigg]  +  \frac{C_\psi}{4\varepsilon}\E\Big[ \inttm  \intsO(1+|u|^2+|\grad u|^2 + |u|^2|\grad u|^2)\dx\ds\Big] \\
    &  \quad + \frac{C_\psi}{4\varepsilon}\E\Big[ \inttm  |u|_{H^2}^2 \ds\Big].
    \end{aligned}
    \end{equation}

\end{proposition}
\begin{proof} 
     We apply~\cite[Lemma A.5]{agrestiNonlinearSPDEsMaximal2025} to
    \begin{equation*}
     u(t\wedge\sigma_m) = u_0 - \inttm Au(s)\ds + \inttm F(u(s))\ds +\inttm G(u(s))\dWs
    \end{equation*}
    noting that we have sufficient regularity due to \eqref{eq:globalWellposedness-verifyIto1} and Lemma~\ref{lem:noiseIsInH}. We may directly compute the square of the seminorm $|u|_{H^1}=\|\grad u\|_{L^2}$ instead of $\|u\|_{H^1}^2$ by subtracting the Itô formula already verified for $\|u\|_{L^2}^2$ and integrating by parts. This gives
    \begin{equation}\label{eq:section7-H1}
    \begin{aligned}
    \frac{1}{2}|u(t \land \sigma_m)|^2_{H^1} = &\frac{1}{2}|u_0|^2_{H^1} - \inttm \intsO |\grad \Delta u|^2
    \dx \ds
    + \inttm \intsO  (3u^2 \grad u- \grad u) \cdot \grad \Delta u \dx \ds  \\
   & + M_{t\land\sigma_m} +  \frac{1}{2} \sum_{n \geq 1}  \inttm  \intsO |\grad (\psi_n(u, \grad u))|^2 \dx \ds
    \end{aligned}
    \end{equation}
    for $M_t \ce  \sum_{n \geq 1} \intt  \intsO \grad u \cdot \grad(\psi_n(u, \grad u))\dx \dWsn$,
    which is a local martingale with quadratic variation
    \begin{equation*}
     [M]_t =  \sum_{n \geq 1}  \intt \left|\intsO \grad u \cdot \grad (\psi_n(u,\grad u)) \dx \right|^2 \ds.
    \end{equation*}
    We bound the last term of \eqref{eq:section7-H1} using Tonelli's theorem for the first equality, and the consequences of Assumption~\ref{assumption_weakened_setting}  
    from Remark~\ref{remark:assumptionsNoiseGlobal} in the second inequality:
     \begin{equation}\label{eq:section7-h1quadvar}
     \begin{aligned}
        &\inttm \sum_{n \geq 1} \|\grad (\psi_n(u,\grad u))\|_{L^2}^2\ds
         = \inttm  \intsO \sum_{n \geq 1} |\grad (\psi_n(u,\grad u))|^2\dx\ds\\
        & \leq 3 \inttm  \intsO \sum_{n \geq 1} |(\grad_x\psi_n)(u,\grad u)|^2 + \sum_{n \geq 1} |\partial_y
     \psi_n(u,\grad u)|^2|\grad u|^2 + \sum_{n \geq 1} |\grad_z\psi_n(u,\grad u))|^2|D^2u|^2 \dx\ds\\
        &\le \inttm  \intsO \big(C_\psi (1+|u|^2+|\grad u|^2)+ C_\psi (1+|u|^2)|\grad
     u|^2 \big)\dx \ds + \inttm C_\psi|u|_{H^2}^2 \ds \\
      &\le C_\psi \inttm  \intsO  (1+|u|^2+|\grad u|^2 + |u|^2|\grad u|^2)
       \dx \ds + C_\psi\inttm |u|_{H^2}^2 \ds.
     \end{aligned}
     \end{equation} We have by Young's inequality and \eqref{eq:section7-h1quadvar} for all $\varepsilon >0$
    \begin{align}
    \label{eq:H1estimateQuadrVar}
    \E\big[ [M]^{1/2}_{t\wedge \sigma_m}\big] &  =\E\Big[\Big(\inttm \sum_{n \geq 1} \Big|\intsO \grad u \cdot \grad (\psi_n(u,\grad u)) \dx \Big|^2 \ds\Big)^{1/2} \Big] \nonumber\\
    & \leq \E\Big[\Big(\inttm \|\grad u \|^2_{L^2} \sum_{n \geq 1} \|\grad (\psi_n(u,\grad u)) \|^2_{L^2} \ds\Big)^{1/2} \Big]\nonumber \\
    & \leq \varepsilon  \E\bigg[\sup_{s\in[0,t\wedge \sigma_m]}  \|\grad u \|^2_{L^2} \bigg] + \frac{1}{4\varepsilon} \E\Big[\inttm \sum_{n \geq 1} \|\grad (\psi_n(u,\grad u)) \|^2_{L^2} \ds \Big] \nonumber\\
    & \leq \varepsilon  \E\bigg[\sup_{s\in[0,t\wedge \sigma_m]}  \|\grad u \|^2_{L^2} \bigg]  + 
    \frac{C_\psi}{4\varepsilon}\E\Big[ \inttm  \intsO\big(1+|u|^2+|\grad u|^2 + |u|^2|\grad u|^2\big)\dx\ds\Big] \nonumber\\
    &  \phantom{\le } + \frac{C_\psi}{4\varepsilon}\E\Big[ \inttm   |u|_{H^2}^2 \ds\Big],
    \end{align}
    which implies the assertion. 
\end{proof}

\subsection{Global well-posedness in dimension 3}
\label{subsec:globalWPd3}

Using the energy estimates from the previous subsections, the local well-posedness from Theorem~\ref{thm:localWP_weakened} can be improved to global well-posedness.

\begin{theorem}[Global well-posedness in $d=3$] \label{thm:global_well_posedness_dim_3}
	Let $d=3$ and suppose that the assumptions of Theorem~\ref{thm:localWP_weakened} hold. Moreover, let $u_0\in L^0_{\F_0}(\Omega;H^1(\sO))$. Then there exists a unique global solution $u$ to \eqref{eq:SCHtorus} such that $u\in C([0,\infty);H^1(\sO))\cap L_{\loc}^2([0,\infty);\nH^3(\sO))$ almost surely. If, in addition, $u_0\in L_{\F_0}^4(\Omega;L^4(\sO))\cap L_{\F_0}^2(\Omega;H^1(\sO))$, then for all $T>0$
    \begin{equation}\label{eq:thm:global_well_posedness_dim_3:EnergyEstimate}
        \E\bigg[\sup_{t\in[0,T]}  \|u(t)\|^2_{H^1(\sO)}\bigg] \le C_{\psi,T,\sO}\big(1+\|u_0\|^4_{L^4(\Omega;L^4(\sO))}+\|u_0\|^2_{L^2(\Omega;H^1(\sO))}\big) < \infty.
    \end{equation}
\end{theorem}

\begin{proof}
    \textit{Step 1:} Suppose that $\P$-a.s.\ $\sup_{t\in[0,\sigma)\cap[0,T]} \|u(t)\|_{H^1}<\infty$ for any $T>0$. We can then apply the subcritical blow-up criterion \eqref{eq:blowUpSubcritd3} of Theorem~\ref{thm:localWP_weakened} to deduce $\sigma=\infty$ a.s.\ from
    \begin{align*}
	  \P(\sigma<T) = \P\Big(\sigma<T, \sup_{t\in [0,\sigma)} \|u(t)\|_{H^1}=\infty\Big)
	   \le \P\Big(\sup_{t\in [0,\sigma)\cap[0,T]} \|u(t)\|_{H^1}=\infty\Big)=0.
    \end{align*}
    We conclude by using continuity from below of probability measures, i.e.\  $\P(\sigma<\infty) = \lim_{T\to\infty} \P(\sigma < T) = 0$.
  
    \textit{Step 2:} We show that a.s.\ for all $T\ge 0$, it holds that $\sup_{t\in[0,T]\cap[0,\sigma)}  \|u(t)\|_{H^1}<\infty$. By the $L^2$-energy estimate \eqref{eq:L2energyEstimateAS} from Lemma~\ref{lem:L2energyEstimate}, it suffices to show this for the $H^1$-seminorm instead of the $H^1$-norm. To prove it, we derive an energy estimate for $\mathcal{E}$ from the $L^4$- and $H^1$-energy estimates in \eqref{eq:L4energyEst} from Proposition~\ref{prop:L4energyEstimate} and \eqref{eq:H1energyEst} from Proposition~\ref{prop:H1energyEstimate}, respectively. After summing and rearranging them, we can complete the square and integrate 
    by parts to obtain a negative sign in front of the terms of highest order. Indeed,
    \begin{align*}
        \frac{1}{4}&\|u(t\land \sigma_m)\|_{L^4}^4 + \frac{1}{2}|u(t \land \sigma_m)|^2_{H^1} \le \frac{1}{4}\|u_0\|_{L^4}^4 + \frac{1}{2}|u_0|^2_{H^1} - \inttm \intsO |\grad \Delta u|^2 \dx \ds\\
        &\phantom{\le }+ 6 \inttm\intsO |u|^2 \grad u \cdot \grad \Delta u \dx\ds - 9 \inttm\intsO |u^2 \grad u|^2\dx\ds- \inttm \intsO \grad u \cdot \grad \Delta u \dx \ds\\
         &\phantom{\le }+ 3 \inttm\intsO |u|^2 \grad u \cdot \grad u \dx\ds + M^{(1)}_{t\land \sigma_m} + M^{(2)}_{t\land\sigma_m}+ \frac{3C_\psi}{2} \inttm\intsO |u|^2 (1+|u|^2 + |\grad u|^2)  \dx \ds\\
         &\phantom{\le } +  C_\psi\inttm  \intsO  (1+|u|^2+|\grad u|^2 + |u|^2|\grad u|^2) \dx \ds + C_\psi \inttm  |u|_{H^2}^2 \ds, \\
     &\le\frac{1}{4}\|u_0\|_{L^4}^4 + \frac{1}{2}|u_0|^2_{H^1} - \|\grad \Delta u-3u^2 \grad u\|_{L_t^2L_x^2}^2+ \|\Delta u\|_{L_t^2L_x^2}^2+ (3+C_\psi)\|u\grad u\|_{L_t^2L_x^2}^2 + M^{(1)}_{t\land \sigma_m} + M^{(2)}_{t\land\sigma_m}\\
         &\phantom{\le }+C_\psi|\sO|(t\land \sigma_m)+ C_\psi \|u\|_{L_t^2L_x^2}^2+ C_\psi \inttm \Big(\frac{1}{4}\|u\|_{L^4}^4+\frac{1}{2}|u|_{H^1}^2\Big)  \ds  + C_\psi \inttm    |u|_{H^2}^2 \ds,
    \end{align*}
    where $M^{(1)}$ and $M^{(2)}$ denote the local martingales from \eqref{eq:L4energyEst} and \eqref{eq:H1energyEst}, respectively, and $L_t^2$-norms are taken over the interval $[0,t\land \sigma_m]$. By elliptic regularity of the Neumann-Laplacian and smoothness of the domain~\cite[Cor.~2.2.2.6]{grisvardEllipticProblemsNonsmooth2011}, $|u|_{H^2}^2\le C_\sO(\|\Delta u\|_{L^2}^2+\|u\|_{L^2}^2)$ for some constant $C_\sO\ge 0$ because $u(s)\in V\seq \nH^2(\sO)$ for a.e.\ $s\in [0,\sigma_m]$.  Hence, we have for $ X_t \ce \frac{1}{4} \|u(t\land \sigma_m)\|_{L^4}^4 + \frac{1}{2}|u(t \land \sigma_m)|^2_{H^1}$ and all $t\geq 0$
    \begin{equation}\label{eq:Proof_global_well_posedness_dim_3-gronwall}
    \begin{aligned}
       X_t  & \le  C_\psi \inttm X_s  \ds  + M^{(1)}_{t\land \sigma_m} + M^{(2)}_{t\land\sigma_m}  +  \frac{1}{4}\|u_0\|_{L^4}^4 + \frac{1}{2}|u_0|^2_{H^1} \\
       &\phantom{\le } + C_{\psi, \sO}\int_0^{t\wedge\sigma_m} \big( \|\Delta u\|^2_{L^2} +\|u\grad u\|^2_{L^2} + \|u\|^2_{L^2}\big)\ds + C_\psi|\sO|(t\land \sigma_m),
    \end{aligned}
    \end{equation}
    By the $L^2$-energy estimate \eqref{eq:L2energyEstimateAS}, the second integral is a.s.\ finite, so that the stochastic Gronwall inequality is applicable to \eqref{eq:Proof_global_well_posedness_dim_3-gronwall} and yields the claim of Step 2.

    \textit{Step 3:} We show \eqref{eq:thm:global_well_posedness_dim_3:EnergyEstimate}. To this end, we take the expectation of the supremum over time up to some $T_0\in[0,T]$, drop the non-positive term, and apply the BDG inequality with constant $\CBDG>0$ to obtain
    \begin{align}
    \label{eq:calEenergyProof}
     \varphi(&T_0)\ce \E\bigg[\sup_{t\in[0,T_0\wedge \sigma_m]}  \Big(\frac{1}{4}\|u(t)\|_{L^4}^4 +\frac{1}{2}|u(t)|^2_{H^1}\Big) \bigg] \le \frac{1}{4} \E\big[\|u_0\|^{4}_{L^4}\big] + \frac{1}{2}\E\big[|u_0|^2_{H^1}\big]+ \E\big[\|\Delta u\|_{L_t^2L_x^2}^2\big]\nonumber\\
          &+ (3+C_\psi)\E\big[\|u\grad u\|_{L_t^2L_x^2}^2\big]+ \CBDG\E\big[[M^{(1)}]^{1/2}_{T_0\wedge \sigma_m}\big]  + \CBDG\E\big[[M^{(2)}]^{1/2}_{T_0\wedge \sigma_m}\big] +C_\psi|\sO|T_0\nonumber\\
          &+ C_\psi \E\big[\|u\|_{L_t^2L_x^2}^2\big]+ C_\psi \int_0^{T_0} \E\Big[\frac{1}{4}\|u(\cdot\land\sigma_m)\|_{L^4}^4+\frac{1}{2}|u(\cdot\land\sigma_m)|_{H^1}^2\Big]  \ds  + C_\psi \E\Big[\int_0^{T_0\land\sigma_m}  |u|_{H^2}^2 \ds\Big]
    \end{align}
    with $L_t^2$-norms taken over the interval $[0,T_0\land \sigma_m]$.  Again, by elliptic regularity and smoothness of the domain, the last term in \eqref{eq:calEenergyProof} can be estimated by a multiple of the terms $\E[\|\Delta u\|_{L_t^2L_x^2}^2+\|u\|_{L_t^2L_x^2}^2]$ already present in the estimate.
    Combining the quadratic variation estimates \eqref{eq:M1quadrVarEstimate} from Proposition~\ref{prop:L4energyEstimate} and \eqref{eq:M2quadrVarEstimated3} from Proposition~\ref{prop:H1energyEstimate}, rearranging and elliptic regularity give
    \begin{align*}
        &\E\big[[M^{(1)}]_{T_0\wedge \sigma_m}^{1/2}\big]+\E\big[ [M^{(2)}]^{1/2}_{T_0\wedge \sigma_m}\big] 
        \le 6\varepsilon \E\bigg[\sup_{s\leq {T_0\wedge \sigma_m}}  \Big(\frac{1}{4}\|u(s)\|_{L^4}^4+\frac{1}{2}\|\grad u(s)\|^2_{L^2}\Big) \bigg] + \frac{C_\psi}{\varepsilon}|\sO|(T_0\land\sigma_m)\\
        &\phantom{\le} + \frac{C_\psi}{\varepsilon}\big(\E\big[\|u\|_{L_t^2L_x^2}^2\big]+ \E\big[\|\Delta u\|_{L_t^2L_x^2}^2\big]+ \E\big[\|u\grad u\|_{L_t^2L_x^2}^2\big]\big)+ \frac{C_\psi}{\varepsilon}\int_0^{T_0} \E\Big[\frac{1}{4}\|u(\cdot\land\sigma_m)\|_{L^4}^4+\frac{1}{2}|u(\cdot\land\sigma_m)|_{H^1}^2\Big]\ds.
    \end{align*}
    Inserting these two estimates in \eqref{eq:calEenergyProof}, subtracting the first term in the above from both sides, and choosing $\varepsilon>0$ such that $1-6\CBDG\varepsilon=\frac{1}{2}$ results in
    \begin{align*}
        \frac{1}{2}\varphi(T_0)&\le \frac{1}{4} \E\big[\|u_0\|^{4}_{L^4}\big] + \frac{1}{2}\E\big[|u_0|^2_{H^1}\big]+ C_\psi\big(\E\big[\|u\|_{L_t^2 L_x^2}^2\big]+ \E\big[\|\Delta u\|_{L_t^2L_x^2}^2\big]+ \E\big[\|u\grad u\|_{L_t^2L_x^2}^2\big]\big)\\
        &\phantom{\le } +C_\psi|\sO|T+ C_\psi \int_0^{T_0} \varphi(s) \ds\\
        &\le \frac{1}{4} \E\big[\|u_0\|^{4}_{L^4}\big] + \frac{1}{2}\E\big[|u_0|^2_{H^1}\big]+ C_{\psi,T,\sO}\big(\E\big[\|u_0\|_{L^2}^2\big]+1\big)+ C_\psi \int_0^{T_0} \varphi(s) \ds
    \end{align*}
    for all $T_0\in[0,T]$, where we made use of the $L^2$-energy estimate \eqref{eq:L2energyEstimateMoment}  of Lemma~\ref{lem:L2energyEstimate} and $L_t^2$-norms are taken over the interval $[0,T\land \sigma_m]$. Multiplying by $2$ and a deterministic Gronwall argument yield
    \begin{align*}
        \E\bigg[\sup_{t\in[0,T\wedge \sigma_m]}  \|u(t)\|^2_{H^1}\bigg] \le C_{\psi,T,\sO}\big(1+\|u_0\|^4_{L^4(\Omega;L^4(\sO))}+\|u_0\|^2_{L^2(\Omega;H^1(\sO))}\big) < \infty.
    \end{align*}
    Fixing $T>0$, and letting $m\to\infty$ in the energy estimate above yields the same statement with the supremum taken over $[0,T]\cap [0,\sigma)$.
\end{proof}

\subsection{Global well-posedness in dimension 4}
\label{subsec:globalWPd4}

As outlined in Subsection~\ref{subsec:overviewglobalWPd4}, the critical blow-up criterion for $d=4$ contains the $L_t^2H_x^3$-norm of the solution in addition to the $L_t^\infty H_x^1$-norm compared to the subcritical dimension three. Hence, we present two energy estimates in Propositions~\ref{prop:global_well_posedness_dim_4-part1} and~\ref{prop:global_well_posedness_dim_4-part2} before concluding global well-posedness in Theorem~\ref{thm:global_well_posedness_dim_4-complete}. 

\begin{proposition} \label{prop:global_well_posedness_dim_4-part1}
	Let $d=4$ and suppose that the assumptions of Theorem~\ref{thm:localWP_weakened} hold. Moreover, let $u_0\in L^0_{\F_0}(\Omega;H^1(\sO))$ and $T>0$. Let $(u,\sigma)$ be the unique maximal solution to \eqref{eq:SCHtorus} and $\mu$ the chemical potential of Definition~\ref{def:chemPotential}. Then, we have
	\begin{equation}\label{eq:prop-global_well_posedness_dim_4-claim1}
	       \sup_{t\in[0,T]\cap [0,\sigma)}\|u(t)\|_{H^1}+\int_0^{\sigma \land T} \|\grad \mu(t)\|_{L^2}^2 \dt < \infty  \qquad \mathbb{P}\text{-a.s.}
	\end{equation}
	If, in addition, $u_0\in L_{\F_0}^4(\Omega;L^4(\sO))\cap L_{\F_0}^2(\Omega;H^1(\sO))$, then
    \begin{equation}\label{eq:prop-global_well_posedness_dim_4-claim2}
    \E\bigg[\sup_{t\in[0,T]\cap[0,\sigma)}  \|u(t)\|^2_{H^1(\sO)} + \int_0^{\sigma \land T} \|\grad \mu(t)\|_{L^2}^2 \dt \bigg] \le C_{\psi,T,\sO}\big(1+\|u_0\|^4_{L^4(\Omega;L^4(\sO))}+\|u_0\|^2_{L^2(\Omega;H^1(\sO))}\big) < \infty.
    \end{equation}
\end{proposition}

\begin{proof}
This proof closely follows the proof of Theorem~\ref{thm:global_well_posedness_dim_3}. 
By identical calculations as in Step 2 of the case $d=3$, and by moving the non-positive term to the left-hand side,  we obtain for all $t\in[0,T]$
    \begin{align}\label{eq:proof-thm-global-well-posedness-dim4-eq1}
        &\frac{1}{4}\|u(t\land \sigma_m)\|_{L^4}^4 + \frac{1}{2}|u(t \land \sigma_m)|^2_{H^1} + \|\grad \Delta u-3u^2 \grad u\|_{L_t^2L_x^2}^2 \nonumber \\
        & \quad \le\frac{1}{4}\|u_0\|_{L^4}^4 + \frac{1}{2}|u_0|^2_{H^1}+ (1+C_{\psi,\theta}) \|\Delta u\|_{L_t^2L_x^2}^2 
         + (3+C_\psi)\|u\grad u\|_{L_t^2L_x^2}^2 + M^{(1)}_{t\land \sigma_m} + M^{(2)}_{t\land\sigma_m}\nonumber \\
         &\qquad +C_\psi|\sO|(t\land \sigma_m)+ C_{\psi.\theta} \|u\|_{L_t^2L_x^2}^2+ C_\psi \inttm \Big(\frac{1}{4}\|u\|_{L^4}^4+\frac{1}{2}|u|_{H^1}^2\Big)  \ds.
    \end{align}
Applying the stochastic Gronwall inequality~\cite[Corollary 5.4]{Geiss24} to \eqref{eq:proof-thm-global-well-posedness-dim4-eq1} as was done for \eqref{eq:lemL2energyEstimate-martingaleEq} and \eqref{eq:Proof_global_well_posedness_dim_3-gronwall}, choosing hereby $X(t) \ce \frac{1}{4}\|u(t\land \sigma_m)\|_{L^4}^4 + \frac{1}{2}|u(t \land \sigma_m)|^2_{H^1} + \|\grad \Delta u-3u^2 \grad u\|_{L_t^2L_x^2}^2$, implies that the supremum of the $H^1$-seminorm as well as $\|\grad \Delta u - 3u^2\grad u\|_{L_t^2L_x^2}$ are a.s.\ finite. Combined with the $L^2$-energy estimate \eqref{eq:L2energyEstimateAS} from Lemma~\ref{lem:L2energyEstimate}, it follows that the first term in \eqref{eq:prop-global_well_posedness_dim_4-claim1} is a.s.\ finite. Note that $\grad \mu = -\grad\Delta u +3u^2\grad u-\grad u$ and $\|\grad u\|^2_{L_t^2L_x^2}\le (t\land \sigma_m)\sup_{s\in [0,t\land \sigma_m]} |u(s)|_{H^1}^2 <\infty$ a.s.\ by Hölder's inequality. By a triangle inequality and a.s.\ finiteness of $\|\grad \Delta u - 3u^2\grad u\|_{L_t^2L_x^2}$, the gradient of the chemical potential is seen to be finite a.s., too. This yields \eqref{eq:prop-global_well_posedness_dim_4-claim1}.

Now suppose that in addition $u_0\in L_{\F_0}^4(\Omega;L^4(\sO))\cap L_{\F_0}^2(\Omega;H^1(\sO))$.  
Define as before for $T_0 \in[0,T]$
\begin{equation*}
\varphi(T_0)\ce \E\bigg[\sup_{t\in[0,T_0\wedge \sigma_m]}  \Big(\frac{1}{4}\|u(t)\|_{L^4}^4 +\frac{1}{2}|u(t)|^2_{H^1}\Big)\bigg].
\end{equation*}
By the same calculations as done in Step 3 of the case $d=3$, with the only difference that we do not discard the non-positive term but instead retain it on the left-hand side, we obtain for all $T_0 \in [0,T]$: 
   \begin{align*}
        & \E\bigg[\frac{1}{2} \sup_{t\in[0,T_0\wedge \sigma_m]}  \Big(\frac{1}{4}\|u(t)\|_{L^4}^4 +\frac{1}{2}|u(t)|^2_{H^1}\Big) + \int_0^{\sigma_m \land T_0} \|\grad \Delta u-3u^2 \grad u\|_{L^2}^2 \dt \bigg]\\
        &\quad \le \frac{1}{4} \E\big[\|u_0\|^{4}_{L^4}\big] + \frac{1}{2}\E\big[|u_0|^2_{H^1}\big]+ C_{\psi,T,\sO}\big(\E\big[\|u_0\|_{L^2}^2\big]+1\big)+ C_\psi \int_0^{T_0} \varphi(s) \ds.
    \end{align*}
Noting $\grad \mu = -\grad\Delta u +3u^2\grad u-\grad u$, adding  $\E[\int_0^{\sigma_m \land T_0} \|\grad u\|_{L^2}^2 \dt]$ on both sides, applying the triangle inequality and using $\varphi(t) \leq \tilde{\varphi}(t)$ implies for all $T_0 \in [0,T]$
   \begin{align*}
        \tilde{\varphi}(T_0) \ce & \frac{1}{2}\E\bigg[\sup_{t\in[0,T_0\wedge \sigma_m]}  \Big(\frac{1}{4}\|u(t)\|_{L^4}^4 +\frac{1}{2}|u(t)|^2_{H^1}\Big) + \int_0^{\sigma_m \land T_0} \|\grad \mu\|_{L^2}^2 \dt \bigg]\\
        &\quad \le \frac{1}{4} \E\big[\|u_0\|^{4}_{L^4}\big] + \frac{1}{2}\E\big[|u_0|^2_{H^1}\big]+ C_{\psi,T,\sO}\big(\E\big[\|u_0\|_{L^2}^2\big]+1\big)+ (C_\psi+2) \int_0^{T_0} \tilde{\varphi}(s) \ds.
    \end{align*}
The assertion \eqref{eq:prop-global_well_posedness_dim_4-claim2} follows by applying the deterministic Gronwall inequality and taking the limit $m\to\infty$.
\end{proof}

\begin{lemma} \label{lem:globalWP_d4_claims}
	Let $d=4$ and suppose that the assumptions of Theorem~\ref{thm:localWP_weakened} hold. Moreover, let $u_0\in L^0_{\F_0}(\Omega;H^1(\sO))$ and $T>0$. Let $(u,\sigma)$ be the unique maximal solution to \eqref{eq:SCHtorus} and $\mu$ the associated chemical potential from Definition~\ref{def:chemPotential}. Then the following three estimates hold a.s.\ for a.e.\ $0\le s <\sigma \land T$
    \begin{align}
        \|\Delta u(s)\|_{L^2}^2  & \le \|\mu(s)\|_{L^2}^2 + 2\| \grad u(s)\|_{L^2}^2 \label{eq:globalWP_d4_eq1}, \\
         \|u^2(s)\nabla u(s)\|_{L^2}^2 & \lesssim
        \|u(s)\|^3_{H^1}
        \left(\|\Delta u(s)\|_{L^2}^2
        +\|u(s)\|^2_{H^1}  \right)
        \left( \|\grad \Delta u(s)\|_{L^2}
        +\|u(s)\|_{H^1}\right),
        \label{eq:globalWP_d4_eq2} \\
        \|\mu(s)\|_{L^2}^4 & \lesssim \|\nabla\mu(s)\|_{L^2}^2 \left(\|u(s)\|_{H^1}^6 + 1\right) + 1
        + \|u(s)\|_{H^1}^{12}. \label{eq:globalWP_d4_eq3}
    \end{align}
\end{lemma}
\begin{remark}
Lemma~\ref{prop:global_well_posedness_dim_4-part1} and inequality \eqref{eq:globalWP_d4_eq3} of Lemma~\ref{lem:globalWP_d4_claims} imply that
\begin{equation*}
\int_0^{\sigma\wedge T} \|\mu(s)\|^4_{L^2} \ds<\infty \text{ almost surely,}
\end{equation*}
which allows us to avoid bounding
$\sup_{s\in[0,\sigma)\wedge[0,T]}\|\mu(s)\|_{L^2}$ in order to bound $\int_0^{\sigma\wedge T} \|u(s)\|^2_{H^3} \ds$. See the proof of Lemma~\ref{prop:global_well_posedness_dim_4-part2} for further details.
\end{remark}
\begin{proof}
Let $(\sigma_m)_{m\ge 1}$ be a localising sequence of $(u,\sigma\wedge T)$. Rearranging Definition~\ref{def:chemPotential} of the chemical potential $\mu$ gives
\begin{equation}\label{eq:muRearranged}
\Delta u = -\mu + u^3 - u,
\end{equation}
where we dropped the $s$ in the notation here and in the following. 
Applying \eqref{eq:muRearranged}, Young's inequality, integration by parts, and $u(t)\in \nH^3(\sO)$ for a.e.\ $t\in[0,\sigma)$ yields
\begin{equation*}
    \|\Delta u\|_{L^2}^2 
    =\int_{\sO} (\Delta u)(-\mu + u^3 - u) \dx
    \leq \frac12\|\Delta u\|_{L^2}^2 + \frac12\|\mu\|_{L^2}^2 - \int_{\sO} 3u^2|\nabla u|^2 \dx 
    +\|\nabla u\|_{L^2}^2
\end{equation*}
and rearranging the terms implies \eqref{eq:globalWP_d4_eq1} via
\begin{equation*}
    \|\Delta u\|_{L^2}^2
    \le \|\mu\|_{L^2}^2 - 6\|u\nabla u\|_{L^2}^2 + 2\|\nabla u\|_{L^2}^2
    \le \|\mu\|_{L^2}^2 + 2\|\grad u\|_{L^2}^2.
\end{equation*}

We proceed to show \eqref{eq:globalWP_d4_eq2}. From the elliptic regularity estimates in~\cite[Chapter~5.3.4]{triebelInterpolationTheoryFunction1978} with $m=1$, $p=2$, and $s\in\{0,1\}$, we obtain
\begin{align} \label{eq:ellReg_H2_seminorm}
    |u|_{H^2} &\lesssim \|\Delta u\|_{L^2}+\|u\|_{L^2},\\
    |u|_{H^3} &\lesssim \|\Delta u\|_{H^1}+\|u\|_{L^2} \lesssim \|\grad \Delta u\|_{L^2}+\|\Delta u\|_{L^2}+\|u\|_{L^2},\nonumber
\end{align}
where the final inequality uses the definition of the $H^1$-norm.
Moreover, by the divergence theorem, 
\begin{equation}
\label{eq:proofDivergenceThm}
    \int_{\sO} \Delta u(s) \dx=\int_{\partial\sO} \partial_{\mathbf{n}}u(s) \,\rmd S=0\text{ for a.e.\ }0 \le s <\sigma \land T \text{ a.s.\ }
\end{equation}
because $u(s)\in \nH^3(\sO)$. Hence, the Poincaré--Wirtinger inequality yields that $\|\grad \Delta u\|_{L^2} \gtrsim \|\Delta u\|_{L^2}$. Thus, the $H^3$-seminorm estimate simplifies to
\begin{align}
    |u|_{H^3} &\lesssim \|\grad \Delta u\|_{L^2}+\|\Delta u\|_{L^2}+\|u\|_{L^2} \lesssim \|\grad \Delta u\|_{L^2}+\|u\|_{L^2}.\label{eq:ellReg_H3_seminorm}
\end{align}
By Hölder's inequality, the Gagliardo-Nirenberg interpolation
inequality~\cite{Nirenberg59}, the elliptic regularity estimates \eqref{eq:ellReg_H2_seminorm} and \eqref{eq:ellReg_H3_seminorm} and the embedding $H^1(\sO) \hookrightarrow L^4(\sO)$, we obtain \eqref{eq:globalWP_d4_eq2} via
\begin{align*}
    \|u^2\nabla u\|_{L^2}^2
    & \leq \|u\|_{L^8}^4 \|\nabla u\|_{L^4}^2  
    \lesssim \left( \|u\|_{L^4}^{1/2} |u|_{H^2}^{1/2} + \|u\|_{L^2} \right)^4 \left( \|\nabla u\|_{L^2}^{1/2} |u|_{H^3}^{1/2} + \|\grad u\|_{L^2} \right)^2 \\
    &\lesssim\left( \|u\|_{L^4}^{1/2} \big(\|\Delta u\|_{L^2}+\|u\|_{L^2}\big)^{1/2} + \|u\|_{L^2} \right)^4
    \left(\|\nabla u\|_{L^2}^{1/2}\big(\|\grad \Delta u\|_{L^2}+\|u\|_{L^2}\big)^{1/2} + \|\grad u\|_{L^2} \right)^2 \\
    & \lesssim \|u\|^3_{H^1} \left(\|\Delta u\|_{L^2}^2 +\|u\|^2_{L^2}\right) \left( \|\grad \Delta u\|_{L^2} + \|u\|_{L^2}\right).
\end{align*}
Note that this is slightly stronger than \eqref{eq:globalWP_d4_eq2} because the lower-order terms in the upper bound involve $L^2$ rather than the $H^1$-norms. However, this does not improve the global well-posedness argument.

The Poincaré--Wirtinger inequality 
combined with \eqref{eq:proofDivergenceThm} and $H^1(\sO) \hra L^3(\sO)\hra L^1(\sO)$ yields
\begin{equation*}
    \|\mu\|_{H^1}  \lesssim \|\nabla\mu\|_{L^2} + \Big| \int_{\sO }\mu \dx \Big|
    = \|\nabla\mu\|_{L^2}+\Big|\int_{\sO }u^3 - u\Big| \lesssim \|\nabla\mu\|_{L^2} +\|u\|_{L^3}^3+\|u\|_{L^1}
    \lesssim \|\nabla\mu\|_{L^2} +1+\|u\|_{H^1}^{3}.
\end{equation*}
Moreover, integration by parts, Hölder's inequality, and  $H^1(\sO) \hookrightarrow L^4(\sO) \hra L^{4/3}(\sO)$ imply
\begin{equation*}
    \|\mu\|_{(H^1)^*} \le  \sup_{v\in H^1} \frac{\int_{\sO} \nabla u \cdot \nabla v\dx }{\|v\|_{H^1}} + \sup_{v\in H^1} \frac{\int_{\sO}(u^3-u)v\dx }{\|v\|_{H^1}}
    \lesssim \|\nabla u\|_{L^2} + \|u^3\|_{L^{4/3}} + \|u\|_{L^{4/3}}
    \lesssim \|u\|_{H^1}^3 + 1,
\end{equation*}
where $(H^1)^*$ denotes the $L^2$-dual of $H^1(\sO)$. An interpolation inequality combined with the previous two inequalities gives \eqref{eq:globalWP_d4_eq3} via
\begin{equation*}
    \|\mu\|_{L^2}^4
    \leq \|\mu\|^2_{(H^1)^*}\|\mu\|^2_{H^1}
    \lesssim \left( \|u\|_{H^1}^3 + 1 \right)^2 \left( \|\nabla\mu\|_{L^2} + 1 + \|u\|_{H^1}^{3} \right)^2
    \lesssim \|\nabla\mu\|_{L^2}^2 \left(\|u\|_{H^1}^6+1\right)+1+\|u\|_{H^1}^{12}. \qedhere
\end{equation*}
\end{proof}

\begin{proposition} \label{prop:global_well_posedness_dim_4-part2}
	Let $d=4$ and suppose that the assumptions of Theorem~\ref{thm:localWP_weakened} hold. Moreover, let $u_0\in L^0_{\F_0}(\Omega;H^1(\sO))$ and $T>0$. Let $(u,\sigma)$ be the unique maximal solution to \eqref{eq:SCHtorus}. Then we have
	\begin{equation}\label{eq:thm-global_well_posedness_dim_4-claim1}
	  \int_0^{\sigma\wedge T} \|u\|^2_{H^3} \dt < \infty  \qquad \mathbb{P}\text{-a.s.}
	\end{equation}
\end{proposition}

\begin{proof}
It suffices to show $\int_0^{\sigma\wedge T} \|\grad \Delta u\|^2_{L^2} \dt <\infty$ a.s., which can be seen as follows. By the elliptic regularity estimates \eqref{eq:ellReg_H2_seminorm} as well as \eqref{eq:ellReg_H3_seminorm} and absorbing $\|\Delta u\|_{L^2}$ into $\|\grad\Delta u\|_{L^2}$ via \eqref{eq:proofDivergenceThm} and the Poincaré--Wirtinger inequality as before, we obtain
\begin{align}\label{eq:thm:globalWP_d4_H3norm}
    \|u\|_{H^3}^2 = |u|_{H^3}^2+|u|_{H^2}^2 +\|u\|_{H^1}^2
    \lesssim \|\grad\Delta u\|_{L^2}^2+\|\Delta u\|_{L^2}^2+\|u\|_{H^1}^2 \lesssim \|\grad\Delta u\|_{L^2}^2+\|u\|_{H^1}^2.
\end{align} 
Proposition~\ref{prop:global_well_posedness_dim_4-part1} ensures that $\sup_{t\le \sigma \land T} \|u(t)\|_{H^1}<\infty$ a.s. Hence, the integral in time of the second 
term in \eqref{eq:thm:globalWP_d4_H3norm} 
is almost surely finite, so that it suffices to control the integral of $\|\grad\Delta u\|_{L^2}^2$. 

Let $\sigma_m$ be a localising sequence of $(u,\sigma\wedge T)$. We show that $\int_0^{\sigma\wedge T} \|\grad \Delta u\|^2_{L^2} \dt <\infty$ a.s.\ follows from the three estimates \eqref{eq:globalWP_d4_eq1}, \eqref{eq:globalWP_d4_eq2}, and \eqref{eq:globalWP_d4_eq3} of Lemma~\ref{lem:globalWP_d4_claims}.
We use the chemical potential via \eqref{eq:muRearranged} in the first inequality, \eqref{eq:globalWP_d4_eq2} in the second inequality and Young's inequality in the third inequality. We apply \eqref{eq:globalWP_d4_eq1} in the fourth inequality and
\eqref{eq:globalWP_d4_eq3} in the fifth inequality. This results in
\begin{align*}
    & \int_0^{\sigma_m} \|\nabla\Delta u\|_{L^2}^2\dt 
     \le 3\int_0^{\sigma_m} \|\nabla\mu\|_{L^2}^2
    + 9\|u^2\nabla u\|_{L^2}^2
    + \|\nabla u\|_{L^2}^2\dt  \\
    & \le 
    3\int_0^{\sigma_m}
    \|\nabla\mu\|_{L^2}^2
    +
    C\|u\|^3_{H^1}
    \left(\|\Delta u\|_{L^2}^2
    +\|u\|^2_{H^1}  \right)
    \left( \|\grad \Delta u\|_{L^2} +\|u\|_{H^1}\right)+
    \|\nabla u\|_{L^2}^2\dt\\
    & \le  
    3\int_0^{\sigma_m}
    \|\nabla\mu\|_{L^2}^2 \dt +
    C\int_0^{\sigma_m}(1+\|u\|_{H^1}^{10})\dt
    +
    \frac12
    \int_0^{\sigma_m}
    \|\nabla\Delta u\|_{L^2}^2\dt
    +
    C\int_0^{\sigma_m}
    \|u\|_{H^1}^6
    \|\Delta u\|_{L^2}^4\dt  \\
    & \le
    3\int_0^{\sigma_m}
    \|\nabla\mu\|_{L^2}^2 \dt +
    C\int_0^{\sigma_m}(1+\|u\|_{H^1}^{10})\dt
    +
    \frac12
    \int_0^{\sigma_m}
    \|\nabla\Delta u\|_{L^2}^2\dt
    +
    C\int_0^{\sigma_m}
    \|u\|_{H^1}^6
    \|\mu\|_{L^2}^4\dt \\
    & \le
    3\int_0^{\sigma_m}
    \|\nabla\mu\|_{L^2}^2 \dt +
    C\int_0^{\sigma_m}(1+\|u\|_{H^1}^{18})\dt
    +
    \frac12
    \int_0^{\sigma_m}
    \|\nabla\Delta u\|_{L^2}^2\dt
     + C\Big(1+ \sup_{t\leq \sigma_m} \|u\|^{12}_{H^1} \Big)
    \int_0^{\sigma_m}
    \|\nabla\mu\|_{L^2}^2\dt.
 \end{align*}
Here, the constant $C$ changes from line to line. By Proposition~\ref{prop:global_well_posedness_dim_4-part1}, both $\sup_{t\le \sigma_m} \|u(t)\|_{H^1}<\infty$ a.s.\ and $\int_0^{\sigma_m} \|\grad \mu\|_{L^2}^2\dt<\infty$ a.s. Rearranging the terms and bounding all polynomial $H^1$ terms by the pathwise supremum implies the claim.
\end{proof}

\begin{theorem}[Global well-posedness in $d=4$] \label{thm:global_well_posedness_dim_4-complete}
	Let $d=4$, suppose that the assumptions of Theorem~\ref{thm:localWP_weakened} hold, and let $u_0\in L^0_{\F_0}(\Omega;H^1(\sO))$. Then there exists a unique global solution $u$ to \eqref{eq:SCHtorus} such that $u\in C([0,\infty);H^1(\sO))\cap L_{\loc}^2(0,\infty;\nH^3(\sO))$ almost surely. If, in addition, $u_0\in L_{\F_0}^4(\Omega;L^4(\sO))\cap L_{\F_0}^2(\Omega;H^1(\sO))$, then \eqref{eq:prop-global_well_posedness_dim_4-claim2} holds for all $T>0$. 
\end{theorem}

\begin{proof}
 The claim follows from combining the blow-up criterion from Theorem~\ref{thm:localWP_weakened} with the energy estimates from Proposition~\ref{prop:global_well_posedness_dim_4-part1} and Proposition~\ref{prop:global_well_posedness_dim_4-part2}.
\end{proof}

\renewcommand{\theenumi}{(\alph{enumi})}

\section{Rough initial data}\label{section_rough_initial_data}

In this section we go beyond the $L^2$-setting and employ a more flexible $L^p(L^q)$-theory to study the Cahn--Hilliard equation with  initial data that are rougher than those admitted by the $L^2$-theory. 
A deterministic $L^p(L^q)$-theory for the Cahn--Hilliard equation in critical spaces was developed in~\cite{prussCriticalSpacesQuasilinear2018}. In particular, for the double-well potential, the critical space for the initial data was identified as $_\nu B^{d/q-1}_{q,p}(\sO)$. The stochastic counterpart was subsequently derived in~\cite[Section 7.3]{agrestiNonlinearParabolicStochastic2022}.

The purpose of this section is to obtain a more flexible local well-posedness theory, which we will connect to the $L^2$-setting dealt with in the sections above. To be more precise, by combining our analysis in this section with Theorem~\ref{thm:GWP:weak:setting}, Theorem~\ref{thm:global_well_posedness_dim_3} and Theorem~\ref{thm:global_well_posedness_dim_4-complete} above, we obtain new global well-posedness results in various settings. A similar strategy has been carried out for reaction-diffusion equations in~\cite{agresti2023reactiondiffusion_local}.

Let $a\in [0,\infty)$. The stochastic evolution equation considered in this section is
\begin{equation}\tag{$\text{SCH}_{p, \kappa}$}\label{SCH}
    \left\{
    \begin{aligned}
        \dd u +\,_\nu\Delta_{\alpha,q}^2 u\, \dd t&= \,_\nu\Delta_{\alpha,q}(u^3-u)\,\dd t +\sum_{n\ge1}\psi_n(u,\nabla u)\, \dd W^n(t), &t\ge a\\
        u(a)&=u_a&
    \end{aligned}
    \right.
\end{equation}
for certain choices of $\alpha$ and $q$. 
Compared to Section~\ref{sec:variational:lwp} and Section~\ref{sec:globalWP} above, we allow for a slightly more general form of the noise term.
In addition, we shall work with several sets of assumptions on $\psi$, which lead to different regularity results. 
	Throughout this section, we will distinguish among several sets of assumptions.
	\begin{assumption}[General assumptions]\label{Assumption:general}
     $\sO \seq \R^d,d\in\N$, is a bounded $C^\infty$ domain.
		Assume that either
		\begin{itemize}
		\item 	$p\in (2,\infty), \kappa\in [0,p/2-1), q\in [2,\infty)$, $s\in [0,2)$
		\end{itemize}
		or 
		\begin{itemize}
			\item $p=q=2$, $s=\kappa=0$. 
		\end{itemize}
	\end{assumption}
	\begin{assumption}[Assumptions on $\psi$]\label{Assumption:noise}
Let $d\in \N$, let $\psi^0\ce (\psi^0_n)_{n\ge 1}\colon [0,\infty) \times\Omega\times\sO\times \R\to \ell^2$ be $\mathscr{P}\otimes\mathscr{B}(\sO)\otimes \sB(\R)$-measurable, and let $\psi^1\ce (\psi^1_n)_{n\ge 1}, \psi^2\ce (\psi^2_n)_{n\ge 1}\colon [0,\infty)\times\Omega\times\sO \to \ell^2(\R^d)$ be $\sP\otimes \sB(\sO)$-measurable. 
\begin{equation}
    \psi(\cdot, y,z)\ce(\psi_n(\cdot,y,z))_{n\ge 1}\ce \psi^0(\cdot, y)+\big( \psi^1+ y\psi^2\big)\cdot z, \quad y\in \R, z\in \R^d.
\end{equation}
Moreover, $\psi^0(\cdot,0)\in L^\infty([0,\infty)\times\Omega;L^q(\sO;\ell^2))$. There exists a constant $L>0$ such that a.e.\ on $[0,\infty)\times\Omega\times\sO$ for all $y,y'\in \R$ we have 
            \begin{equation}\label{eq:assumption:hg0:noise}
                \nrm{\psi^0(\cdot,y)-\psi^0(\cdot,y')}_{\ell^2}\le L(1+|y|^2+|y'|^2)|y-y'|.
            \end{equation}
            We additionally distinguish among the following sets of assumptions:
\begin{enumerate}
			\item\label{assumption:hg1} 
            It holds that $\psi^1=\psi^2=0$.
            \item\label{assumption:hg2} It holds that $\psi^1, \psi^2 \in L^\infty([0,\infty)\times\Omega\times\sO;\ell^2(\R^d)).$
            \item\label{assumption:hg3}  
            We have $d\in\{3,4\}$. It holds that $\psi^1=\psi^2=0.$
            For almost all $(t,\omega)$ we have $\psi=\psi^0\in C^1(\sO\times \R;\ell^2).$ 
            Moreover, a.e.\ on $[0,\infty)\times\Omega\times\sO$ for all $y,y'\in \R$ and all $j\in\{1,...,d\}$ we have  
             \begin{align}
             \nrm{\psi^0(\cdot,y)}_{\ell^2}+ \nrm{\partial_{x_j}\psi^0(\cdot, y)}_{\ell^2}&\le L(1+|y|^3),\label{Assumption:hg3:noise:0}\\
                \nrm{(\partial_{x_j}\psi^0)(\cdot,y)-(\partial_{x_j}\psi^0)(\cdot,y')}_{\ell^2}&\le L  (1+|y|^2+|y'|^2)|y-y'|,\label{Assumption:hg3:noise:2}\\
                \nrm{(\partial_y\psi^0)(\cdot,y)-(\partial_y\psi^0)(\cdot,y')}_{\ell^2}&\le L  (1+|y|^{1/2}+|y'|^{1/2})|y-y'|,\label{Assumption:hg3:noise:3}\\
                \nrm{\partial_y\psi^0(\cdot,y)}_{\ell^2}&\le L(1+|y|^{3/2})\label{Assumption:hg3:noise:4}.
             \end{align}
    \item\label{assumption:hg4}
    We have $d\in \{3,4\}$. 
    For almost all $(t,\omega)$ we have $\psi^0\in C^1(\sO\times \R;\ell^2)$  
            and $ \psi^1,\psi^2\in C^1(\sO;\ell^2(\R^d))$. We assume that~\eqref{Assumption:hg3:noise:0}, \eqref{Assumption:hg3:noise:2}, \eqref{Assumption:hg3:noise:3}, and~\eqref{Assumption:hg3:noise:4} hold. 
		\end{enumerate}
	\end{assumption}
    \begin{remark}
        The reason we distinguish between the various different assumptions is that we  analyse~\eqref{SCH} on different levels of Sobolev smoothness. Assumption~\ref{Assumption:noise}\ref{assumption:hg1} and~\ref{assumption:hg2} serve for the analysis on less regular levels of Sobolev smoothness than Assumption~\ref{Assumption:noise}\ref{assumption:hg3} and~\ref{assumption:hg4}, respectively. 
    \end{remark}
	\subsection{Local well-posedness}
    Our next result proves local well-posedness for~\eqref{SCH} in the different settings in Assumption~\ref{Assumption:noise}. It also singles out the subcriticality conditions under which the expected critical space is reproduced. 
	\begin{theorem}[Local well-posedness]\label{lwp:rough:initial:data}
		Let $d\in \N, s\in \R$, and assume that one of the following holds:
		\begin{enumerate}
			\item[i)] $p\in (2,\infty)$, $\kappa\in [0,p/2-1)$, and $q\in [2,\infty)$.
			\item[ii)] $p=q=2$ and $s=\kappa=0$.
		\end{enumerate} 
		Assume that one of the following settings holds:
		\begin{enumerate}
			\item\label{it:lwp:rough:initial:data:1} $X_0=\bpsn{-2-s,q}$, $X_1=\bpsn{2-s,q}, s\in [0,2\wedge d)$, and Assumption~\ref{Assumption:noise}\ref{assumption:hg1} holds.
			Moreover, assume the (sub)criticality condition
		\begin{equation}\label{eq:lwp:rough:initial:data:1}
			\frac{1+\kappa}{p}\le \frac{3}{4}-\frac{1}{4}\left(s+\frac{d}{q}\right), 
		\end{equation}
		and assume that, if $\psi^2=0$,
		\begin{equation}\label{eq:lwp:rough:initial:data:2}
			\max\left\{\frac{d}{3-s},\frac{d}{d-s}\right\}< q <2\frac{d}{s},
		\end{equation}
        otherwise assume 
        \begin{equation}
            \max\left\{\frac{d}{3-s},\frac{d}{d-s}\right\}<q<\frac{d}{s}.
        \end{equation}
        \item\label{it:lwp:rough:initial:data:2} $X_0=\bpsn{-2-s,q}, X_1=\bpsn{2-s,q}, s\in [0,1)$, \eqref{eq:lwp:rough:initial:data:1}, \eqref{eq:lwp:rough:initial:data:2}, and Assumption~\ref{Assumption:noise}\ref{assumption:hg2} holds.
			\item\label{it:lwp:rough:initial:data:3} $X_0=\bpsn{-1-s,q}, X_1=\bpsn{3-s,q}, s\in [0,2)$, and Assumption~\ref{Assumption:noise}\ref{assumption:hg3} holds. Moreover, assume the (sub)criticality condition
            \begin{equation}\label{eq:lwp:rough:initial:data:3}
                \frac{1+\kappa}{p}\le 1-\frac{1}{4}\left(s+\frac{d}{q}\right),
            \end{equation}
            and assume that 
            \begin{equation}\label{eq:lwp:rough:initial:data:4}
                \max\left\{\frac{d}{4-s},\frac{d}{d-s}\right\}<q<\frac{d}{1+s/2}.
            \end{equation}
            \item\label{it:lwp:rough:initial:data:4}
             $X_0=\bpsn{-1-s,q}$, $X_1=\bpsn{3-s,q}$,
            $s\in[0,1)$, \eqref{eq:lwp:rough:initial:data:3}, \eqref{eq:lwp:rough:initial:data:4}, and Assumption~\ref{Assumption:noise}\ref{assumption:hg4}
            holds.
		\end{enumerate}
		Then for all $u_a\in L^0_{\F_a}(\Omega; X_{1-\frac{1+\kappa}{p},p})$ there exists a maximal solution $(u,\sigma)$ to \eqref{SCH} with $\sigma>a$ a.s.\ Moreover, for each localising sequence $(\sigma_n)_{n\ge 1}$ for $(u,\sigma)$, one has 
		\begin{itemize}
			\item if $p>2$ and $\kappa\in [0,p/2-1)$, then for all $n\in \N$ and all $\theta\in [0,1/2)$
				\[u\in H^{\theta,p}((a,\sigma_n),w^a_\kappa; X_{1-\theta})\cap C([a,\sigma_n];X_{1-\frac{1+\kappa}{p},p})\quad \text{a.s.}\]
			and we have the instantaneous regularity
			\[u\in H_{\loc}^{\theta,p}((a,\sigma), w^a_\kappa; X_{1-\theta})\cap C((a,\sigma); X_{1-\frac{1}{p},p})\quad \text{a.s.}\]
			\item if $p=2$, then for all $n\in \N$
			\[u\in L^2(a,\sigma_n;X_1)\cap C([a,\sigma_n];X_{1/2})\quad \text{a.s.}\]
		\end{itemize}
		In addition, if equality holds in~\eqref{eq:lwp:rough:initial:data:1} (or in~\eqref{eq:lwp:rough:initial:data:3}) the critical space is $ X_{p,\kappa}^{\textup{Tr}}= \,_\nu B_{q,p}^{\frac{d}{q}-1}(\sO)$. We call $(u,\sigma)$ the $(p,\kappa,s,q)$-solution to \eqref{SCH}.
	\end{theorem}
    \begin{remark} The following remarks are in order. 
        \begin{enumerate}
            \item Note that $(p,\kappa,s,q)$-solutions are, by Definition~\ref{def:solution:local} of a maximal solution, unique. Any $(p,\kappa,s,q)$-solution depends on $X_0$ and $X_1$. In order not to overburden the notation, we omit this dependence. 
            \item Let us comment on the admissible parameters in Theorem~\ref{lwp:rough:initial:data}. Note that $p=q=2$ and $s=\kappa=0$ are admissible choices in~\ref{it:lwp:rough:initial:data:1} and~\ref{it:lwp:rough:initial:data:2}, provided $d\le 2$. This corresponds to the setting in Section~\ref{weak_setting}. Similarly, the same choices are admissible in~\ref{it:lwp:rough:initial:data:3} and~\ref{it:lwp:rough:initial:data:4}, provided $d\le 4$. This corresponds to Section~\ref{weakened_setting}. 
            In Corollary~\ref{cor:GWP:rough} below we prove that these settings actually coincide, i.e.\ $_\nu\Delta_{-2,2}^2$ and $_\nu\Delta_{-1,2}^2$ coincide with the operator $A$ defined in Section~\ref{weak_setting} and Section~\ref{weakened_setting}, respectively. 

            The parameter $s$ controls the roughness of the settings. In the cases where the nonlinearity $\psi$ does not depend on $\nabla u$, the parameter $s$ can be chosen larger than in those cases where gradient dependence is present. For larger choices of $s$, the admissible range of $q$ becomes smaller. The subcriticality conditions in~\eqref{eq:lwp:rough:initial:data:1} and~\eqref{eq:lwp:rough:initial:data:3} become more restrictive for larger choices of $s$ and $d$. However, the lower bounds for $q$ in~\eqref{eq:lwp:rough:initial:data:2} and~\eqref{eq:lwp:rough:initial:data:4} ensure that for any  admissible choice of $s$ and $d$  there exists an admissible choice for $p\in (2,\infty)$ and $\kappa \in[0,p/2-1)$. 
        \end{enumerate}
    \end{remark}
    \begin{remark}
    We chose the assumptions on $\psi$ so that the abstract framework for local well-posedness employed in the proof of Theorem~\ref{lwp:rough:initial:data} reproduces the critical space $_\nu B^{d/q-1}_{q,p}(\sO)$ known from the literature, cf.~\cite[Section 7]{agrestiNonlinearParabolicStochastic2022} or~\cite{agrestiNonlinearSPDEsMaximal2025}.
    If one requires $_\nu B^{d/q-1}_{q,p}(\sO)=(\bpsn{\alpha,q},\bpsn{\alpha+4,q})_{\beta,p}$ across a range of $\alpha,\beta\in \R$, this means as $\alpha$ increases, $\beta$ needs to be chosen smaller. 
    This is why in Theorem~\ref{lwp:rough:initial:data} the (sub-)criticality conditions in~\eqref{eq:lwp:rough:initial:data:1} and~\eqref{eq:lwp:rough:initial:data:3} differ from each other. 
    This in turn means that for the estimates of the nonlinear terms, finer nonlinear estimates are required.
    
        In general, if one is only interested in local well-posedness for \emph{some} critical space, it is possible to choose less restrictive conditions for the noise term.
        However, this would imply that Theorem~\ref{lwp:rough:initial:data} does not reproduce the critical spaces one expects from a scaling analysis of the Cahn--Hilliard equation. 
    \end{remark}
	\begin{proof}
     Throughout the proof, set $r\ce d(s+\frac{d}{q})^{-1}$. The assumption that $\frac{d}{d-s}<q$ ensures that $r>1$, and the sharp Sobolev embedding $\bpsn{k,r}\hookrightarrow \bpsn{k-s,q}$,  $k\in \{0,1\}$.
     
		First, the case~\ref{it:lwp:rough:initial:data:1} follows from~\cite[Theorem 7.9]{agrestiNonlinearParabolicStochastic2022} and the lines preceding it. 

    For the cases~\ref{it:lwp:rough:initial:data:2},~\ref{it:lwp:rough:initial:data:3}, and~\ref{it:lwp:rough:initial:data:4} proven below, note that by  Remark~\ref{rem:assumption:semilinear:evolution:equation} it is enough to verify that Assumption~\ref{Assumption:semilinear:evolution:equation} holds in order to obtain the result from~\cite[Theorem~4.3]{agrestiNonlinearParabolicStochastic2022a}.

        Next, we prove case~\ref{it:lwp:rough:initial:data:3}. To this end, we let $r$ be as above, which ensures that $\bpsn{1,r}\hookrightarrow\bpsn{1-s,q} $, which hence ensures that 
        $$ \nrm{ u^3-v^3}_{1-s,q}\lesssim \nrm{u^3-v^3}_{1,r}\lesssim \nrm{u^3-v^3}_r+\nrm{\nabla(u^3-v^3)}_r=:\nrm{I_1}_r+\nrm{I_2}_{r}.$$
        We first estimate the second term. To this end, note that 
        $$\begin{aligned}
            \nabla (u^3-v^3)&= \nabla\big[(u^2+uv+v^2)(u-v)\big]\\
            &=(u-v)\nabla(u^2+uv+v^2) + (u^2+uv+v^2)\nabla(u-v)\\
            &=(u-v)\big( 2u\nabla u + u\nabla v +v\nabla u +2v \nabla v \big) + (u^2+uv+v^2)\nabla (u-v)\\
            &\eqqcolon I_{21}+I_{22}.
        \end{aligned}$$
        To estimate $I_{21}$, we employ Hölder's inequality with $\frac1r=\frac{2}{r_1}+\frac{1}{r_2}$, which gives 
        \[\begin{aligned}
            \nrm{I_{21}}_{r}&\le \nrm{u-v}_{r_1}\big(\nrm{2u}_{r_1}\nrm{\nabla u}_{r_2}+ \nrm{u}_{r_1}\nrm{\nabla v}_{r_2}+ \nrm{v}_{r_1}\nrm{\nabla u}_{r_2}+ \nrm{2v}_{r_1}\nrm{\nabla v}_{r_2}\big)\\
            &\lesssim \nrm{u-v}_{r_1}\big(\nrm{u}_{r_1}\nrm{u}_{1,r_2}+ \nrm{u}_{r_1}\nrm{ v}_{1,r_2}+ \nrm{v}_{r_1}\nrm{ u}_{1,r_2}+ \nrm{v}_{r_1}\nrm{ v}_{1,r_2}\big).
        \end{aligned}\]
        Next, we choose $r_1$ and $r_2$ such that the Sobolev indices of $L^{r_1}(\sO)$ and $\bpsn{1,r_2}$ coincide, i.e.\ that $\bpsn{1,r_2} \hookrightarrow L^{r_1}$ sharply. This means that in addition to $\frac{1}{r}=\frac{2}{r_1}+\frac{1}{r_2}$, we impose $-\frac{d}{r_1}=1-\frac{d}{r_2}$, which yields that $r_1=\frac{dr_2}{d-r_2}$, which combined with $\frac{1}{r}=\frac{s+d/q}{d}$  gives $\frac{2+s+d/q}{3d}=\frac{1}{r_2}$.
        Let $\beta\ce \frac{1}{3}+\frac16(s+\frac{d}{q})$. 
        Then condition that $\frac{d}{4-s}<q< \frac{d}{1+s/2}$ and Sobolev embeddings ensure that 
        \begin{equation*} \bpsn{3-s,q}\hookrightarrow\bpsn{-1-s+4\beta,q}\hookrightarrow \bpsn{1,r_2}\hookrightarrow L^{r_1}(\sO),
        \end{equation*}
        which combined with Young's inequality gives 
        $$\nrm{I_{21}}_{r} \lesssim \big(\nrm{u}_{-1-s+4\beta,q}^2+\nrm{v}_{-1-s+4\beta,q}^2\big)\nrm{u-v}_{-1-s+4\beta,q}.$$
        The estimate of $I_{22}$ follows in the same way, with the same choices for $r_1, r_2, $ and $\beta$. 
        To estimate $\nrm{I_1}_r$, we use the Sobolev embedding $\bpsn{-1-s+4\beta,q}\hookrightarrow L^{3r}$, which combined with Hölder's inequality give 
        \[\begin{aligned}
            \nrm{I_1}_r&\lesssim \nrm{(u^2+v^2)(u-v)}_r \le (1+\nrm{u}_{3r}^2+\nrm{v}_{3r}^2)\nrm{u-v}_{3r}\\
            &\lesssim (1+\nrm{u}^2_{\bpsn{-1-s+4\beta,q}}+\nrm{v}^2_{\bpsn{-1-s+4\beta,q}})\nrm{u-v}_{\bpsn{-1-s+4\beta,q}}.
        \end{aligned}\]

    Next, we estimate the noise term. For all $u,v\in \bpsn{3-s,q}$ the Sobolev embedding $\bpsn{1-s,q}\hookrightarrow\bpsn{1,r}$ ensures  
		 	\[\begin{aligned}
		 	\nrm{\psi(u)-\psi(v)}_{\gamma(\ell^2, \nH^{1-s,q}(\sO))}&\eqsim\nrm{\psi(u)-\psi(v)}_{\nH^{1-s,q}(\sO;\ell^2)}
            \lesssim \nrm{\psi(u)-\psi(v)}_{_\nu H^{1,r}(\sO;\ell^2)}\\
		 	&\eqsim \nrm{\psi(u)-\psi(v)}_{L^r(\sO;\ell^2)}+\nrm{\nabla\psi(u)-\nabla\psi(v)}_{L^r(\sO;\ell^2)}.
		 \end{aligned}\]
        Next, \eqref{eq:assumption:hg0:noise} ensures for the first term on the right-hand side 
		 \[\nrm{\psi(u)-\psi(v)}_{L^r(\sO;\ell^2)}\lesssim \nrm{(1+|u|^2+|v|^2)|u-v|}_r,\]
		 and therefore Hölder's inequality and the fact that $\bpsn{-1-s+4\beta,q}\hookrightarrow L^{3r}(\sO)$ as above ensure
		 \[\nrm{\psi(u)-\psi(v)}_{L^r(\sO;\ell^2)} \lesssim (1+\nrm{u}^2_{-1-s+4\beta,q}+\nrm{v}_{-1-s+4\beta,q}^2)\nrm{u-v}_{-1-s+4\beta,q}.\]
          It remains to estimate the spatial-gradient-dependent term. Note that 
		 \[\begin{aligned}
		 	\nrm{\nabla\psi(u)-\nabla\psi(v)}_{L^r(\sO;\ell^2)}&\le \sum_{j=1}^d\Big[\nrm{(\partial_{x_j}\psi)(\cdot, u)-(\partial_{x_j}\psi)(\cdot,v)}_{L^r(\sO;\ell^2)}\\
		 	&\qquad \qquad +\nrm{(\partial_y\psi)(\cdot, u)\partial_{x_j}u-(\partial_y\psi)(\cdot,v)\partial_{x_j}v}_{L^r(\sO;\ell^2)}\Big].
		 \end{aligned}\]
         For the first term on the right-hand side, \eqref{Assumption:hg3:noise:2} ensures that 
		 \[\begin{aligned}
		 	\sum_{j=1}^d\nrm{(\partial_{x_j}\psi)(\cdot, u)-(\partial_{x_j}\psi)(\cdot,v)}_{L^r(\sO;\ell^2)}\lesssim (1+\nrm{u}_{3r}^2+\nrm{v}_{3r}^2)\nrm{u-v}_{3r}.
		 \end{aligned}\]

         Next, Hölder's inequality with $\frac{1}{r}=\frac{1}{r_1}+\frac{1}{r_2}$ ensures
         \[\begin{aligned}
             &\nrm{(\partial_y \psi)(u)\partial_{x_j} u - (\partial_y\psi)(v)\partial_{x_j}v}_{L^r(\sO)}\\
		 	&\le \nrm{(\partial_y\psi)(u)\partial_{x_j}u-(\partial_y\psi)(v)\partial_{x_j}u}_{L^r(\sO;\ell^2)}+\nrm{(\partial_y\psi)(v)\partial_{x_j}(u-v)}_{L^r(\sO;\ell^2)}\\
            &\le \nrm{(\partial_y\psi)(u)-(\partial_y \psi)(v)}_{L^{r_1}} \nrm{\partial_{x_j}u}_{L^{r_2}} + \nrm{(\partial_y\psi)(v)}_{L^{r_1}}\nrm{\partial_{x_j}(u-v)}_{L^{r_2}}
         \end{aligned}\]
         We impose $-\frac{d}{r_1}=1-\frac{d}{r_2}$, which as above ensures  $\bpsn{1,r_2}$ and $L^{r_1}(\sO)$ have the same Sobolev index.  This leads to $\frac{r_1}{d}= \frac{r_2}{d-r_2}$ and $\frac{d}{r_2}=\frac{1+s+d/q}{2}$. 

         In order to apply Sobolev embeddings, note that $\bpsn{-1-s+4\beta,q}\hookrightarrow L^{\rho r_1}(\sO)$ if and only if $-1-s+4\beta-\frac{d}{q}\ge -\frac{d}{\rho r_1}=\frac{1}{\rho}(1-\frac{d}{r_2})$. This leads to the condition $(\frac{1}{2\rho}-\frac{1}{3})(s+\frac{d}{q})\ge \frac{1}{2\rho}-\frac{1}{3}$. As $s+\frac{d}{q}>1$ by~\eqref{eq:lwp:rough:initial:data:4}, the latter holds if and  only if  $\rho\le 3/2$. Therefore, \eqref{Assumption:hg3:noise:3}, \eqref{Assumption:hg3:noise:4}, the Sobolev embeddings $\bpsn{-1-s+4\beta,q}\hookrightarrow L^{\frac{3}{2}r_1}$ and $\bpsn{-1-s+4\beta,q}\hookrightarrow \bpsn{1,r_2}$, and Young's inequality ensure
         \[\begin{aligned}
             &\nrm{(\partial_y\psi)(u)-(\partial_y \psi)(v)}_{L^{r_1}} \nrm{\partial_{x_j}u}_{L^{r_2}} + \nrm{(\partial_y\psi)(v)}_{L^{r_1}}\nrm{\partial_{x_j}(u-v)}_{L^{r_2}}\\
             &\lesssim \big(1+\nrm{u}_{\frac{3}{2}r_1}^{1/2} +\nrm{v}_{\frac{3}{2}r_1}^{1/2}\big) \nrm{u-v}_{\frac{3}{2}r_1} \nrm{u}_{1,r_2} +\big(1+ \nrm{v}_{\frac{3}{2}r_1}^{3/2}\big)\nrm{u-v}_{1,r_2}\\
             &\lesssim \big(1+\nrm{u}^{1/2}_{-1-s+4\beta,q}+\nrm{v}_{-1-s+4\beta,q}^{1/2}\big)\nrm{u-v}_{-1-s+4\beta,q} \nrm{u}_{-1-s+4\beta,q}\\
             &\qquad\qquad+ \big(1+\nrm{v}_{-1-s+4\beta,q}^{3/2}\big)\nrm{u-v}_{-1-s+4\beta,q}\\
             &\lesssim \big( 1+\nrm{u}_{-1-s+4\beta,q}^2+\nrm{v}_{-1-s+4\beta,q}^2\big)\nrm{u-v}_{-1-s+4\beta,q}
         \end{aligned}\]
         This verifies Assumption~\ref{Assumption:semilinear:evolution:equation}, and hence the case~\ref{it:lwp:rough:initial:data:3} follows from~\cite[Theorem 4.3]{agrestiNonlinearParabolicStochastic2022a}.

        \medskip
        
        Next, we treat case~\ref{it:lwp:rough:initial:data:4}. As the noise estimate for $\psi^0$ has been carried out above, it suffices to find a suitable estimate for $(\psi^1+ u\psi^2)\nabla u$. 
        To this end, the triangle inequality and the fact that $\Bpsn{1,r}\hookrightarrow \Bpsn{1-s,q}$ ensure that \begin{equation}\label{eq:lwp:rough:initial:data:proof:hessian}        \begin{aligned}
            &\nrm{(\psi^1+ u\psi^2)\cdot\nabla u-(\psi^1+v\psi^2)\cdot\nabla v}_{\Bpsn{1-s,q}}\\
            &\lesssim \nrm{\psi^1\cdot\nabla(u-v)}_{\Bpsn{1-s,q}} +\nrm{(u-v)Du}_{L^r} +\nrm{(u-v) Dv}_{L^r} \\
            &\quad + \nrm{Du D(u-v)}_{L^r} +\nrm{u D^2(u-v)}_{L^r}+ \nrm{D(u-v)Dv}_{L^r} +\nrm{(u-v)D^2 v}_{L^r}.
        \end{aligned} 
        \end{equation}
        It is easy to check the estimates for the first three terms on the right-hand side of~\eqref{eq:lwp:rough:initial:data:proof:hessian}. For the two mixed terms, we apply Hölder's inequality with $\frac{1}{r}=\frac{1}{2r}+\frac{1}{2r}$, and the Sobolev embedding $X_\beta\hookrightarrow \bpsn{1,2r}$. It remains to find suitable estimates for the terms $\nrm{(u-v)D^2v}_{L^r}$ and $\nrm{u D^2(u-v)}_{L^r}$.
         We will find suitable $r_1,r_2 \in (r,\infty)$ with $\frac{1}{r}=\frac{1}{r_1}+\frac{1}{r_2}$, use Hölder's inequality, and try to find suitable $1-\frac{1+\kappa}{p}<\tilde\beta\le\tilde\varphi<1$, such that $X_{\tilde\beta}$ and $X_{\tilde\varphi}$ embed in a suitable space. To be precise, for the first term, we require $X_{\tilde\beta} \hookrightarrow L^{r_1}(\sO)$ and $X_{\tilde\varphi}\hookrightarrow \bpsn{2,r_2}$, while for the second term, we require (with possibly different choices for $\tilde\beta$ and $\tilde\varphi$)  $X_{\tilde\beta} \hookrightarrow \bpsn{2,r_2}$ and $X_{\tilde\varphi}\hookrightarrow L^{r_1}(\sO)$. 
         This means that for the second term we have less flexibility with the embedding. 
        To find admissible choices for $r_1$ and $r_2$, first note that 
        $$ \frac{d}{r}=s+\frac{d}{q} =\frac{d}{r_1}+\frac{d}{r_2}.$$
        Therefore, let $\lambda\in (0,1)$ and choose $r_1\ce r_{\lambda,1}$ with $r_{\lambda,1}\ce \frac{d}{\lambda}$, i.e.\ $\frac{d}{r_{\lambda,1}}=\lambda$,  and $r_2\ce r_{\lambda,2}$ with $\frac{d}{r_{\lambda,2}}= s+\frac{d}{q}-\lambda$.

        We begin with the estimate of $\nrm{|u-v| D^2u}_{L^r(\sO)}$ by  finding $\beta_\lambda$ and $\varphi_\lambda$ such that $X_{\beta_\lambda}\hookrightarrow L^{r_{\lambda,1}}(\sO)$ and $X_{\varphi_\lambda}\hookrightarrow \bpsn{2,r_{\lambda,2}}$. 
        By choosing sharp embeddings, we find that $\beta_\lambda=\frac{1}{4}(1+s+\frac{d}{q}-\lambda)$ and $\varphi_\lambda=\frac{3+\lambda}{4}$. In addition, we require that $s<\lambda<\frac{d}{q}$. Note that as $q<\frac{d}{1+s/2}$ and $s<1$, this imposes no additional assumption on $q$.

        Then as described above, Hölder's inequality and sharp Sobolev embeddings hence prove
        \[\nrm{|u-v| D^2u}_{L^r(\sO)}\lesssim\nrm{D^2 u}_{L^{r_{\lambda,2}}}\nrm{u-v}_{L^{r_{\lambda,1}}(\sO)}\lesssim \nrm{u}_{X_{\varphi_\lambda}} \nrm{u-v}_{X_{\beta_\lambda}}. \]
        It is straightforward to check that~\eqref{eq:lwp:rough:initial:data:3} implies
        \begin{equation}
        	\varphi_\lambda -1+\frac{1+\kappa}{p}+\beta_\lambda = \frac{1}{4}(s+\frac{d}{q})+\frac{1+\kappa}{p} \le 1.
        \end{equation}
        Next, we deal with  $\||v|D^2(u-v)\|_{L^{r}(\sO)}$. 
        In this case, as we require $\beta_\lambda \le \varphi_\lambda$ and at the same time $X_{\beta_{\lambda}}\hookrightarrow \bpsn{2,r_{\lambda,2}}$, we choose $\varphi_\lambda$ as above, and set $\beta_\lambda=\varphi_\lambda$. This is reasonable, as we chose $\varphi_\lambda$ precisely such that the sharp Sobolev embedding holds.  It is however not feasible to repeat above calculations, because the subcriticality condition would possibly be violated. 
        Thus, we are going to rely on the full power of Assumption~\ref{Assumption:semilinear:evolution:equation} and factor a part of the $L^{r_{\lambda,1}}$-norm into the norm of the trace space
         $X_{p,\kappa}^{\textup{Tr}}=\,_\nu B^{3-s-4\frac{1+\kappa}{p}}_{q,p}(\sO)$. 
         With the previous choice of $\lambda$, it is ensured that $_\nu B_{q,1}^{d/q-\lambda}(\sO)\hookrightarrow\bpsn{d/q-\lambda,q}\hookrightarrow L^{r_{\lambda,1}}$. 
         In the following, we will distinguish between three cases. 
            First, assume that $d/q-\lambda <  3-s-4\frac{1+\kappa}{p}$. Then $_\nu B_{q,p}^{3-s-4\frac{1+\kappa}{p}}(\sO)\hookrightarrow\,_\nu B_{q,1}^{d/q-\lambda}(\sO)$, which ensures that
            \[\||v|D^2(u-v)\|_{L^{r}(\sO)}\lesssim (1+\|u\|_{X_{p,\kappa}^{\textup{Tr}}}+\|v\|_{X_{p,\kappa}^{\textup{Tr}}})\|u-v \|_{X_{\varphi_\lambda}}.\]
            Second, assume that $d/q-\lambda> 3-s-4\frac{1+\kappa}{p}$. Then we find numbers $h$ and $\rho_\lambda$ such that 
            \begin{equation*}_\nu B_{q,1}^{d/q-\lambda}(\sO)=\big(\,_\nu B^{3-s-4\frac{1+\kappa}{p}}_{q,p}(\sO), \,_\nu B^{h}_{q,\infty}(\sO)\big)_{\rho_\lambda,1},
            \end{equation*}for which we in addition require that $X_{\varphi_\lambda}\hookrightarrow \,_\nu B_{q,\infty}^{h}(\sO)$, i.e.\ $-1-s+4\varphi_\lambda \ge h$, and such that $\rho_\lambda(\varphi_\lambda-1+\frac{1+\kappa}{p})+\varphi_\lambda\le 1$. 
            This and~\cite[Theorem 1.3.3]{triebelInterpolationTheoryFunction1978} then ensure
            \begin{equation*}\| v\|_{L^{d/\lambda}(\sO)}\lesssim \|v\|_{X_{p,\kappa}^{\textup{Tr}}}^{1-\rho_\lambda} \|v\|_{X_{\varphi_\lambda}}^{\rho_\lambda},\end{equation*}
            which in turn yields
        	\[\nrm{|v|D^2(u-v)}_{L^r(\sO)} \lesssim (1+\nrm{v}_{X_{p,\kappa}^{\textup{Tr}}}^{1-\rho_\lambda}+\nrm{u}_{X_{p,\kappa}^{\textup{Tr}}}^{1-\rho_\lambda})(1+\nrm{u}_{X_{\varphi_\lambda}}^{\rho_\lambda}+\nrm{v}_{X_{\varphi_\lambda}}^{\rho_\lambda})\nrm{u-v}_{X_{\varphi_\lambda}}.\]
            In order to find suitable choices for $\rho_\lambda$ and $h$, note that they are coupled by 
            $$\rho_\lambda=\frac{(d/q-\lambda) -(3-s-4\frac{1+\kappa}{p})}{h-(3-s-4\frac{1+\kappa}{p})}.$$
            Therefore, in order to make $\rho_\lambda$ as small as possible, we have to choose $h$ as large as possible. As $\varphi_\lambda=\frac{3+\lambda}{4}$, the required embedding $X_{\varphi_\lambda}\hookrightarrow \,_\nu B_{q,\infty}^{h}(\sO)$ gives the condition that $h\le 2-s+\lambda$.
            A direct computation shows that $h=2-s+\lambda$, and therefore~\eqref{eq:lwp:rough:initial:data:3} implies
            \[
            \begin{aligned}
                \rho_\lambda\big( \varphi_\lambda-1+\frac{1+\kappa}{p}\big) +\varphi_\lambda&=\frac{d/q+s-\lambda-3+4\frac{1+\kappa}{p}}{\lambda-1+4\frac{1+\kappa}{p}}\frac{1}{4}\big(\lambda-1+4\frac{1+\kappa}{p}\big)+\frac{3+\lambda}{4}\\
                &=\frac{1}{4}\big(s+\frac{d}{q}\big) +\frac{1+\kappa}{p} \le 1.
            \end{aligned}
            \]
            It remains to treat the case in which  $3-s-4\frac{1+\kappa}{p}= d/q-\lambda$. 
            In this case, we combine both mechanisms from the first and second case. This means: choose $\varepsilon_0,\varepsilon_1$ small, and set $\varepsilon\ce \frac{\varepsilon_0}{\varepsilon_0+\varepsilon_1}$. Let $\delta_1>0$ to be determined below, and set $\delta\ce \frac{\varepsilon_1}{\delta_1}$. Then 
            \[\begin{aligned}
                _\nu B_{q,1}^{3-s-4\frac{1+\kappa}{p}}(\sO)&=\big(\,_\nu B^{d/q-\lambda-\varepsilon_0}_{q,\infty}(\sO), \,_\nu B_{q,1}^{d/q-\lambda+\varepsilon_1}\big)_{\varepsilon,1} \hookleftarrow \big(X_{p,\kappa}^{\textup{Tr}}, \,_\nu B_{q,1}^{d/q-\lambda+\varepsilon_1}(\sO)\big)_{\varepsilon,1}\\
                &=\big( X_{p,\kappa}^{\textup{Tr}}, \big(\,_\nu B_{q,\infty}^{d/q-\lambda}(\sO), \,_\nu B_{q,\infty}^{d/q-\lambda+\delta_1}(\sO)\big)_{\delta,1}\big)_{\varepsilon,1}\\
                &\hookleftarrow \big(X_{p,\kappa}^{\textup{Tr}}, \big( X_{p,\kappa}^{\textup{Tr}}, \,_\nu B_{q,\infty}^{d/q-\lambda+\delta_1}(\sO)\big)_{\delta,1}\big)_{\varepsilon,1}.
            \end{aligned}\]
            We choose $\delta_1$ sufficiently small such that $X_{\varphi_\lambda}\hookrightarrow\,_\nu B_{q,\infty}^{d/q-\lambda+\delta_1} $. The latter is true provided $\delta_1\le  2+2\lambda-(s+\frac{d}{q})$. Choosing $\delta_1= 2+2\lambda-(s+\frac{d}{q})$ ensures 
            \[\|v\|_{\,_\nu B_{q,1}^{d/q-\lambda}(\sO)}\lesssim \| v\|_{X_{p,\kappa}^{\textup{Tr}}}^{1-\varepsilon+\varepsilon(1-\delta)}\|v\|_{X_{\varphi_\lambda}}^{\varepsilon\cdot \delta}=\|v\|_{X_{p,\kappa}^{\textup{Tr}}}^{1-\varepsilon\delta} \|v\|_{X_{\varphi_\lambda}}^{\varepsilon\delta},\]
            and thus
            \[\nrm{|v|D^2(u-v)}_{L^r(\sO)} \lesssim (1+\nrm{v}_{X_{p,\kappa}^{\textup{Tr}}}^{1-\varepsilon\delta}+\nrm{u}_{X_{p,\kappa}^{\textup{Tr}}}^{1-\varepsilon\delta})(1+\nrm{u}_{X_{\varphi_\lambda}}^{\varepsilon\delta}+\nrm{v}_{X_{\varphi_\lambda}}^{\varepsilon\delta})\nrm{u-v}_{X_{\varphi_\lambda}}.\]
            Hence choosing both $\varepsilon_0$ and $\varepsilon_1$ sufficiently small yields the required estimate. 
            An application of~\cite[Theorem 4.3]{agrestiNonlinearParabolicStochastic2022a} hence proves~\ref{it:lwp:rough:initial:data:4}.

        Finally, in order to verify the assertion regarding the critical space, set $\frac{1+\kappa}{p}=\frac{3}{2}(1-\beta)$. Then 
        $$(\bpsn{-1-s,q},\bpsn{3-s,q})_{1-\frac{1+\kappa}{p},p}= \,_\nu B_{q,p}^{3-s-4\frac{1+\kappa}{p}}(\sO) =\,_\nu B_{q,p}^{\frac{d}{q}-1}(\sO).$$

Next, we prove case~\ref{it:lwp:rough:initial:data:2}. 
		Set $\theta\ce\frac{1}{2}+\frac{1}{6}(s+\frac{d}{q})$. The condition that $\frac{d}{3-s}<q<2\frac{d}{s}$ then ensures that $0<-2-s+4\theta<2-s$, and the sharp Sobolev embedding $\bpsn{-2-s+4\theta,q}\hookrightarrow L^{3r}(\sO)$ holds. 
		This and Hölder's inequality ensure for all $u,v\in \bpsn{2-s,q}$ that 
			\[\begin{aligned}
			\nrm{\,_\nu\Delta_{-2-s,q} (u^3-v^3)}_{-2-s,q} &\lesssim  \nrm{u^3-v^3}_{-s,q} \lesssim \nrm{u^3-v^3}_r\lesssim \nrm{(u^2+v^2)(u-v)}_r\\
			&\le (1+\nrm{u}_{3r}^2+\nrm{v}_{3r}^2)\nrm{u-v}_{3r}\\
			&\lesssim (1+\nrm{u}_{-2-s+4\theta,q}^2 + \nrm{v}_{-2-s+4\theta,q}^2) \nrm{u-v}_{-2-s+4\theta,q}.
		\end{aligned}\]
		The estimate for the noise term works similarly to the proof of case~\ref{it:lwp:rough:initial:data:4} and is thus omitted. For the estimate for the term $\nrm{u\psi^2 \nabla u- v\psi^2\nabla v}_{\gamma(\ell^2;\bpsn{-s,q})}$, in order to find a suitable $\lambda \in (0,1)$ as above, it is again required that $s<\lambda<\frac{d}{q}$, which gives the additional condition that $q<\frac{d}{s}$. 
        
		This verifies Assumption~\ref{Assumption:semilinear:evolution:equation}, and hence~\ref{it:lwp:rough:initial:data:2} follows from~\cite[Theorem 4.3]{agrestiNonlinearParabolicStochastic2022a}.
        
        To verify the assertion regarding the critical space, set $\frac{1+\kappa}{p}=\frac{3}{2}(1-\theta)=\frac34-\frac14(s+\frac{d}{q})$. Then it holds that 
        $$ (\bpsn{-2-s,q},\bpsn{2-s,q})_{1-\frac{1+\kappa}{p},p}=\,_\nu B_{q,p}^{2-s-4\frac{1+\kappa}{p}}(\sO)=\,_\nu B_{q,p}^{\frac{d}{q}-1}(\sO). $$

        This finishes the proof.
         
	\end{proof}
	\subsection{Regularisation}
	In this subsection we prove instantaneous regularisation for the solutions granted by Theorem~\ref{lwp:rough:initial:data} by employing the abstract regularisation machinery from~\cite[Section 6]{agrestiNonlinearParabolicStochastic2022a}. Here the cases $i)$ and $ii)$ from Theorem~\ref{lwp:rough:initial:data} work a bit differently. We start out with the regularisation result in case $i)$. 
	\begin{lemma}[Regularisation]\label{thm:regularization:rough:initial:data:LpLq}
		Assume the setting of Theorem~\ref{lwp:rough:initial:data} $i)$.
		\begin{enumerate}
			\item\label{it::regularization:rough:initial:data:LpLq:1} Let $(u,\sigma)$ be the $(p,\kappa,s,q)$-solution granted by Theorem~\ref{lwp:rough:initial:data}\ref{it:lwp:rough:initial:data:1} or~\ref{it:lwp:rough:initial:data:2}. Then
			\begin{align}
				u&\in H^{\theta,r}_{\loc}(a,\sigma; \nH^{2-4\theta,\zeta}(\sO)), \quad \zeta\in [q,\infty),\, r\in (2,\infty),\, \theta\in [0,1/2),\\
				u &\in C_{\loc}^{\theta,4\theta}((a,\sigma)\times \sO), \quad \theta \in [0,1/2).
			\end{align}
			\item\label{it::regularization:rough:initial:data:LpLq:2} Let $(u,\sigma)$ be the $(p,\kappa,s,q)$-solution granted by Theorem~\ref{lwp:rough:initial:data}\ref{it:lwp:rough:initial:data:3} or~\ref{it:lwp:rough:initial:data:4}. Then 
			\begin{align}
				u&\in H^{\theta,r}_{\loc}(a,\sigma; \nH^{3-4\theta,\zeta}(\sO)), \quad \zeta\in [q,\infty),\, r\in (2,\infty),\, \theta\in [0,1/2),\\
				u &\in C_{\loc}^{\theta,6\theta}((a,\sigma)\times \sO), \quad \theta \in [0,1/2).
			\end{align}
		\end{enumerate}
	\end{lemma}
	\begin{proof}
				For $r\in [2,\infty), \alpha\in [0,r/2-1)\cup \{0\}$, and an interpolation couple $(Z_0,Z_1)$, define $Z_r^\textup{Tr}\ce Z_{1-\frac{1}{r},r}$ and  $Z_{\alpha,r}^\textup{Tr}\ce Z_{1-\frac{1+\alpha}{r},r}$. 
				
				Since $\sO$ is bounded, we have $L^{q_1}(\sO)\hookrightarrow L^{q_2}(\sO)$ for $q_1\ge  q_2$ by Hölder's inequality, which extends to $H^{t,q_1}(\sO)\hookrightarrow H^{t,q_2}(\sO)$, $t\in\R, \infty >q_1\ge q_2>1$ by first proving this for $t\in \N_0$ and then using duality and interpolation.
                This furthermore ensures $B^t_{q_1,p}(\sO)\hookrightarrow B^t_{q_2,p}(\sO)$, $t\in\R, \infty> q_1\ge q_2>1, 1<p<\infty$.
				
				We start by proving~\ref{it::regularization:rough:initial:data:LpLq:1}.
				
				\emph{Step 1: Regularisation in time.} By~\cite[Theorems 5.6 and 5.7]{agrestiNonlinearSPDEsMaximal2025} we have
				\[u\in H_{\loc}^{\theta,r}(0,\sigma; \nH^{2-s-4\theta,q}(\sO))\cap C_{\loc}^{\theta-\varepsilon}(0,\sigma; \nH^{2-s-4\theta,q}(\sO)) \text{ a.s.}, \quad \theta\in [0,1/2), r\in[2,\infty), \varepsilon\in (0,\theta).\]
				
				\emph{Step 2: Sobolev smoothness in space.} Without loss of generality we may assume that $s>0$. 
				We are going to employ~\cite[Theorem 6.3]{agrestiNonlinearParabolicStochastic2022a}(applied for $j\in \{0,1\}$ with $Y_j\curvearrowleft\bpsn{-2-s+4j,q}, $ $ X_j \curvearrowleft \bpsn{-2-s+4j,q}$, $ \widehat{Y}_j\curvearrowleft \bpsn{-2+4j,q}$, $\delta\curvearrowleft0$, $p\curvearrowleft r$,  $r\curvearrowleft r$, $ \widehat{r}\curvearrowleft r $, $\alpha\curvearrowleft 0$, $\widehat{\alpha}\curvearrowleft \frac{r}{4}s$ in the notation of~\cite[Theorem 6.3]{agrestiNonlinearParabolicStochastic2022a}, where $r$ is chosen below) in order to improve the smoothness in space. 
                
                The assumptions of Theorem~\ref{lwp:rough:initial:data} are met  for the $(Y_0,Y_1, r,\alpha)$ and $(\widehat Y_0, \widehat Y_1, \widehat r, \widehat \alpha)$ settings, and the corresponding trace spaces are subcritical for $r$ large enough. Moreover, the choice of $\widehat\alpha$ ensures that $Y_r^{\textup{Tr}}= \widehat Y_{\widehat\alpha,\widehat r}^\textup{Tr}$.  
				The assumptions (1) and (2) from~\cite[Theorem 6.3]{agrestiNonlinearParabolicStochastic2022a} are easily verified. 
				For the verification of assumption (3), it only remains to prove that~\cite[(6.1)]{agrestiNonlinearParabolicStochastic2022a} holds. To this end we apply~\cite[Lemma 6.2 (4)]{agrestiNonlinearParabolicStochastic2022a} with $\varepsilon=\frac{\widehat\alpha}{r}$. We thus need to prove that $\widehat Y_{1-\frac{\widehat \alpha}{r}}\hookrightarrow Y_1$ and $\widehat Y_0 \hookrightarrow Y_{\frac{\widehat\alpha}{r}}$ for some $\widehat\alpha\in (0,\frac{r}{2}-1)$. This is the case if we can choose $4\frac{\widehat \alpha}{r}\le s$. In order to ensure $\widehat\alpha =  \frac{r}{4}s \in (0,\frac{r}{2}-1)$, we need to choose $r$  so that 
				$4\frac{\widehat\alpha}{r}\le s < 2-\frac{4}{r}$. Therefore, an application of~\cite[Theorem 6.3]{agrestiNonlinearParabolicStochastic2022a} ensures that 
				$$ u\in H^{\theta,r}_{\loc}(0,\sigma; \nH^{2-4\theta,q}(\sO)), \quad r\in [2,\infty),\, \theta\in [0,\frac{1}{2}).$$
				
				\emph{Step 3: Integrability in space.} Let $r$ be so large that $\frac{1}{r}+\frac{1}{4}(s+\frac{d}{q})<\frac{3}{4}$. Note that such a choice is always possible due to \eqref{eq:lwp:rough:initial:data:1}. Let $\alpha>0$ be so small that also $\frac{1+\alpha }{r}+\frac{1}{4}(s+\frac{d}{q})< \frac{3}{4}$. In addition, we assume that $\frac{1+\alpha}{r}<\frac{1}{8}$.  We will prove for all $\zeta \ge q$ that 
				\begin{equation}
					u\in \bigcap_{\theta\in [0,\frac{1}{2})}H_{\loc}^{\theta,r}(0,\sigma; \nH^{2-4\theta,\zeta}(\sO))\text{ a.s.} \quad \Longrightarrow \quad u\in \bigcap_{\theta\in [0,\frac{1}{2})} H_{\loc}^{\theta,r}(0,\sigma; \nH^{2-4\theta,\zeta+\varepsilon}(\sO)) \text{ a.s.,}
				\end{equation}
				where $\varepsilon$ is to be determined below. 
				Note that the left-hand side holds for $\zeta=q$ due to step 2. 
				For $j\in \{0,1\}$ set $X_j\ce \nH^{-2-s+4j,q}(\sO)$, $Y_j\ce \nH^{-2+4j,\zeta}(\sO)$ and $\widehat Y_j \ce \nH^{-2+4j,\zeta+\varepsilon}(\sO)$, $\widehat r\ce r, \widehat \alpha \ce \alpha$, where $\varepsilon $ only depends on $\alpha, d$, and $r$. 
				
				We will employ~\cite[Theorem 6.3]{agrestiNonlinearParabolicStochastic2022a}. Assumptions (1) and (2) in~\cite[Theorem 6.3]{agrestiNonlinearParabolicStochastic2022a} are easy to verify. First, note that 
				$Y_j\hookrightarrow X_j, Y_r^{\textup{Tr}}\hookrightarrow X_p^{\textup{Tr}}$ by Sobolev embeddings, and $u$ is suitably progressively measurable by Theorem~\ref{lwp:rough:initial:data}. This implies (1). For (2), we constructed $\widehat Y_{\widehat \alpha,\widehat r}^{\textup{Tr}}$ in such a way that it is \emph{not} critical. 
				
				To verify Assumption (3) in~\cite[Theorem 6.3]{agrestiNonlinearParabolicStochastic2022a}, first note that since $\sO $ is bounded, it holds that $\widehat{Y}_i \hookrightarrow Y_i, i\in \{0,1\}$, see the remark at the beginning of this proof. In addition, we need to make sure that $Y_r^{\textup{Tr}}\hookrightarrow \widehat Y^{\textup{Tr}}_{\widehat\alpha,\widehat r}$, i.e.\ $_\nu B_{\zeta,r}^{2-\frac{4}{r}}(\sO)\hookrightarrow \,_\nu B_{\zeta+\varepsilon,r}^{2-4\frac{1+\alpha}{r}}(\sO)$, which holds by Sobolev embeddings, provided that 
				\[2-\frac{4}{r}-\frac{d}{\zeta}\ge 2-4\frac{1+\alpha}{r}-\frac{d}{\zeta+\varepsilon} \quad \Longleftrightarrow \quad \frac{4\alpha}{dr}\ge \frac{1}{\zeta}-\frac{1}{\zeta+\varepsilon}= \frac{\varepsilon}{\zeta(\zeta+\varepsilon)}.\]
				This can be achieved by setting $\varepsilon\ce \frac{4\alpha}{dr}$. 
                Moreover,~\cite[(6.1)]{agrestiNonlinearParabolicStochastic2022a} holds by~\cite[Lemma 6.2 (1)]{agrestiNonlinearParabolicStochastic2022a}. Therefore,~\cite[Theorem 6.3]{agrestiNonlinearParabolicStochastic2022a} is applicable (uniformly in $\zeta$), which ensures that 
				\begin{equation}
					u\in \bigcap_{\theta\in [0,\frac{1}{2})} H_{\loc}^{\theta,r}(0,\sigma; \nH^{2-4\theta,\zeta+\varepsilon}(\sO)) \seq C(0,\sigma; \, _\nu B_{\zeta+\varepsilon,r}^{2-\frac{4}{r}}(\sO)).
				\end{equation}
				This finishes the proof of~\ref{it::regularization:rough:initial:data:LpLq:1}. 
                
                For~\ref{it::regularization:rough:initial:data:LpLq:2},   set $Y_j\ce X_j$ and $\widehat{Y}_j\ce \bpsn{-1+4j,q}$ in Step 2 and $Y_j\ce\bpsn{-1+4j,\zeta}, \widehat{Y}_j\ce\bpsn{-1+4j,\zeta+\varepsilon}$ in Step 3.
                However, note that in this case, the local well-posedness result Theorem~\ref{lwp:rough:initial:data} does not allow us to iterate the integrability in space up to infinity, but only up to $\zeta<d$, i.e.\ we have 
                \[u\in H_\loc^{\theta,r}(0,\sigma; \bpsn{3-4\theta, \zeta}), \quad \zeta\in [q,d), r\in (2,\infty), \theta\in [0,1/2).\]
                To obtain further integrability in space, we prove this `by hand' with a maximal regularity technique. 
                Let $\zeta\ge q$ be given and let $t_0>0$. In addition, set 
                \[F(u)\ce  \,_\nu\Delta_{-1,\zeta+\varepsilon}(u^3-u), \quad G(u)\ce  \psi(u,\nabla u).\]
                Let $(\sigma_m)_{m\ge1}$ be a localising sequence for $\sigma$. 
                Then on the set $\{\sigma_m>t_0\}\times(t_0,\sigma_m)$ we prove that 
                \begin{equation*}F(u)\in L^r(t_0,\sigma_m;\bpsn{-1,\zeta+\varepsilon}),\quad  G(u)\in L^r(t_0,\sigma_m;\gamma(\ell^2,\bpsn{1,\zeta+\varepsilon})),
                \end{equation*}where $\varepsilon$ can be chosen independently of $\zeta$. Moreover, we prove that $\1_{\sigma_m>t_0}u(t_0) \in \,_\nu B_{\zeta,r}^{3-\frac{4}{r}}(\sO) \hookrightarrow \,_\nu B_{\zeta+\varepsilon,r}^{3-4\frac{1+\alpha}{r}}(\sO)$. Since $_\nu\Delta_{-1,\zeta+\varepsilon}^2 \in \mathcal{SMR}_{r,\alpha}^\bullet$, the implication follows from stochastic maximal $L^r_{\alpha}$-regularity, combined with~\cite[Proposition 3.11] {agrestiNonlinearSPDEsMaximal2025} and the forcing terms $F(u), G(u)$. 
                We start out with the initial value. It follows from~\cite[Proposition 2.1]{agrestiNonlinearSPDEsMaximal2025} and the already proven regularity that it holds a.s.\ 
                \[ \1_{\sigma_m>t_0} u(t_0) \in \,_\nu B_{\zeta,r}^{3-4/r}(\sO) \hookrightarrow \,_\nu B^{3-4\frac{1+\alpha}{r}}_{\zeta+\varepsilon,r}(\sO),\]
                where the embedding holds with the same $\varepsilon$ as above. 
                Next, we deal with the deterministic forcing term.  Using the fact that we can start our procedure for $\zeta$ arbitrarily close to $d$, one verifies that we always have $\bpsn{3,\zeta}\hookrightarrow \bpsn{1,\zeta+\varepsilon}$, proving 
                \[\nrm{\Delta(u^3-u)}_{-1,\zeta+\varepsilon}\lesssim \nrm{u^3}_{1,\zeta+\varepsilon}+\nrm{u}_{1,\zeta+\varepsilon}\lesssim \nrm{u^3}_{1,\zeta+\varepsilon} + \nrm{u}_{3,\zeta}.\]
                It remains to estimate the first term on the right-hand side. Here we use a similar argument as in the proof of Theorem~\ref{lwp:rough:initial:data}. That is, after using the chain rule, we apply Hölder's inequality with $\frac{1}{\zeta+\varepsilon}=\frac{2}{r_1}+\frac{1}{r_2}$ and $ -\frac{d}{r_1}=1-\frac{d}{r_2}$, which yields $r_1=\frac{dr_2}{d-r_2}$ and $\frac{1}{r_2}=\frac{1}{3}(\frac{2}{d}+\frac{1}{\zeta+\varepsilon})$. This ensures that $L^{r_1}(\sO)$ and $\bpsn{1,r_2}$ have the same Sobolev index. In order to ensure that $\bpsn{3,\zeta}\hookrightarrow \bpsn{1,r_2}$, by Sobolev embeddings, it is sufficient that $\frac{8}{3d}\ge \frac{1}{\zeta}-\frac{1}{3(\zeta+\varepsilon)}$. As the right-hand side is monotonically decreasing in $\zeta$, one easily verifies that one can always choose $\varepsilon$ as above. Analogous reasoning also ensures that $\bpsn{3,\zeta}\hookrightarrow L^{3(\zeta+\varepsilon)}(\sO)$. This finally gives 
                \[\begin{aligned}
                    \nrm{ u^3}_{1,\zeta+\varepsilon}&\lesssim \nrm{u^3}_{\zeta+\varepsilon}+\nrm{\nabla (u^3)}_{\zeta+\varepsilon}\le \nrm{u}_{3(\zeta+\varepsilon)}^3 +\nrm{u^2 \nabla u}_{\zeta+\varepsilon}\\
                    &\lesssim \nrm{u}_{3,\zeta}^3 + \nrm{u}^2_{r_1}\nrm{\nabla u}_{r_2}\lesssim \nrm{u}_{3,\zeta}^3.
                \end{aligned}\]
                In particular, this entails
                \[\|u^3\|_{L^r(t_0,\sigma_m;\bpsn{1,\zeta+\varepsilon})} \lesssim \|u\|_{L^{3r}(t_0,\sigma_m;\bpsn{3,\zeta})}^3.\]
                The corresponding estimates for the noise term can be derived by the same calculations by using the growth assumptions arising from~\eqref{eq:assumption:hg0:noise}, \eqref{Assumption:hg3:noise:0}, and \eqref{Assumption:hg3:noise:4}. Letting $m\to\infty$ finishes the proof of~\ref{it::regularization:rough:initial:data:LpLq:2} in case of Assumption~\ref{assumption:hg3}. 
                
                For the setting from Assumption~\ref{assumption:hg4}, the situation is more delicate: all estimates but one carry over verbatim. The problematic term is the estimate for $\nrm{|u|D^2u}_{L^{\zeta+\varepsilon}(\sO)}$. In the proof of Theorem~\ref{lwp:rough:initial:data}, we partially absorbed $|u|$ into the trace space. Here, we will combine this with a stopping time argument.
				 
				Fix $r$ large enough, let $\zeta\in [2,d)$, and let $0<t_0<T<\infty$ be arbitrary. 
				Choose a localising sequence $(\sigma_m)_{m\in \N} $ for $\sigma$. On $\{\sigma_m>t_0\}$ define 
                $$N_m(t)\ce  \|u\|_{C([t_0,t\wedge\sigma_m];\,_\nu B_{\zeta,r}^{3-4/r})}+ \|u\|_{L^r(t_0,t\wedge\sigma_m;\bpsn{3,\zeta})}, \quad m\in \N, t\ge t_0.$$
				For $n\ge 1$ let $\tau_{m,n}\ce \sigma_m\wedge T \wedge \inf\{ t\ge t_0 \colon N_m(t)\ge n\}$ on $\{\sigma_m>t_0\}$ and $\tau_{m,n}\ce t_0$ on $\{\sigma_m\le t_0\}$. Note that $N_m$ is the sum of a stochastic process with continuous adapted paths and an adapted continuous increasing process. Therefore, $\tau_{m,n}, m,n\in \N$, are stopping times. Moreover, $\tau_{m,n}\nearrow \sigma_m\wedge T$ on $\{\sigma_m >t_0\}$. 
				
				Next, choose $\xi>\zeta$ such that $\frac{1}{\zeta}-\frac{1}{\xi}<\frac{1}{d}$. Then, if $r>8$ and $\zeta$ is chosen sufficiently close to $4$, it holds that $3-\frac{4}{r}-\frac{d}{\zeta}>1$, which ensures that $\,_\nu B^{3-4/r}_{\zeta,r}(\sO)\hookrightarrow W^{1,\infty}(\sO)$. This in turn ensures that on $(t_0,\tau_{m,n})$ we have  
				$$\nrm{u}_{W^{1,\infty}(\sO)}\lesssim \nrm{u}_{\,_\nu B^{3-4/r}_{\zeta,r}(\sO)} \le n, \quad \nrm{u}_{L^r(t_0,\tau_{m,n};\bpsn{3,\zeta})}\le n.$$ 
				Due to our choice of $\xi$, it is ensured that $\bpsn{3,\zeta}\hookrightarrow \bpsn{2,\xi}$. 
				With that, one can check that 
				\[\nrm{|u(t)|D^2u(t)}_{L^{\xi}(\sO)} \lesssim_n \nrm{u(t)}_{\bpsn{3,\zeta}}.\]
				Since $_\nu\Delta_{-1,\xi}^2 \in \mathcal{SMR}_{r,\alpha}^\bullet$, it follows from stochastic maximal $L^r_{\alpha}$-regularity, combined with~\cite[Proposition 3.11] {agrestiNonlinearSPDEsMaximal2025} that  there exists a $v$ solving the equation with
				$v\in H^{\theta,r}(t_0,\tau_{m,n}, w_\alpha^{t_0}; \bpsn{3-4\theta, \xi})$. Therefore,  $v$ and $\1_{\{\tau_{m,n}>t_0\}} u$ both belong to $L^r(t_0,\tau_{m,n},w_\alpha^{t_0};\bpsn{3,\zeta})$. Moreover, since   $(\,_\nu\Delta_{\alpha,\zeta}^2)_{\alpha\in(-4,0],\zeta\in(1,\infty)}$ is a consistent family of operators, it holds that $_\nu\Delta_{-1,\xi}^2 u = \,_\nu\Delta_{-1,\zeta}^2u$. Thus, by uniqueness, it holds that 
				$$v=u \quad \text{on}\quad \{\tau_{m,n}>t_0\}\times(t_0,\tau_{m,n}) \quad \text{a.s.}.$$
				Now, for any $t_1>t_0 $ it holds for $t\ge t_1$ that $w_\alpha^{t_0}\eqsim 1$, which proves that 
				$$u\in \bigcap_{\theta\in [0,1/2)} H^{\theta,r}(t_1,\tau_{m,n}; \bpsn{3-4\theta, \xi}).$$
				Since $0<t_0<t_1$ were arbitrary, letting $m,n\to\infty$ thus yields the claim. 
                This finishes the proof. 
	\end{proof}
    \begin{remark} 
The proof above is in part formulated in terms of the abstract bootstrapping machinery developed in~\cite{agrestiNonlinearParabolicStochastic2022a}. Roughly speaking, one has to ensure sufficient flexibility for local well-posedness along a scale of function spaces with suitable embeddings. In the present setting, these embeddings are provided by Sobolev embeddings, whose scaling is linear. This makes it possible to treat several settings simultaneously in a unified way.

We emphasise that the boundedness of $\sO$ is used explicitly. Indeed, in order to bootstrap spatial integrability by means of the techniques from~\cite{agrestiNonlinearParabolicStochastic2022a}, this assumption is required to ensure $\widehat{Y}_i\hookrightarrow Y_i$. One possible way to circumvent this restriction is to use more explicit arguments based on stochastic maximal regularity estimates and term-by-term calculations, as in the last part of the proof above.
Another restriction for these techniques is --- as mentioned above --- that they are only applicable in settings where local well-posedness is available. It is, however, often possible to obtain further regularity by means of maximal regularity techniques. The proof above showcases such a scenario.

One additional advantage of the approach in~\cite{agrestiNonlinearParabolicStochastic2022a} is that it yields regularisation in the critical case $p=2,\kappa=0$. This is noteworthy, since temporal regularisation is not available in this setting. However, the proof of Lemma~\ref{thm:regularization:rough:initial:data:LpLq} relies crucially on the availability of flexible temporal regularity with weights. The key idea is to show that the solution also exists in a less regular, but subcritical, setting. This then gives temporal regularity, which ultimately makes it possible to apply the techniques described above.
   \end{remark}
	Next, we prove a corresponding result for the case $p=q=2$ and $s=\kappa=0$. 
	\begin{lemma}[Regularity for $p=q=2, s=\kappa=0$]\label{thm:regularization:rough:initial:data:L2L2}
		Let either $d\in \{1,2\}$ and $X_0=\bpsn{-2}, X_1=\bpsn{2}$ or  $d\in \{3,4\}$ and $X_0=\bpsn{-1}, X_1=\bpsn{3}$.
 Let $u_0\in L^0_{\F_0}(\Omega; X_{1/2})$, and let $(u,\sigma)$ be the $(2,0,0,2)$-solution provided by Theorem~\ref{lwp:rough:initial:data}. Then, 
		\begin{itemize}
			\item if $d\in \{1,2\}$ and $X_0=\bpsn{-2}, X_1=\bpsn{2}$, it holds that 
				\begin{align}
				u&\in H_{\loc}^{\theta,r}(a,\sigma; \nH^{2-4\theta,\zeta}(\sO))\quad \text{a.s.\ for all }\theta\in[0,1/2),\quad r,\zeta\in(2,\infty),\\
				u&\in C_{\loc}^{\theta_1,\theta_2}((a,\sigma)\times\sO) \quad \text{ a.s.\ for all }\theta_1\in (0,1/2),\quad  \theta_2\in (0,2).
			\end{align}
			\item if $d\in \{3,4\}$ and $X_0=\bpsn{-1}, X_1=\bpsn{3}$, it holds that 
			\begin{align}
				u&\in H_{\loc}^{\theta,r}(a,\sigma; \nH^{3-4\theta,\zeta}(\sO))\quad \text{a.s.\ for all }\theta\in[0,1/2),\quad r,\zeta\in(2,\infty),\\
				u&\in C_{\loc}^{\theta_1,\theta_2}((a,\sigma)\times\sO) \quad \text{ a.s.\ for all }\theta_1\in (0,1/2),\quad \theta_2\in (0,3).
			\end{align}
		\end{itemize}
	\end{lemma}
	\begin{proof}
		If $d=1$, $X_0=\bpsn{-2}$, $ X_1=\bpsn{2}$ or $d=3$, $ X_0=\bpsn{-1}$, $ X_1=\bpsn{3}$, we are in the subcritical regime and a reiteration of Steps 1 and 3 in the proof of Lemma~\ref{thm:regularization:rough:initial:data:LpLq} is applicable, where the temporal regularisation follows from~\cite[Proposition 5.9]{agrestiNonlinearSPDEsMaximal2025}. 
		
		Next, let $d=2 $ and  $X_0=\bpsn{-2}, X_1=\bpsn{2}$. We will find $\varepsilon>0$ such that 
		\begin{equation}\label{eq:regularization:rough:initial:data:L2L2:proof:1}
			u\in H_{\loc}^{\theta,r}(a,\sigma; \bpsn{2-\varepsilon-4\theta}) \,\text{ a.s.}, \quad \theta\in [0,1/2), \quad r\in (2,\infty).
		\end{equation}
		To this end, we employ~\cite[Proposition 6.8]{agrestiNonlinearParabolicStochastic2022a}(applied with $Y_i \curvearrowleft \bpsn{-2+4i-\varepsilon}$, $X_i \curvearrowleft \bpsn{-2+4i}$, $p\curvearrowleft2$, $\rho_j\curvearrowleft 2, \varphi_j\curvearrowleft \frac12+\frac16(s+\frac{d}{q})=\frac{1}{2}+\frac{d}{12}$, $\delta\curvearrowleft\frac{\varepsilon}{4}$ in the notation of~\cite[Proposition 6.8]{agrestiNonlinearParabolicStochastic2022a}). 
		Upon choosing $\varepsilon>0$ and $r>2$ sufficiently close to $2$ (to be precise,~\cite[Proposition 6.8]{agrestiNonlinearParabolicStochastic2022a} requires $\delta\in (0,1-\max_j \varphi_j)=(0,\frac12-\frac{d}{12})$, i.e.\ $0<\varepsilon< 2-\frac{d}{3}$; moreover, it is required that $\frac1r\ge \frac{d}{12}$, i.e.\ $r\le 6$ if $d=2$), 
        so that $\frac{1}{2}=\frac{1+\alpha}{r}+\frac{\varepsilon}{4}$ with  some $0<\alpha<r/2-1$, 
		\cite[Proposition 6.8]{agrestiNonlinearParabolicStochastic2022a} proves 
        \begin{equation}
            u\in H_{\loc}^{\theta,r}(a,\sigma; \bpsn{2-\varepsilon-4\theta})\,\text{ a.s.}, \quad \theta\in [0,1/2). 
        \end{equation}
        Next, we may further bootstrap temporal regularity by employing~\cite[Corollary 6.5]{agrestiNonlinearParabolicStochastic2022a}(applied with the same choices for $Y_i, X_i, p, \rho_j, \varphi_j, \delta$ in~\cite[Corollary 6.5]{agrestiNonlinearParabolicStochastic2022a} as above). Note that, in the notation of~\cite[Corollary 6.5]{agrestiNonlinearParabolicStochastic2022a}, in this case we have  $X_{1-\delta-\theta}=Y_{1-\theta}$ and that $$\begin{aligned}
            Y_r^{\textup{Tr}} &=(Y_0,Y_1)_{1-\frac1r,r}\hookrightarrow (X_0,X_{1-\delta})_{1-\frac1r,r}=(X_0,X_1)_{(1-\delta)(1-\frac{1}{r}),r}\\&\hookrightarrow (X_0,X_1)_{1-\delta-\frac1r,r}=(X_0,X_1)_{1-\frac1p+\frac{\alpha}{r},r}\hookrightarrow X_p^{\textup{Tr}}.
        \end{aligned}$$
        This shows that we meet the first assumption of~\cite[Corollary 6.5]{agrestiNonlinearParabolicStochastic2022a}. 
        Next, we verify that the second assumption is met. Since the  operator $1+\,_\nu\Delta^2_{-2-\varepsilon,2}$ has a bounded $H^\infty$-calculus of angle $<\frac{\pi}{2}$ and is in particular  time-independent, it is ensured that~\cite[Assumption 4.5]{agrestiNonlinearParabolicStochastic2022a} holds for all $\widehat{r}\in [r,\infty), \widehat{\alpha}\in [0,\frac{\widehat{r}}{2}-1)$. Further, the estimates in the proof of Theorem~\ref{lwp:rough:initial:data} ensure that~\cite[Assumption 4.7]{agrestiNonlinearParabolicStochastic2022a} holds for all $\widehat{r}\in [r,\infty), \widehat{\alpha}\in [0,\frac{\widehat{r}}{2}-1)$.
        This proves the applicability of~\cite[Corollary 6.5]{agrestiNonlinearParabolicStochastic2022a}, which now ensures that
		\begin{equation}
			u\in H_{\loc}^{\theta,\widehat r}(a,\sigma; \bpsn{2-\varepsilon-4\theta})\,\text{ a.s.},\quad \theta\in [0,1/2), \quad \widehat{r}\in [2,\infty). 
		\end{equation}
		After this, we may bootstrap further regularity with the techniques employed in the proof of Lemma~\ref{thm:regularization:rough:initial:data:LpLq}. This proves 
		\begin{equation}
			u\in \bigcap_{\theta\in [0,\frac{1}{2})} H_{\loc}^{\theta,r}(a,\sigma; \bpsn{2-4\theta,\zeta}) \quad \text{a.s.\ for all } r,\zeta\in (2,\infty). 
		\end{equation}
        For the case $d=4$ and $X_0=\bpsn{-1}$, $X_1=\bpsn{3}$ we argue similarly.
        Remarkably, even though the criticality condition~\eqref{eq:lwp:rough:initial:data:3} differs, the choice of $\rho_j$ does not. 
        For the application of~\cite[Theorem 6.8]{agrestiNonlinearParabolicStochastic2022a}, we choose $\varphi_j\curvearrowleft \frac{1}{3}+\frac{1}{6}(s+\frac{d}{q})=\frac{2}{3}$, and thus the choice of $\delta$ is still $\delta\in (0,1-\max_j \varphi_j)=(0,1/3)$. Moreover, it is again required that $r\le 6$. 
	\end{proof}
    
	\subsection{Global well-posedness: transference of blow-up criteria}
    The next result transfers the natural blow-up criteria associated to the $(p,\kappa,s,q)$-solutions $(u,\sigma)$ into `better' spaces. To be more precise, we prove that as a consequence of our regularisation results Theorem~\ref{thm:regularization:rough:initial:data:LpLq} and Theorem~\ref{thm:regularization:rough:initial:data:L2L2} above, the $(p,\kappa,s,q)$-solution inherits the blow-up criterion from the $(2,0,0,2)$-solution. Blow-up criteria for the latter are often easier to verify. 
    As a consequence of this and Theorem~\ref{thm:GWP:weak:setting}, Theorem~\ref{thm:global_well_posedness_dim_3}, and Theorem~\ref{thm:global_well_posedness_dim_4-complete}, we obtain global existence in Corollary~\ref{cor:GWP:rough} in several settings. The technique employed in Theorem~\ref{gwp:rough:initial:data} is similar to~\cite[Theorem 2.10]{agresti2023reactiondiffusion_local}, where reaction--diffusion systems are treated. 
	\begin{theorem}[Transference of blow-up criteria]\label{gwp:rough:initial:data}
		Let $p\in (2,\infty),\kappa\in [0,p/2-1),q\in [2,\infty)$. Assume that one of the following settings holds:
		\begin{enumerate}
			\item\label{it:gwp:rough:initial:data:1} $d\in \{1,2\}, X_0=\bpsn{-2-s,q}$, $X_1=\bpsn{2-s,q}$, $s\in [0,2)$,~\eqref{eq:lwp:rough:initial:data:1},~\eqref{eq:lwp:rough:initial:data:2}, and Assumption~\ref{Assumption:noise}\ref{assumption:hg1} hold.
			\item\label{it:gwp:rough:initial:data:2} $d\in \{1,2\}$, $ X_0=\bpsn{-2-s,q}$, $ X_1=\bpsn{2-s,q}$, $ s\in [0,1)$,~\eqref{eq:lwp:rough:initial:data:1},~\eqref{eq:lwp:rough:initial:data:2}, and Assumption~\ref{Assumption:noise}\ref{assumption:hg2} hold.
			\item\label{it:gwp:rough:initial:data:3} $d\in \{3,4\}$, $ X_0=\bpsn{-1-s,q}$, $ X_1=\bpsn{3-s,q}$, $ s\in [0,2)$,~\eqref{eq:lwp:rough:initial:data:3},~\eqref{eq:lwp:rough:initial:data:4}, and Assumption~\ref{Assumption:noise}\ref{assumption:hg3} hold.
            \item\label{it:gwp:rough:intial:data:4} $d\in \{3,4\}$,  $X_0=\bpsn{-1-s,q}$, $ X_1=\bpsn{3-s,q}$, $ s\in [0,1)$,~\eqref{eq:lwp:rough:initial:data:3},~\eqref{eq:lwp:rough:initial:data:4},
            and Assumption~\ref{Assumption:noise}\ref{assumption:hg4} hold. 
		\end{enumerate}
		
		\begin{itemize}
				\item 	Assume~\ref{it:gwp:rough:initial:data:1} or~\ref{it:gwp:rough:initial:data:2} with $d=1$ holds, let  $u_0\in L^0_{\F_0}(\Omega;\,_\nu B_{q,p}^{2-s-4\frac{1+\kappa}{p}}(\sO))$, and let $(u,\sigma)$ be the $(p,\kappa,s,q)$-solution granted by Theorem~\ref{lwp:rough:initial:data}. Then for all  $0<t_0<T<\infty$ it holds that
				\begin{equation}
				\P\bigg(t_0<\sigma<T, \sup_{t\in [t_0,\sigma)}\nrm{u(t)}_{L^2(\sO)}<\infty\bigg)=0.
				\end{equation}
                \item 	Assume~\ref{it:gwp:rough:initial:data:1} or~\ref{it:gwp:rough:initial:data:2} with $d=2$ holds, let  $u_0\in L^0_{\F_0}(\Omega; \,_\nu B_{q,p}^{2-s-4\frac{1+\kappa}{p}}(\sO))$, and let $(u,\sigma)$ be the $(p,\kappa,s,q)$-solution granted by Theorem~\ref{lwp:rough:initial:data}. Then for all  $0<t_0<T<\infty$ it holds that
				\begin{equation}
				\P\bigg(t_0<\sigma<T, \sup_{t\in [t_0,\sigma)}\nrm{u(t)}_{L^2(\sO)}+\nrm{u}_{L^2(t_0,\sigma; H^2(\sO))}<\infty\bigg)=0.
				\end{equation}
				\item Assume~\ref{it:gwp:rough:initial:data:3} or~\ref{it:gwp:rough:intial:data:4} with $d=3$  holds, let $u_0\in L^0_{\F_0}(\Omega; \,_\nu B_{q,p}^{3-s-4\frac{1+\kappa}{p}}(\sO))$, and let $(u,\sigma)$ be the $(p,\kappa,s,q)$-solution to \eqref{SCH} granted by Theorem~\ref{lwp:rough:initial:data}. Then for all $0<t_0<T<\infty$ it holds that 
				\begin{equation}
					\P\bigg(t_0<\sigma<T, \sup_{t\in [t_0,\sigma)}\nrm{u(t)}_{\bpsn{1}}<\infty\bigg)=0.
				\end{equation}
                \item Assume~\ref{it:gwp:rough:initial:data:3} or~\ref{it:gwp:rough:intial:data:4} with $d=4$ holds, let $u_0\in L^0_{\F_0}(\Omega;\,_\nu B_{q,p}^{3-s-4\frac{1+\kappa}{p}}(\sO))$, and let $(u,\sigma)$ be the $(p,\kappa,s,q)$-solution to \eqref{SCH} granted by Theorem~\ref{lwp:rough:initial:data}. Then for all $0<t_0<T<\infty$ it holds that 
				\begin{equation}
					\P\bigg(t_0<\sigma<T, \sup_{t\in [t_0,\sigma)}\nrm{u(t)}_{\bpsn{1}}+\nrm{u}_{L^2(t_0,\sigma; \bpsn{3})}<\infty\bigg)=0.
				\end{equation}
		\end{itemize}
	\end{theorem}
    \begin{proof}
        Assume that~\ref{it:gwp:rough:initial:data:1} or~\ref{it:gwp:rough:initial:data:2} holds, that $u_0\in \,_\nu B^{2-s-4\frac{1+\kappa}{p}}_{q,p}$, and let $(u,\sigma)$ be the $(p,\kappa,s,q)$-solution to \eqref{SCH} granted by Theorem~\ref{lwp:rough:initial:data}. Then,  Lemma~\ref{thm:regularization:rough:initial:data:LpLq} yields 
        \begin{equation}\label{proof:transference:almost:very:weak:0}
            u\in H_{\loc}^{\theta,r}(0,\sigma; \bpsn{2-4\theta,\zeta}) \text{ a.s.}, \quad r\in [2,\infty),\, \zeta \in [2,\infty)\, \theta\in[0,1/2).
        \end{equation}
        In particular, $u(t_0)\1_{\sigma>t_0}\in L^2(\sO)$ a.s. Therefore, Theorem~\ref{lwp:rough:initial:data}  ensures the existence of a unique maximal $(2,0,0,2)$-solution $(v,\tau)$ with $\tau>t_0$ a.s.\ to the shifted problem 
        \begin{equation}\label{SCH:shifted}
        \begin{cases}
         \dd v + \,_\nu\Delta_{-2,2}^2 v \dt &= \,_\nu\Delta_{-2,2}(v^3-v)\dt + \psi(v,\nabla v) \,\dd W(t), \quad t>t_0\\
         v(t_0)&= \1_{\sigma>t_0}u(t_0).
\end{cases}
    \end{equation}
    Moreover, it holds by~\cite[Theorem 4.10]{agrestiNonlinearParabolicStochastic2022a} that 
    \[\mathbb{P}\big( \tau<T , \sup_{t\in[t_0,\tau) }\|v\|_{L^2(\sO)}+ \|v\|_{L^2(t_0,\tau;\bpsn{2}) }<\infty \big)=0.\]
In addition, Lemma~\ref{thm:regularization:rough:initial:data:L2L2} gives
\begin{equation}\label{proof:transference:almost:very:weak:2}
    v\in H_{\loc}^{\theta,r}(t_0,\tau;\nH^{2-4\theta,\zeta}(\sO)) \,\text{ a.s.\ for all } \theta\in [0,1/2), \, r,\zeta\in (2,\infty). 
\end{equation}

    To conclude, we will prove 
\begin{equation}\label{proof:transference:almost:very:weak:1}
    \tau=\sigma\,  \text{ a.s.\ on }\{\sigma>t_0\} \quad \text{and}\quad u=v \, \text{ a.e.\ on }[t_0,\sigma)\times\{\sigma>t_0\}. 
\end{equation}
Then the claim follows, as ~\cite[Theorem 4.10]{agrestiNonlinearParabolicStochastic2022a} yields the correct blow-up criteria.
In order to prove~\eqref{proof:transference:almost:very:weak:1}, note that due to Lemma~\ref{thm:regularization:rough:initial:data:LpLq} it holds that  $(u|_{[t_0,\sigma)\times\{\sigma>t_0\}}, \1_{\sigma>t_0}\sigma+\1_{\Omega\setminus\{\sigma>t_0\}}t_0)$ is a \emph{local} $(2,0,0,2)$-solution  to \eqref{SCH:shifted} with initial datum $\1_{\sigma>t_0}u(t_0)$. 
As $(v,\tau)$ is the \emph{maximal} $(2,0,0,2)$-solution to \eqref{SCH:shifted} with initial datum $\1_{\{\sigma>t_0\}}u(t_0)$, it holds $\sigma\le \tau$ a.s.\ on $\{\sigma>t_0\}$ and $u=v$ a.e.\ on $[t_0,\sigma)$. 

Thus, in order to prove that indeed $\tau =\sigma$ on $\{\sigma>t_0\}$, i.e.\ that~\eqref{proof:transference:almost:very:weak:1} holds, we will prove that $$\P(\{t_0<\sigma\}\cap\{ \sigma<\tau\})=\P(t_0<\sigma<\tau)=0.$$ 
To this end, we will use the blow-up criterion~\cite[Theorem 4.10]{agrestiNonlinearParabolicStochastic2022a} for $(u,\sigma)$, for which we need specific regularity estimates.
The key observation is that, on the event $\{t_0<\sigma<\tau\}$, the stopping time $\sigma$ lies strictly inside the lifetime $(t_0,\tau)$ of $v$. Consequently, the regularity of $v$ in~\eqref{proof:transference:almost:very:weak:2} yields the required estimates in a neighbourhood of $\sigma$. Since $u=v$ on $[t_0,\sigma)\times\{\sigma>t_0\}$, these estimates transfer to $u$ near $\sigma$. Combined with the previously established regularity of $u$ in~\eqref{proof:transference:almost:very:weak:0} near the initial time, they allow us to verify the following form of a blow-up criterion, which remains valid in the critical case, see~\cite[Theorem 4.10]{agrestiNonlinearParabolicStochastic2022a}:
\[\P(\sigma<T, \sup_{t\in [0,\sigma)} \|u(t)\|_{X_{1-\frac{1+\kappa}{p},p}}+ \|u\|_{L^p(0,\sigma; X_{1-\frac{\kappa}{p}})}<\infty)=0.\]
Set $\beta\ce  2-s-4\frac{1+\kappa}{p}$  and $\gamma\ce  2-s-4\frac{\kappa}{p}$. Then $X_{1-\frac{1+\kappa}{p},p}=\,_\nu B^\beta_{q,p}(\sO)$  and $X_{1-\frac{\kappa}{p}}=\bpsn{\gamma,q}$. 

Note that~\eqref{proof:transference:almost:very:weak:2} in particular ensures $v\in L^p_{\loc}((t_0, \sigma]; \nH^{\gamma,q}(\sO))$ a.s.\ on $\{t_0<\sigma<\tau\}$, and, since $u=v$ a.s.\ on $[t_0,\sigma)\times\{\sigma>t_0\}$,~\eqref{proof:transference:almost:very:weak:0} gives $u\in L^p(0,\sigma;\bpsn{\gamma,q})$ a.s. on $\{t_0<\sigma<\tau\}.$
Next, we check that $\sup_{t\in [0,\sigma)} \nrm{u(t)}_{B_{q,p}^{\beta}(\sO)}<\infty$ a.s.\ on $\{t_0<\sigma<\tau\}$. 
Indeed,~\cite[Proposition 2.1]{agrestiNonlinearSPDEsMaximal2025} ensures that $u\in C\big((0,\sigma];\,_\nu B^{2-\frac{4}{r}}_{\zeta,r}(\sO)\big)$ for all $\zeta,r\in (2,\infty)$, and hence one can always find $\zeta,r \in (2,\infty)$ such that the latter embeds into $C\big((0,\sigma]; \,_\nu B^{\beta}_{q,p}(\sO)\big)$. Now using the fact that $u_0\in \,_\nu B^{\beta}_{q,p}(\sO)$ a.s., it follows that $\sup_{t\in [0,\sigma)}\nrm{u}_{\,_\nu B^{\beta}_{q,p}(\sO)} <\infty$ a.s.\ on $\{t_0<\sigma<\tau\}$.
Hence, from~\cite[Theorem 4.10 (3)]{agrestiNonlinearParabolicStochastic2022a} it follows that 
\begin{equation}
    \begin{aligned}
        \P(t_0<\sigma<\tau)&=\P(t_0<\sigma<\tau, \sup_{t\in[0,\sigma)}\nrm{u(t)}_{_\nu B_{q,p}^\beta(\sO)}+\nrm{u}_{L^p(0,\sigma;\nH^{\gamma,q}(\sO))}<\infty )\\
        &\le \P(\sigma<T, \sup_{t\in[0,\sigma)}\nrm{u(t)}_{_\nu B_{q,p}^\beta(\sO)}+\nrm{u}_{L^p(0,\sigma;\nH^{\gamma,q}(\sO))}<\infty )\\
        &=\P(\sigma<T, \sup_{t\in[0,\sigma)}\nrm{u(t)}_{X_{1-\frac{1+\kappa}{p},p}}+\nrm{u}_{L^p(0,\sigma;X_{1-\frac{\kappa}{p}})}<\infty)=0.
    \end{aligned}
\end{equation}
This proves~\eqref{proof:transference:almost:very:weak:1} , and thus the proof is finished. 
The blow-up criteria in cases~\ref{it:gwp:rough:initial:data:3} and~\ref{it:gwp:rough:intial:data:4} follow similarly. 
    \end{proof}
    By combining Theorem~\ref{gwp:rough:initial:data}
with Theorem~\ref{thm:GWP:weak:setting},  Theorem~\ref{thm:global_well_posedness_dim_3}, and Theorem~\ref{thm:global_well_posedness_dim_4-complete} we arrive at the following result stating global well-posedness for several settings. This in particular entails that we need to set $\psi^2=0$. 
    \begin{corollary}\label{cor:GWP:rough}
    The following assertions hold:
    \begin{enumerate}
        \item Let $d\in\{1,2\}$, suppose that Assumption~\ref{assumption_weak_setting} and Assumption~\ref{Assumption:noise}\ref{assumption:hg2} holds, assume that $(p,\kappa,s,q)$ are as in Theorem~\ref{lwp:rough:initial:data}\ref{it:lwp:rough:initial:data:2}. Then for all $u_0\in L^0_{\F_0}(\Omega;\,_\nu B_{q,p}^{2-s-4\frac{1+\kappa}{p}}(\sO))$ there exists a global solution to~\eqref{SCH}, which is unique in the class $L^p_\loc([0,\infty),w_\kappa;\bpsn{2-s,q}).$
        \item Let $d\in\{3,4\}$, suppose that Assumption~\ref{assumption_weakened_setting} holds,  assume that $(p,\kappa,s,q)$ are as in Theorem~\ref{lwp:rough:initial:data}\ref{it:lwp:rough:initial:data:4}. Then for all $u_0\in L^0_{\F_0}(\Omega;\,_\nu B_{q,p}^{3-s-4\frac{1+\kappa}{p}}(\sO))$ there exists a global solution to~\eqref{SCH}, which is unique in the class $L^p_\loc([0,\infty),w_\kappa;\bpsn{3-s,q}).$
    \end{enumerate}
    \end{corollary}
    \begin{proof}
    Let $(u,\sigma)$ be the $(p,\kappa,s,q)$-solution provided by Theorem~\ref{lwp:rough:initial:data}. In order to prove $\sigma=\infty$, by  Theorem~\ref{gwp:rough:initial:data} it suffices to verify the corresponding blow-up criterion. In addition, in the proof it is described that the restriction of $u$ to $[t_0,\sigma)$ coincides for every $t_0>0$ with the $(2,0,0,2)$-solution of~\eqref{SCH} with initial datum $\1_{\{\sigma>t_0\}} u(t_0)$. Hence, this implies that it suffices to prove that the $(2,0,0,2)$-solution coincides with the variational solution obtained in Theorem~\ref{thm:GWP:weak:setting}, Theorem~\ref{thm:global_well_posedness_dim_3}, or Theorem~\ref{thm:global_well_posedness_dim_4-complete}. 
        This in turn is ensured if  $_\nu\Delta_{-2,2}^2$ coincides with the variational operator $A$ from Section~\ref{weak_setting}, and  $_\nu\Delta_{-1,2}^2$ coincides with the operator $A$ from Section~\ref{weakened_setting}.
        We first identify $_\nu\Delta^2_{-2,2}$. To this end, we note that by construction
        $ A,\,_\nu\Delta_{-2,2}^2\in \sL(\bpsn{2},\bpsn{-2}).$
        Moreover, for all $u\in \bpsn{4}$ and all $v\in \bpsn{2}$ it holds that 
        $$\langle Au,v\rangle=(\Delta u, \Delta v)_{L^2(\sO)}= (\Delta^2 u,v)_{L^2(\sO)}=(\,_\nu\Delta_{-2,2}^2 u,v)_{L^2(\sO)}.$$
        This proves $A u = \,_\nu\Delta^2_{-2,2}u$ for all $u\in \bpsn{4}$. Since both operators are bounded and $\bpsn{4}$ is dense in $\bpsn{2}$, it follows that the identity holds on $\bpsn{2}$.

        Next, let $A$ denote the operator from Section~\ref{weakened_setting}. There it is shown that the dual of $\bpsn{3}$ with respect to $\bpsn{1}$ can be identified with $\bpsn{-1}$. With this identification, we obtain for all $u\in\bpsn{4}$, $v\in \bpsn{2}$  that            $\langle A u,v\rangle =(\Delta^2 u,v)_{L^2}$. It again follows that $A$ and $_\nu\Delta_{-1,2}^2$ agree on $\bpsn{4}$.
In addition, it holds by construction that        $$A, \,_\nu\Delta_{-1,2}^2 \in \sL(\bpsn{3},\bpsn{-1}).$$
Since $\bpsn{4}$ is dense in $\bpsn{3}$ and both $A$ and $_\nu\Delta_{-1,2}^2$ are continuous, it follows that they agree on $\bpsn{3}$.
    \end{proof}
\begin{remark}\label{rem:GWP:rough:additional:regularity}
        Beyond global existence and uniqueness, the solutions constructed above inherit, for every $t_0>0$, the full regularity and the estimates of the corresponding variational solution, since $u|_{[t_0,\infty)}$ coincides a.e.\ with the global variational solution of \eqref{SCH} restarted at $t_0$ with initial datum $u(t_0)$. Precisely:
        \begin{itemize}
            \item In case (a), $u\in C([t_0,\infty);L^2(\sO))\cap L^2_{\loc}([t_0,\infty);\bpsn{2})$ a.s.\ for every $t_0>0$, and, whenever $u(t_0)\in L^2(\Omega;L^2(\sO))$, for every $T>t_0$
            \[\E\big[\nrm{u}^2_{C([t_0,T];L^2(\sO))}\big]+\E\big[\nrm{u}^2_{L^2(t_0,T;H^2(\sO))}\big]\le C_T\big(1+\E\nrm{u(t_0)}^2_{L^2(\sO)}\big),\]
            by Theorem~\ref{thm:GWP:weak:setting}.
            \item In case (b), $u\in C([t_0,\infty);\bpsn{1})\cap L^2_{\loc}([t_0,\infty); \bpsn{3})$ a.s.\ for every $t_0>0$, and, whenever $u(t_0)\in L^4(\Omega;L^4(\sO))\cap L^2(\Omega;H^1(\sO))$, for every $T>t_0$
            \[\E\Big[\sup_{t\in[t_0,T]}\nrm{u(t)}^2_{H^1(\sO)}\Big]\le C_T\big(1+\nrm{u(t_0)}^4_{L^4(\Omega;L^4(\sO))}+\nrm{u(t_0)}^2_{L^2(\Omega;H^1(\sO))}\big),\]
            by Theorem~\ref{thm:global_well_posedness_dim_3} and Theorem~\ref{thm:global_well_posedness_dim_4-complete}.
        \end{itemize}
        In particular, even though $u_0$ may be rough, the solution instantaneously enters, and remains in, the significantly better variational regularity class for \emph{every} positive time due to Lemma~\ref{thm:regularization:rough:initial:data:LpLq}.
    \end{remark}

\printbibliography

\end{document}